\documentclass[11pt]{article}
\usepackage[tbtags]{amsmath}
\usepackage{bbm}
\usepackage{amssymb}
\usepackage{amsthm,mathrsfs}
\usepackage[misc]{ifsym}
\usepackage{graphicx}
\usepackage{tikz}
\usepackage{color}
\usepackage{float}
\usetikzlibrary{shapes,arrows}
\usepackage{cases}
\numberwithin{equation}{section}

\newtheorem{definition}{Definition}[section]
\newtheorem{theorem}{Theorem}[section]
\newtheorem{lemma}{Lemma}[section]
\newtheorem{remark}{Remark}[section]

\normalsize
\usepackage{hyperref,amsmath,amssymb,amscd}
\title{\bf The Maximum Principle for Discounted Optimal Control of Partially Observed Forward-Backward Stochastic Systems with Jumps on Infinite Horizon
\thanks{This work is supported by National Key R\&D Program of China (Grant No. 2018YFB1305400), National Natural Science Foundations of China (Grant Nos. 11971266, 11831010, 11571205), and Shandong Provincial Natural Science Foundations (Grant Nos. ZR2020ZD24, ZR2019ZD42).}}
\author{\normalsize Yueyang Zheng\thanks{\it School of Mathematics, Shandong University, Jinan 250100, P.R.China, E-mail: zhengyueyang0106@163.com} , Jingtao Shi\thanks{\it Corresponding author. School of Mathematics, Shandong University, Jinan 250100, P.R.China, E-mail: shijingtao@sdu.edu.cn}}

\begin{document}

\maketitle

\noindent{\bf Abstract:}
This paper is concerned with a discounted optimal control problem of partially observed forward-backward stochastic systems with jumps on infinite horizon. The control domain is convex and a kind of infinite horizon observation equation is introduced. The uniquely solvability of infinite horizon forward (backward) stochastic differential equation with jumps is obtained and more extended analysis, especially for the backward case, is made. Some new estimates are first given and proved for the critical variational inequality. Then an ergodic maximum principle is obtained by introducing some infinite horizon adjoint equations whose uniquely solvabilities are guaranteed necessarily. Finally, some comparison are made with two kinds of representative infinite horizon stochastic systems and their related optimal controls.

\vspace{1mm}


\vspace{1mm}


\section{Introduction}

Maximum principle is an important tool to solve the stochastic optimal control problem, by which the necessary condition for optimality is usually obtained. It is acknowledged that Peng \cite{Peng90} firstly introduced the general stochastic maximum principle for optimal control problems, where the control domain need not be convex and the diffusion coefficient can contain a control variable. Since then, plenty of mature results about global and local maximum principle are obtained in succession. For non-convex control domain, Situ \cite{Situ91} first obtained the global maximum principle for the controlled {\it stochastic differential equation with Poisson jumps} (SDEP for short) where the control variable is not included in the jump coefficient. Then Tang and Li \cite{TL94} and Li and Tang \cite{LT95} extended the general result in \cite{Peng90} into the case with random jumps and partially observed case, respectively. Wang and Wu \cite{WW09} studied a maximum principle for partially observed stochastic recursive optimal control problem  and forward diffusion coefficients do not contain control variables. Recently, global stochastic maximum principle for recursive utilities and fully coupled {\it forward-backward stochastic differential equations} (FBSDEs for short) were obtained in Hu \cite{Hu17} and Hu et al. \cite{HJX18}, respectively. Hao and Meng \cite{HM20} proved a maximum principle of optimal control problem for a class of general mean-field forward-backward stochastic systems with jumps where the diffusion coefficients depend on control, but the coefficients of jump terms are independent of control. Song et al. \cite{STW20} obtained a maximum principle for progressive optimal control of SDEPs by introducing a new method of variation  which is a correction of \cite{TL94}. For convex control domain, Peng \cite{Peng93} investigated the local maximum principle for forward-backward stochastic optimal control systems. Then Wu \cite{Wu10} studied an optimal control problem for partially observed forward-backward stochastic control system with the forward diffusion term containing control variable. Wang et al. \cite{WWX13} studied a partially observed optimal control problem of FBSDEs with correlated noises between the system and the observation, and obtained three versions of maximum principle. Then Zhang et al. \cite{ZXL18} considered the partially observed optimal control problem of {\it forward-backward stochastic differential equations with Poisson jumps} (FBSDEPs for short) with Markovian regime switching and developed a stochastic maximum principle.

The results mentioned above are based on the finite horizon. However, relatively few papers study the maximum principle for stochastic optimal control problems on infinite horizon. In fact, in many dynamic optimization problems in economics, finance and insurance, one need to consider problems on infinite horizon. Moreover, many mathematical difficulties are encountered when dealing with problems on infinite horizon. Let us mention a few. In an earlier time, Halkin \cite{Halkin74} gave some necessary conditions for the deterministic optimal control problem on infinite horizon. Maslowski and Veverka \cite{MV14} established the sufficient Pontryagin's maximum principle for an infinite horizon discounted stochastic control problem where the control domain is convex and bounded. Haadem et al. \cite{HOP13} proved the sufficient and necessary maximum principles for an infinite horizon control problem of SDEPs with partial information, where they firstly required a limit inequality on the terminal condition of infinite horizon {\it backward stochastic differential equations with Poisson jumps} (BSDEPs for short). Socgnia and Menoukeu-Pamen \cite{SP15} gave an infinite horizon stochastic maximum principle for a forward-backward stochastic systems with non-smooth coefficients since the value function is given by a discounted cost functional, assuming that the state coefficients are Lipschitz (with the diffusion coefficient being degenerate). Orrieri and Veverka \cite{OV17} developed an infinite horizon global stochastic maximum principle for stochastic control problem with a discounted cost functional under a polynomial growth and joint monotonicity assumption on the coefficients where the control domain is not necessarily convex. Moreover, Orrieri et al. \cite{OTV19} presented a stochastic maximum principle for ergodic control problem and gave the necessary and sufficient conditions for optimality for controlled dissipative systems. More related infinite horizon optimal control problems, see Agram et al. \cite{AHOP13}, Agram and \O ksendal \cite{AO14}, Muthukumar and Deepa \cite{MD17}, Yang and Wu \cite{YW20}, Ma and Liu \cite{ML17}, Wei and Yu \cite{WY21}, Mei et al. \cite{MWY21} and the references therein.

In the infinite horizon optimal control problem, it is necessary to guarantee the existence and uniqueness of solutions to the infinite horizon forward-backward stochastic systems and forward-backward stochastic systems with Poisson jumps which is one of the main results in our paper. In fact, there are also many results on this aspect. Peng \cite{Peng91} proved the existence and uniqueness of solutions to BSDEs with a stopping time under the Lipschitz and monotonicity conditions. Chen \cite{Chen98} also discussed the similar problem on a random time interval where $g$ satisfies Lipschitz condition with positive Lipschitzian function. Peng and Shi \cite{PS00} firstly investigated a class of infinite horizon FBSDEs and established its existence and uniqueness result with the method of continuation. Then Wu \cite{Wu03} studied the BSDEP in stopping time (unbounded) duration, under a suitable Lipschitz condition, the existence and uniqueness result was got by fixed point theorem, then proved the uniquely solvability of fully coupled FBSDEPs with stopping time under Lipschitz and monotone assumptions. Yin and Situ \cite{YS03} used a purely probabilistic approach to study FBSDEPs with stopping time and obtained the existence and uniqueness results of solutions under some weak monotonicity conditions and Lipschitz condition. Yin and Mao \cite{YM08} investigated a class of BSDEPs with random terminal time and proved its existence and uniqueness result under the assumption of non-Lipschitzian coefficient. Yin \cite{Yin08}, \cite{Yin11} studied the uniquely solvability of infinite horizon FBSDEs and FBSDEs with random terminal time by constructing a contraction mapping, respectively. Yu \cite{Yu17} extended the infinite horizon forward-backward stochastic systems in \cite{PS00} into the case with random jumps, and similarly, an existence and uniqueness theorem was established under some monotonicity conditions. Very recently, Shi and Zhao \cite{SZ20} considered a similar infinite horizon framework in \cite{PS00} but discussed the existence and uniqueness of solutions in an arbitrarily large space for infinite horizon FBSDEs. Wei and Yu \cite{WY21} put forward a {\it linear-quadratic} (LQ for short) stochastic control problem with random coefficients, and due to the introduction of a parameter $K$ in the discounted cost functional, they introduced a new infinite horizon version of domination-monotonicity condition, and obtained an existence and uniqueness result and its related estimates of the solutions to a kind of infinite horizon coupled FBSDEs. Moreover, they found some different results between the infinite horizon and finite horizon case when applying the theoretical result to the uniquely solvability of LQ FBSDEs (also called Hamiltonian systems).

In this paper, we study a partially observed discounted optimal control problem of forward-backward stochastic systems with jumps on infinite horizon. In summary, the contributions of this paper include the following.

(1) We first obtain the uniqueness and existence of solutions to infinite horizon SDEP and prove the uniquely solvability of BSDEP and their estimates. It is interesting and inspiring to make detailed analysis for the existence of solution to infinite horizon BSDEP. We first prove it in a small space ($\beta<0$) where both two kinds of truncated auxiliary BSDEPs can approximate \eqref{ifbsdep01} very well (see Remark \ref{rem25} and Lemma \ref{lemma23}), then we extend it in a big space ($\beta>0$) by a new Definition \ref{definition22} in which the condition $\beta<0$ discussed for existence in the small space can be relaxed. See Remark \ref{remark26} and the following discussion. Moreover, the uniquely solvability of infinite horizon adjoint FBSDEPs in Section 3.3 are also obtained by combining the above theorems with the existing literature on finite horizon.

(2) We give a kind of infinite horizon observation equation, which contributes to the Girsanov's theorem by a sufficient ``weak Novikov condition" on infinite horizon. Then a transformed completely observed optimal control problem of infinite horizon FBSDEP is studied and the ergodic maximum principle is obtained. Meanwhile, some new and necessary high-order estimates are the most important contribution in our paper, in which Lemma \ref{lemma31} is vital and beneficial to other estimates, Lemma \ref{lemma32}--\ref{lemma36}. The difference is that the flexible discount factors for state and control variable play key roles in these estimates and may take different values, $\beta_0,\cdots,\beta_5$, which is an adjustable idea to obtain our variational inequality \eqref{variation inequality}.

(3) We make some comparison with the existing results, especially for two kinds of representative existing infinite horizon stochastic systems and their related optimal controls, and some necessary analyses are inspiring.

The rest of this paper is organized as follows. In Section 2, we formulate the discounted optimal control problem of partially observed forward-backward stochastic systems with jumps on infinite horizon, and obtain the uniquely solvability of infinite horizon SDEPs and BSDEPs. Then some new and important estimates are given to derive the vital variational inequality and infinite horizon adjoint equations are introduced to obtain the ergodic maximum principle in Section 3. In Section 4, we make some comparisons with two kinds of representative infinite horizon systems and their related optimal control problems.

\section{Problem formulation and preliminaries}

Consider a complete filtered probability space $(\Omega,\mathcal{F},(\mathcal{F}_t)_{t\geq0},\bar{\mathbb{P}})$ and two one-dimensional independent standard Brownian Motions $W(\cdot)$ and $\tilde{W}(\cdot)$ defined in $\mathbb{R}^2$ with $W_0=\tilde{W}_0=0$. Let $(\mathcal{E},\mathcal{B}(\mathcal{E}))$ be a Polish space with the $\sigma$-finite measure $\nu$ on $\mathcal{E}$. Suppose that $N(de,dt)$ is a Poisson random measure on $(\mathcal{E}\times\mathbb{R}^+,\mathcal{B}(\mathcal{E})\times\mathcal{B}(\mathbb{R}^+))$ under $\bar{\mathbb{P}}$ and for any $E\in\mathcal{B}(\mathcal{E})$, since $\nu(E)<\infty$, then the compensated Poisson random measure is given by $\tilde{N}(de,dt)=N(de,dt)-\nu(de)dt$. Moreover, $W(\cdot),\tilde{W}(\cdot),N(\cdot,\cdot)$ are mutually independent under $\bar{\mathbb{P}}$, and let $(\mathcal{F}_t^W)_{t\geq0},(\mathcal{F}_t^{\tilde{W}})_{t\geq0},(\mathcal{F}_t^N)_{t\geq0}$ be the $\bar{\mathbb{P}}$-completed natural filtrations generated by $W(\cdot),\tilde{W}(\cdot),N(\cdot,\cdot)$, respectively. Set $\{\mathcal{F}_t\}_{t\geq0}:=\{\mathcal{F}_t^W\vee\mathcal{F}_t^{\tilde{W}}\vee\mathcal{F}_t^N\}_{t\geq0}\vee\mathcal{N}$ and $\mathcal{F}_\infty=\bigvee_{t\geq0}\mathcal{F}_t\subset\mathcal{F}$, where $\mathcal{N}$ denotes the totality of $\bar{\mathbb{P}}$-null sets. $\bar{\mathbb{E}}$ denotes the expectation under the probability $\bar{\mathbb{P}}$, and $|\cdot|$ denotes the Euclidean norm in $\mathbb{R}$.

We consider the following controlled infinite horizon FBSDEP:
\begin{equation}\label{iffbsdep}
\left\{
\begin{aligned}
 dx_t^u&=b(x_t^u,u_t)dt+\sigma(x_t^u,u_t)dW_t+\tilde{\sigma}(x_t^u,u_t)d\tilde{W}_t+\int_{\mathcal{E}}l(x_{t-}^u,u_{t-},e)\tilde{N}(de,dt),\\
-dy^u_t&=g\Big(x_t^u,y_t^u,z_t^u,\tilde{z}_t^u,\int_\mathcal{E}\gamma_{(t,e)}^u\nu(de),u_t\Big)dt-z_t^udW_t-\tilde{z}_t^ud\xi_t-\int_{\mathcal{E}}\gamma_{(t,e)}^u\tilde{N}(de,dt),\ t\in[0,\infty),\\
  x_0^u&=x_0,
\end{aligned}
\right.
\end{equation}
and the observation process satisfied the following SDE:
\begin{equation}\label{observation}
\left\{
\begin{aligned}
d\xi_t^u&=h(x_t^u,u_t)dt+d\tilde{W}_t,\ t\in[0,\infty),\\
\xi_0&=0,
\end{aligned}
\right.
\end{equation}
where $u(\cdot)$ is a control process taking values in the convex subset $U$ of $\mathbb{R}$, the mappings $b:\mathbb{R}\times U\rightarrow\mathbb{R}$, $\sigma:\mathbb{R}\times U\rightarrow\mathbb{R}$, $\tilde{\sigma}:\mathbb{R}\times U\rightarrow\mathbb{R}$ and $l:\mathbb{R}\times U\times\mathcal{E}\rightarrow\mathbb{R}$ satisfy the detailed conditions shown in the following {\bf (A1)--(A4)}, $g:\mathbb{R}\times\mathbb{R}\times\mathbb{R}\times\mathbb{R}\times\mathbb{R}\times U\rightarrow\mathbb{R}$ satisfies {\bf (A5)--(A8)}. And the mapping $h:\mathbb{R}\times U\rightarrow\mathbb{R}$ satisfies the following {\bf (H0)}.

{\bf (H0)} $h(\cdot,\cdot)$ is $\mathscr{B}(\mathbb{R})\otimes\mathscr{B}(\mathbb{R})/\mathscr{B}(\mathbb{R})$ measurable and continuous on $\mathbb{R}\times U$ and bounded. $h$ is continuously differential in $(x,u)$ and its derivatives $h_x,h_u$ are both bounded.

\begin{remark}
In the partially observed case, there exists correlated noises $\tilde{W}(\cdot)$ between state equation and observation equation. Thus we usually need assume the diffusion coefficient $\tilde{\sigma}$ and the drift coefficient $h$ in the related observation equation are bounded by some constants in the finite horizon case. However, in our partially observed problem with infinite horizon, the simple constant bounded condition is too strong to support the well-posedness of the term $\int_0^\infty\tilde{\sigma}(\cdot,\cdot)^2dt$. Therefore, we need give a new condition in {\bf (A1)} that $\tilde{\sigma}(\cdot,\cdot)\leq\mu_0(t)$, where $\mu_0(\cdot)$ should satisfy the convergence when integrating in the infinite time domain, which is different from the finite horizon case in \cite{WWX13}.
\end{remark}

Our aim is to minimize the discounted cost functional as follows:
\begin{equation}\label{cf}
J(u)=\bar{\mathbb{E}}\bigg[\int_0^\infty e^{-\beta t}f\Big(x_t^u,y_t^u,z_t^u,\tilde{z}_t^u,\int_{\mathcal{E}}\gamma_{(t,e)}^u\nu(de),u_t\Big)dt+\phi(y_0^u)\bigg],\ \text{with }\beta\ \text{big enough},
\end{equation}
where the mappings $f:\mathbb{R}\times\mathbb{R}\times\mathbb{R}\times\mathbb{R}\times\mathbb{R}\times U\rightarrow\mathbb{R}$ and $\phi:\mathbb{R}\rightarrow\mathbb{R}$ satisfy the following {\bf (H1)}.

{\bf (H1)} $f(\cdot,\cdot,\cdot,\cdot,\cdot,\cdot)$ is $\mathscr{B}(\mathbb{R})\otimes\mathscr{B}(\mathbb{R})\otimes\mathscr{B}(\mathbb{R})\otimes\mathscr{B}(\mathbb{R})\otimes\mathscr{B}(\mathbb{R})\otimes\mathscr{B}(\mathbb{R})/\mathscr{B}(\mathbb{R})$ measurable and is continuously differential in $(x,y,z,\tilde{z},\gamma,u)$, and $\phi(\cdot)$ is $\mathscr{B}(\mathbb{R})/\mathscr{B}(\mathbb{R})$ measurable and  continuously differential in $y$, and there exists some constant $C$ such that
\begin{equation}\label{f}
\begin{aligned}
&\big(1+|x|^2+|y|^2+|z|^2+|\tilde{z}|^2+||\gamma||^2+|u|^2\big)^{-1}|f(x,y,z,\tilde{z},\gamma,u)|\\
&\quad +\big(1+|x|+|y|+|z|+|\tilde{z}|+||\gamma||+|u|\big)^{-1}|f_{(x,y,z,\tilde{z},\gamma,u)}(x,y,z,\tilde{z},\gamma,u)|\leq C,\\
&\big(1+|y|^2\big)^{-1}|\phi(y)|+\big(1+|y|\big)^{-1}|\phi_y(y)|\leq C.
\end{aligned}
\end{equation}

\begin{remark}\label{rem22}
For the infinite horizon discounted cost functional \eqref{cf}, different from the case without discounting in \cite{HOP13}, we introduce the parameter $\beta$ (which is usually called discount factor in terminology) in the improper integrable form. In fact, one one hand, we can guarantee that the cost functional is well-posed by {\bf (H1)} and some estimates. The discussion will be mentioned in the following Remark \ref{rem26} in details. On the other hand, we can give an alternative condition ${\bf (H1^\prime)}$ of {\bf (H1)} as follows.

${\bf (H1^\prime)}$ $f(\cdot,\cdot,\cdot,\cdot,\cdot,\cdot), \phi(\cdot)$ satisfy the same conditions with {\bf (H1)} except for the following. There exists some positive function $C(t):=e^{-Kt}$ decided by some adjustable parameter $K$ such that \eqref{f} holds by replacing $C$ with $C(t)$.

Different from \eqref{f} itself, we call the new form in ${\bf (H1^\prime)}$ the exponential weighting quadratic growth condition. Moreover, the alternative condition can show nice property. To be specific, when $K$ are positive enough, we may take negative value $\beta$ such that the cost functional is still well-posedness, which can relax the condition satisfied by $\beta$. In addition, ${\bf (H1^\prime)}$ is also adapted to the derivation of variational inequality in Section 3.2 and can improve and optimize our result by adjusting another parameter $K$. Even if it is enough that all the result in this paper are obtained under {\bf (H1)}, which can imply the same result under ${\bf (H1^\prime)}$, the condition ${\bf (H1^\prime)}$ is still worthy of consideration in the weaker framework which will be investigated in our future work, in the sense, {\bf (H1)} is stronger than ${\bf (H1^\prime)}$.
\end{remark}

\begin{remark}
Notice that the Girsanov's theorem can be applied only if the sufficient Novikov's condition holds. However, different from the finite horizon case, the assumption $|h(\cdot,\cdot)|\leq C$ can not guarantee that
\begin{equation}
\bar{\mathbb{E}}\bigg[\exp\bigg(\frac{1}{2}\int_0^\infty|h(x_t^u,u_t)|^2dt\bigg)\bigg]<\infty
\end{equation}
holds. Thus, in the infinite horizon case, we need put forward a ``weak Novikov's condition"
\begin{equation}
\bar{\mathbb{E}}\bigg[\exp\bigg(\frac{1}{2}\int_0^\infty e^{-\beta t}|h(x_t^u,u_t)|^2dt\bigg)\bigg]<\infty.
\end{equation}
Under this weak Novikov's condition, we should consider the corresponding observation equation as follows
\begin{equation}\label{observation01}
\left\{
\begin{aligned}
d\xi_t^u&=e^{-\frac{\beta}{2}t}h(x_t^u,u_t)dt+d\tilde{W}_t,\ t\in[0,\infty),\\
\xi_0&=0.
\end{aligned}
\right.
\end{equation}
Indeed, when $h$ is bounded, we can guarantee that the integral $\int_0^\infty e^{-\frac{\beta}{2}t}h(x^u_t,u_t)dt$ is finite. Therefore, it is reasonable to give the weak Novikov's condition in the infinite horizon case.
\end{remark}

By the above discussion, and the {\it infinite horizon Novikov's condition}
\begin{equation}
\bar{\mathbb{E}}\bigg[\exp\bigg(\frac{1}{2}\int_0^\infty|e^{-\frac{\beta}{2}t}h(x_t^u,u_t)|^2dt\bigg)\bigg]<\infty,
\end{equation}
then by Girsanov theorem, we introduce a new probability measure $\mathbb{P}$ such that
\begin{equation}
\frac{d\bar{\mathbb{P}}}{d\mathbb{P}}\bigg|_{\mathcal{F}_t}=\mathcal{Z}_t,
\end{equation}
where
\begin{equation}\label{RN}
\mathcal{Z}_t=\exp\bigg\{\int_0^te^{-\frac{\beta}{2}s}h(x_s^u,u_s)d\xi_s-\frac{1}{2}\int_0^te^{-\beta s}|h(x_s^u,u_s)|^2ds\bigg\}
\end{equation}
is an exponential martingale, which satisfies
\begin{equation}\label{mathcalZ}
\left\{
\begin{aligned}
d\mathcal{Z}_t^u&=e^{-\frac{\beta}{2}t}\mathcal{Z}_th(x_t^u,u_t)d\xi_t,\ t\in[0,\infty),\\
\mathcal{Z}_0^u&=1.
\end{aligned}
\right.
\end{equation}
Substituting \eqref{observation01} into the state equation \eqref{iffbsdep}, we have
\begin{equation}\label{iffbsdep01}
\left\{
\begin{aligned}
 dx_t^u&=\tilde{b}(x_t^u,u_t)dt+\sigma(x_t^u,u_t)dW_t+\tilde{\sigma}(x_t^u,u_t)d\xi_t+\int_\mathcal{E}l(x_{t-}^u,u_{t-},e)\tilde{N}(de,dt),\\
-dy^u_t&=g\Big(x_t^u,y_t^u,z_t^u,\tilde{z}_t^u,\int_{\mathcal{E}}\gamma_{(t,e)}^u\nu(de),u_t\Big)dt-z_t^udW_t-\tilde{z}_t^ud\xi_t-\int_{\mathcal{E}}\gamma_{(t,e)}^u\tilde{N}(de,dt),\ t\in[0,\infty),\\
  x_0^u&=x_0,
\end{aligned}
\right.
\end{equation}
where $\tilde{b}(x^u_t,u_t):=b(x_t^u,u_t)-e^{-\frac{\beta}{2}t}\tilde{\sigma}(x_t^u,u_t)h(x_t^u,u_t)$.

\begin{definition}
Let $\beta\in\mathbb{R}$, and $\mathbb{E}$ be the expectation corresponding to $\mathbb{P}$, then we introduce the following spaces:
\begin{equation*}
\begin{aligned}
L^2(\mathcal{E},\mathcal{B}(\mathcal{E}),\nu;\mathbb{R})&:=\bigg\{\gamma:\mathcal{E}\rightarrow\mathbb{R}\mbox{ such that }\bigg(\int_{\mathcal{E}}|\gamma(e)|^2\nu(de)\bigg)^{\frac{1}{2}}<\infty\bigg\},\\
L^{2,-\beta}_{\mathcal{F}}(0,\infty;\mathbb{R})&:=\bigg\{v:[0,\infty)\times\Omega\rightarrow\mathbb{R}\bigg|\text{v is}\ (\mathcal{F}_t)_{t\geq0}\text{-progressively measurable }\\
                                               &\qquad \text{ with}\ \mathbb{E}\int_0^\infty e^{-\beta t}|v_t|^2dt<\infty\bigg\},\\
L^{2,-\beta}_{\mathcal{F},\nu}(0,\infty;\mathbb{R})&:=\bigg\{v:[0,\infty)\times\Omega\times\mathcal{E}\rightarrow\mathbb{R}\bigg|\text{v is}\ (\mathcal{F}_t)_{t\geq0}\text{-predictable}\\
                                                   &\qquad \text{ with}\ \mathbb{E}\int_0^\infty\int_{\mathcal{E}} e^{-\beta t}|v_{(t,e)}|^2\nu(de)dt<\infty\bigg\},
\end{aligned}
\end{equation*}
and a smaller space, for $k\geq1$,
\begin{equation*}
\begin{aligned}
L^{2k,-\beta}_{\mathcal{F},N}(0,\infty;\mathbb{R})&:=\bigg\{v:[0,\infty)\times\Omega\times\mathcal{E}\rightarrow\mathbb{R}\bigg|\text{v is}\ (\mathcal{F}_t)_{t\geq0}\text{-predictable}\\
                                                  &\qquad \text{ with}\ \mathbb{E}\bigg(\int_0^\infty\int_{\mathcal{E}}e^{-\beta t}|v_{(t,e)}|^2N(de,dt)\bigg)^{k}<\infty\bigg\},
\end{aligned}
\end{equation*}
where $L^{2k,-\beta}_{\mathcal{F},N}(0,\infty;\mathbb{R})=L^{2,-\beta}_{\mathcal{F},\nu}(0,\infty;\mathbb{R})$ when taking $k=1$, and
\begin{equation*}
\begin{aligned}
\mathcal{S}^{2k,-\beta}(0,\infty)&:=\bigg\{v:[0,\infty)\times\Omega\rightarrow\mathbb{R}\bigg|\text{v is}\ (\mathcal{F}_t)_{t\geq0}\text{-progressively measurable}\\
                                 &\qquad \text{ with}\ \mathbb{E}\bigg[\sup_{t\in\mathbb{R}_+} e^{-\beta k t}|v_t|^{2k}\bigg]<\infty\bigg\},
\end{aligned}
\end{equation*}
where $\mathcal{S}^{2k,-\beta}(0,\infty)$ reduce to $\mathcal{S}^{2k}[0,T]$ when taking finite horizon and $\beta=0$.
\end{definition}

We need the following assumptions.

{\bf (A1)} $x_0\in L_{\mathcal{F}_0}^{2k}(\Omega;\mathbb{R})$ with $\mathbb{E}|x_0|^{2k}<\infty$. $b(\cdot,\cdot),\sigma(\cdot,\cdot),\tilde{\sigma}(\cdot,\cdot),l(\cdot,\cdot,\cdot)$ are $\mathscr{B}(\mathbb{R})\otimes\mathscr{B}(\mathbb{R})\otimes\mathscr{B}(\mathcal{E})/\mathscr{B}(\mathbb{R})$ measurable and continuous on $\mathbb{R}\times U$, and $b,\sigma,\tilde{\sigma},l$ are continuously differential in $(x,u)$. The derivatives $b_x,b_u,\sigma_x,\sigma_u,\tilde{\sigma}_x,\tilde{\sigma}_u$ and $l_x,l_u$ are bounded. Moreover, there exists a positive function $\mu_0(t)$ (with respect to $t$) such that $\tilde{\sigma}(\cdot,\cdot)$ is bounded by $\mu_0(t)$, where $\mu_0(\cdot)$ satisfies $\int_0^\infty\mu_0^2(t)dt<\infty$, i.e., $\mu_0(t)$ can be $e^{-\frac{\mu_0}{2}t}$.

{\bf (A2)} There exists $\mu_1\in \mathbb{R}$ such that, for $\forall u\in U$,
\begin{equation*}
\left\langle x_1-x_2,b(x_1,u)-b(x_2,u)\right\rangle\leq \mu_1|x_1-x_2|^2.
\end{equation*}

{\bf (A3)} There exists positive constants $L_b,L_\sigma,L_{\tilde{\sigma}},L_l$ such that
\begin{equation*}
\begin{aligned}
&|\Psi(x_1,u)-\Psi(x_2,u)|\leq L_{\Psi}|x_1-x_2|,\ \Psi=b,\sigma,\tilde{\sigma}, \ ||l(x_1,u,e)-l(x_2,u,e)||\leq L_l|x_1-x_2|,
\end{aligned}
\end{equation*}
for any $(x_1,x_2,u)\in\mathbb{R}\times\mathbb{R}\times U$.

{\bf (A4)} For any $u\in U$, there exists a $\beta$ big enough such that $b(0,u)\in L^{2,-\beta}_{\mathcal{F}}(0,\infty;\mathbb{R}),\sigma(0,u)\in L^{2,-\beta}_{\mathcal{F}}(0,\infty;\mathbb{R}),\tilde{\sigma}(0,u)\in L^{2,-\beta}_{\mathcal{F}}(0,\infty;\mathbb{R}),l(0,u,e)\in L^{2,-\beta}_{\mathcal{F},v}(0,\infty;\mathbb{R})$.

{\bf (A5)} $g(\cdot,\cdot,\cdot,\cdot,\cdot,\cdot)$ is $\mathscr{B}(\mathbb{R})\otimes\mathscr{B}(\mathbb{R})\otimes\mathscr{B}(\mathbb{R})\otimes\mathscr{B}(\mathbb{R})\otimes\mathscr{B}(\mathbb{R})\otimes\mathscr{B}(\mathbb{R})/\mathscr{B}(\mathbb{R})$ measurable and continuous on $\mathbb{R}^5\times U$, continuously differentiable in $(x,y,z,\tilde{z},\gamma,u)$ with bounded derivatives.

{\bf (A6)} There exists $\mu_2\in\mathbb{R}$ such that, for $\forall (x,z,\tilde{z},\gamma,u)\in\mathbb{R}^4\times U$,
\begin{equation*}
\left\langle y_1-y_2,g(x,y_1,z,\tilde{z},\gamma_{(t,e)},u)-g(x,y_2,z,\tilde{z},\gamma_{(t,e)},u)\right\rangle\leq\mu_2|y_1-y_2|^2,
\end{equation*}
where for notational simplicity, we define $g(x_t,y_t,z_t,\tilde{z}_t,\gamma_{(t,e)},u)=g(x_t,y_t,z_t,\tilde{z}_t,\int_{\mathcal{E}}\gamma_{(t,e)}\nu(de),u)$, and $(x,y,z,\tilde{z},\gamma_{(t,e)})\in\mathbb{R}^4\times L^2(\mathcal{E},\mathcal{B}(\mathcal{E}),\nu;\mathbb{R})$.

{\bf (A7)} The Lipschitz coefficient $K_3(t)$ of $g$ with respect to $\gamma$ is sufficiently small, i.e., there also exists other positive functions $K_0(t),K_1(t),K_2(t),K_4(t)$ such that, for $\forall x\in\mathbb{R}$,
\begin{equation*}
\begin{aligned}
&|g(x,y_1,z_1,\tilde{z}_1,\gamma_{(1,t,e)},u_1)-g(x,y_2,z_2,\tilde{z}_2,\gamma_{(2,t,e)},u_2)|\\
&\leq K_0(t)|y_1-y_2|+K_1(t)|z_1-z_2|+K_2(t)|\tilde{z}_1-\tilde{z}_2|+K_3(t)||\gamma_{(1,t,e)}-\gamma_{(2,t,e)}||+K_4(t)|u_1-u_2|,
\end{aligned}
\end{equation*}
where $K_0(t),K_1(t),K_2(t),K_3(t)$ and $K_4(t)$ are deterministic functions satisfying
\begin{equation*}
\begin{aligned}
&\int_0^\infty K_i^2(t)dt<\infty,\ i=0,1,2,3,4.
\end{aligned}
\end{equation*}
{\bf (A8)} For $\beta\in\mathbb{R}$,
\begin{equation*}
\beta+2\mu_2+2K_1^2(t)+2K_2^2(t)+2K_3^2(t)+2K_4^2(t)<0,
\end{equation*}
where $\beta$ may be negative or positive. There exists a $\beta_0\in\mathbb{R}$ (which will be decided in the following) such that, for any $(x,u)\in\mathbb{R}\times U$, $\mathbb{E}\big[\int_0^\infty e^{-\beta_0 t}|g(x,0,0,0,0,u)|^2dt\big]<\infty$.

\begin{remark}\label{remark24}
The assumption {\bf (A7)} is similar to {\bf (H2)-(H3)} in \cite{Chen98} which studied the existence and uniqueness for BSDE with stopping time. Moreover, in our infinite horizon case, we take some detail form of $K_0(t),K_1(t),K_2(t),K_3(t)$ as follows:
\begin{equation*}
K_0(t)=e^{-\frac{K_0}{2}t},\ K_1(t)=e^{-\frac{K_1}{2}t},\ K_2(t)=e^{-\frac{K_2}{2}t},\ K_3(t)=e^{-\frac{K_3}{2}t}.
\end{equation*}
\end{remark}

Next, for given $u(\cdot)\in U$, we focus on the solvability of FBSDEP \eqref{iffbsdep01}. For this target, we first consider the infinite horizon SDEP:
\begin{equation}\label{ifsdep01}
\left\{
\begin{aligned}
dx_t&=\tilde{b}(x_t)dt+\sigma(x_t)dW_t+\tilde{\sigma}(x_t)d\xi_t+\int_\mathcal{E}l(x_{t-},e)\tilde{N}(de,dt),\ t\in[0,\infty),\\
 x_0&=x_0,
\end{aligned}
\right.
\end{equation}
with the same coefficients $\tilde{b},\sigma,\tilde{\sigma},l$ as in \eqref{iffbsdep01} but omitting $u$ here temporarily. Firstly, a basic deduction can be given.

\begin{lemma}\label{lem21}
Let {\bf (A1)--(A4)} hold. If infinite horizon SDEP \eqref{ifsdep01} admits a unique solution $x(\cdot)\in L^{2,-\beta}_{\mathcal{F}}(0,\infty;\mathbb{R})$, then we have the following limit asymptotic property
\begin{equation}
\lim_{T\rightarrow\infty}\mathbb{E}\big[|x_T|^2e^{-\beta T}\big]=0.
\end{equation}
\end{lemma}

Then the existence and uniqueness result and related estimate of \eqref{ifsdep01} are given as follows.
\begin{theorem}\label{the21}
Under {\bf (A1)--(A4)} and assuming that $\beta>2\mu_1+L^2_\sigma+L^2_{\tilde{\sigma}}+L^2_l$, then \eqref{ifsdep01} admits a unique solution $x(\cdot)\in L^{2,-\beta}_{\mathcal{F}}(0,\infty;\mathbb{R})$. Moreover, for any $\epsilon>0$, we have the following estimate:
\begin{equation}\label{est1}
\begin{aligned}
&\big(\beta-2\mu_1-3\epsilon-L^2_{\sigma}-L^2_{\tilde{\sigma}}-L^2_l-\epsilon L^2_l\big)\mathbb{E}\int_0^\infty e^{-\beta t}|x_t|^2dt\\
&\leq x^2_0+\mathbb{E}\int_0^\infty e^{-\beta t}\bigg[\frac{1}{\epsilon}|\tilde{b}(0)|^2+\bigg(\frac{L^2_{\sigma}}{\epsilon}+1\bigg)|\sigma(0)|^2
 +\bigg(\frac{L^2_{\tilde{\sigma}}}{\epsilon}+1\bigg)|\tilde{\sigma}(0)|^2+\bigg(1+\frac{1}{\epsilon}\bigg)||l(0,e)||^2\bigg]dt.
\end{aligned}
\end{equation}
Meanwhile, let $x_1(\cdot),x_2(\cdot)\in L^{2,-\beta}_{\mathcal{F}}(0,\infty;\mathbb{R})$ be two solutions to \eqref{ifsdep01} with the coefficients $(\tilde{b}_1,\sigma_1,\tilde{\sigma}_1,l_1)$ and $(\tilde{b}_2,\sigma_2,\tilde{\sigma}_2,l_2)$ and initial values $x_1(0),x_2(0)$, respectively. Then, for any $\epsilon>0$, we also give the a priori estimate as follows:
\begin{equation}\label{est2}
\begin{aligned}
&\big(\beta-2\mu_1-3\epsilon-L^2_{\sigma}-L^2_{\tilde{\sigma}}-L^2_l-\epsilon L^2_l\big)\mathbb{E}\int_0^\infty e^{-\beta t}|x_1-x_2|^2dt\\
&\quad \leq |x_1(0)-x_2(0)|^2+\mathbb{E}\int_0^\infty e^{-\beta t}\bigg[\frac{1}{\epsilon}|\tilde{b}_1(x_2)-\tilde{b}_2(x_2)|^2+\bigg(\frac{L^2_{\sigma}}{\epsilon}+1\bigg)|\sigma_1(x_2)-\sigma_2(x_2)|^2\\
&\qquad\qquad +\bigg(\frac{L^2_{\tilde{\sigma}}}{\epsilon}+1\bigg)|\tilde{\sigma}_1(x_2)-\tilde{\sigma}_2(x_2)|^2+\bigg(1+\frac{1}{\epsilon}\bigg)||l_1(x_2,e)-l_2(x_2,e)||^2\bigg]dt.
\end{aligned}
\end{equation}
\end{theorem}

\begin{proof}
We can easily obtain the existence and uniqueness result of infinite horizon SDEP \eqref{ifsdep01} in $L^{2,-\beta}_{\mathcal{F}}(0,\infty;\mathbb{R})$ by the classical method for uniquely solvability of SDEP on finite horizon. Moreover, similar study on infinite horizon SDE (or SDEP) can also be seen in \cite{WY21} and \cite{Yu17}. Therefore, we mainly give the proof of the estimates \eqref{est1} and \eqref{est2}. For any $T>0$, applying It\^{o}'s formula to $e^{-\beta t}x^2_t$ on $[0,T]$, and by {\bf (A2)--(A3)} we get (we omit time variable $t$ in some places for simplicity)
\begin{equation*}
\begin{aligned}
&\mathbb{E}e^{-\beta T}x^2_T+\mathbb{E}\int_0^T\big(\beta-2\mu_1-3\epsilon-L^2_\sigma-L^2_{\tilde{\sigma}}-L^2_l-\epsilon L^2_l\big)e^{-\beta t}x^2_tdt\\
&\leq x^2_0+\mathbb{E}\int_0^T e^{-\beta t}\bigg[\frac{1}{\epsilon}|\tilde{b}(0)|^2+\bigg(\frac{L^2_{\sigma}}{\epsilon}+1\bigg)|\sigma(0)|^2
 +\bigg(\frac{L^2_{\tilde{\sigma}}}{\epsilon}+1\bigg)|\tilde{\sigma}(0)|^2+\bigg(1+\frac{1}{\epsilon}\bigg)||l(0,e)||^2\bigg]dt.
\end{aligned}
\end{equation*}
Let $T\rightarrow\infty$, we get the estimate \eqref{est1}. Then, due to $\beta>2\mu_1+L^2_\sigma+L^2_{\tilde{\sigma}}+L^2_l$ and by {\bf (A4)}, it is easy to check that $x(\cdot)\in L^{2,-\beta}_\mathcal{F}(0,\infty;\mathbb{R})$. Next, we set $\hat{x}(\cdot)=x_1(\cdot)-x_2(\cdot)$, and
\begin{equation*}
\begin{aligned}
&\hat{\tilde{b}}(x)=\tilde{b}_1(x_2+x)-\tilde{b}_2(x_2),\ \ \hat{\sigma}(x)=\sigma_1(x_2+x)-\sigma_2(x_2),\\
&\hat{\tilde{\sigma}}(x)=\tilde{\sigma}_1(x_2+x)-\tilde{\sigma}_2(x_2),\ \ \hat{l}(x,e)=l_1(x_2+x,e)-l_2(x_2,e).
\end{aligned}
\end{equation*}
Then, we get the equation of $\hat{x}(\cdot)$:
\begin{equation*}
\left\{
\begin{aligned}
d\hat{x}_t&=\hat{\tilde{b}}(\hat{x}_t)dt+\hat{\sigma}(\hat{x}_t)dW_t+\hat{\tilde{\sigma}}(\hat{x}_t)d\xi_t+\int_{\mathcal{E}}\hat{l}(\hat{x}_{t-},e)\tilde{N}(de,dt),\ t\in[0,\infty),\\
 \hat{x}_0&=x_1(0)-x_2(0).
\end{aligned}
\right.
\end{equation*}
And it is easy to verify that $\hat{\tilde{b}},\hat{\sigma},\hat{\tilde{\sigma}},\hat{l}$ satisfy {\bf (A1)--(A4)}, so we can directly get the a priori estimate \eqref{est2} by estimate \eqref{est1} above. The proof is complete.
\end{proof}

Moreover, we need the unique solvability of the following infinite horizon BSDEP:
\begin{equation}\label{ifbsdep01}
-dy_t=g(x_t,y_t,z_t,\tilde{z}_t,\gamma_{(t,e)})dt-z_tdW_t-\tilde{z}_td\xi_t-\int_{\mathcal{E}}\gamma_{(t,e)}\tilde{N}(de,dt),\ t\in[0,\infty),
\end{equation}
with the same coefficient $g$ as in \eqref{iffbsdep01}, which is coupled with the infinite horizon SDEP \eqref{ifsdep01}.

\begin{lemma}\label{lem22}
Assume that {\bf (A5)--(A8)} hold. Let $(y_1(\cdot),z_1(\cdot),\tilde{z}_1(\cdot),\gamma_1(\cdot,\cdot))$ and $(y_2(\cdot),z_2(\cdot),\tilde{z}_2(\cdot),\\\gamma_2(\cdot,\cdot))$ be solutions to \eqref{ifbsdep01} with generator $g=g^1$ and $g^2$, respectively, then, for any $\delta>0$, we have
\begin{equation}
\begin{aligned}
&\mathbb{E}\int_0^\infty e^{-\beta t}\bigg[\frac{1}{2}(|z_1-z_2|^2+|\tilde{z}_1-\tilde{z}_2|^2+||\gamma_1-\gamma_2||^2)\\
&\qquad\qquad\quad +\big(-\beta-2\mu_2-2K_1^2(t)-2K_2^2(t)-2K_3^2(t)-\delta\big)|y_1-y_2|^2\bigg]dt\\
&\quad \leq\frac{1}{\delta}\mathbb{E}\int_0^\infty e^{-\beta t}\big|g^1(x,y_2,z_2,\tilde{z}_2,\gamma_2)-g^2(x,y_2,z_2,\tilde{z}_2,\gamma_2)\big|^2dt.
\end{aligned}
\end{equation}
\end{lemma}

\begin{proof}
We set $\hat{y}=y_1-y_2,\hat{z}=z_1-z_2,\hat{\tilde{z}}=\tilde{z}_1-\tilde{z}_2,\hat{\gamma}=\gamma_1-\gamma_2$. First, we have
\begin{equation}\label{haty}
\begin{aligned}
-d\hat{y}_t&=\big[g^1(x,y_1,z_1,\tilde{z}_1,\gamma_{(1,t,e)})-g^2(x,y_2,z_2,\tilde{z}_2,\gamma_{(2,t,e)})\big]dt\\
           &\quad-\hat{z}_tdW_t-\hat{\tilde{z}}_td\xi_t-\int_{\mathcal{E}}\hat{\gamma}_{t,e}\tilde{N}(de,dt),\ t\in[0,\infty].
\end{aligned}
\end{equation}
For any $T>0$, applying It\^{o}'s formula to $e^{-\beta t}|\hat{y}_t|^2$ on $[0,T]$, then we deduce that
\begin{equation*}
\begin{aligned}
&|\hat{y}_0|^2+\mathbb{E}\int_0^T e^{-\beta t}\big(-\beta|\hat{y}_t|^2+|\hat{z}|^2+|\hat{\tilde{z}}|^2+||\hat{\gamma}||^2\big)dt=\mathbb{E}\big[e^{-\beta T}|\hat{y}_T|^2\big]\\
&\quad +\mathbb{E}\int_0^Te^{-\beta t}\Big\{2|\hat{y}_t|\big[g^1(x,y_1,z_1,\tilde{z}_1,\gamma_{(1,t,e)})-g^2(x,y_2,z_2,\tilde{z}_2,\gamma_{(2,t,e)})\big]\Big\}dt\\
&\leq \mathbb{E}\big[e^{-\beta T}|\hat{y}_T|^2\big]+\mathbb{E}\int_0^Te^{-\beta t}\Big\{\big(2\mu_2+2K_1^2(t)+2K_2^2(t)+2K_3^2(t)+\delta\big)|\hat{y}_t|^2+\frac{1}{2}|\hat{z}|^2\\
&\quad +\frac{1}{2}|\hat{\tilde{z}}|^2+\frac{1}{2}||\hat{\gamma}||^2+\frac{1}{\delta}\big|g^1(x,y_2,z_2,\tilde{z}_2,\gamma_{(2,t,e)})-g^2(x,y_2,z_2,\tilde{z}_2,\gamma_{(2,t,e)})\big|^2\Big\}dt.
\end{aligned}
\end{equation*}
Furthermore, we have
\begin{equation*}
\begin{aligned}
&\mathbb{E}\int_0^T e^{-\beta t}\Big[\big(-\beta-2\mu_2-2K_1^2(t)-2K_2^2(t)-2K_3^2(t)-\delta\big)|\hat{y}_t|^2+\frac{1}{2}|\hat{z}|^2+\frac{1}{2}|\hat{\tilde{z}}|^2+\frac{1}{2}||\hat{\gamma}||^2\Big]dt\\
&\leq \mathbb{E}\big[e^{-\beta T}|\hat{y}_T|^2\big]+\frac{1}{\delta}\mathbb{E}\int_0^Te^{-\beta t}\big|g^1(x,y_2,z_2,\tilde{z}_2,\gamma_{(2,t,e)})-g^2(x,y_2,z_2,\tilde{z}_2,\gamma_{(2,t,e)})\big|^2dt.
\end{aligned}
\end{equation*}
Because $\hat{y}(\cdot)$ is the unique solution to \eqref{haty}, then taking $T\rightarrow\infty$ and considering Lemma \ref{lem21} similarly for $\hat{y}(\cdot)$, we get
\begin{equation*}
\begin{aligned}
&\mathbb{E}\int_0^\infty e^{-\beta t}\Big[\big(-\beta-2\mu_2-2K_1^2(t)-2K_2^2(t)-2K_3^2(t)-\delta\big)|\hat{y}_t|^2+\frac{1}{2}|\hat{z}|^2+\frac{1}{2}|\hat{\tilde{z}}|^2+\frac{1}{2}||\hat{\gamma}||^2\Big]dt\\
&\leq \frac{1}{\delta}\mathbb{E}\int_0^\infty e^{-\beta t}\big|g^1(x,y_2,z_2,\tilde{z}_2,\gamma_{(2,t,e)})-g^2(x,y_2,z_2,\tilde{z}_2,\gamma_{(2,t,e)})\big|^2dt,
\end{aligned}
\end{equation*}
where we take $-\beta-2\mu_2-2K_1^2(t)-2K_2^2(t)-2K_3^2(t)-\delta>0$. The proof is complete.
\end{proof}

Next, we begin to prove the existence and uniqueness of infinite horizon BSDEP \eqref{ifbsdep01}.
\begin{theorem}\label{the22}
Under Assumption {\bf (A5)--(A8)}, \eqref{ifbsdep01} admits a unique solution $(y(\cdot),z(\cdot),\tilde{z}(\cdot),\\\gamma(\cdot,\cdot))$ in $L^{2,-\beta}_{\mathcal{F}}(0,\infty;\mathbb{R}^3)\times L^{2,-\beta}_{\mathcal{F},\nu}(0,\infty;\mathbb{R})$.
\end{theorem}

\begin{proof}
{\bf Existence.} For $n=1,2,3,\cdots$, define $\varphi_n(t)=\mathbb{I}_{[0,n]}(t)g(x_t,0,0,0,0)$. First, \eqref{ifbsdep01} can be rewritten as follows:
\begin{equation}
\begin{aligned}
-dy_t&=\big[g(x_t,y_t,z_t,\tilde{z}_t,\gamma_{(t,e)})-g(x_t,0,0,0,0)+g(x_t,0,0,0,0)\big]dt\\
     &\quad -z_tdW_t-\tilde{z}_td\xi_t-\int_{\mathcal{E}}\gamma_{(t,e)}\tilde{N}(de,dt)\\
     &:=\big[G(y_t,z_t,\tilde{z}_t,\gamma_{(t,e)})+\varphi(t)\big]dt-z_tdW_t-\tilde{z}_td\xi_t-\int_{\mathcal{E}}\gamma_{(t,e)}\tilde{N}(de,dt).\\
\end{aligned}
\end{equation}
Obviously, $G(\cdot,\cdot,\cdot,\cdot)$ is Lipschitzian with respect to $(y,z,\tilde{z},\gamma)$ and monotonicity condition in $y$ with the same Lipschitz coefficients and $\mu_2$ as $g$, and the sequence $\{\varphi_n\}_{n=1}^\infty$ converges to $\varphi(t)$ in $L^{2,-\beta}_{\mathcal{F}}(0,\infty;\mathbb{R})$. For each $n$, let $(y_n(\cdot),z_n(\cdot),\tilde{z}_n(\cdot),\gamma_n(\cdot,\cdot))\in L^{2,-\beta}_{\mathcal{F}}(0,\infty;\mathbb{R}^3)\times L^{2,-\beta}_{\mathcal{F},\nu}(0,\infty;\mathbb{R})$ be the unique adapted solution to the following finite horizon BSDEP:
\begin{equation}\label{fhbsdep01}
\left\{
\begin{aligned}
-dy_{(n,t)}&=\bigg[G\big(y_{(n,t)},z_{(n,t)},\tilde{z}_{(n,t)},\gamma_{(n,t,e)}\big)+\varphi(t)\bigg]dt-z_{(n,t)}dW_t-\tilde{z}_{(n,t)}d\xi_t\\
           &\quad-\int_{\mathcal{E}}\gamma_{(n,t,e)}\tilde{N}(de,dt),\ t\in[0,n],\\
  y_{(n,n)}&=0,
\end{aligned}
\right.
\end{equation}
which admits unique solution by Proposition A.2 in Quenez and Sulem \cite{QS13}. And due to $G(0,0,0,0)=0$, it is easy to check that $(y_{(n,t)},z_{(n,t)},\tilde{z}_{(n,t)},\gamma_{(n,t,e)})=(0,0,0,0)$ when $t\in(n,\infty)$.
Therefore, $(y_n(\cdot),z_n(\cdot),\tilde{z}_n(\cdot),\gamma_n(\cdot,\cdot))$ satisfies
\begin{equation}
\begin{aligned}
-dy_{(n,t)}&=\bigg[G\big(y_{(n,t)},z_{(n,t)},\tilde{z}_{(n,t)},\gamma_{(n,t,e)}\big)+\varphi_n(t)\bigg]dt-z_{(n,t)}dW_t-\tilde{z}_{(n,t)}d\xi_t\\
&\quad-\int_{\mathcal{E}}\gamma_{(n,t,e)}\tilde{N}(de,dt),\ t\in[0,\infty).
\end{aligned}
\end{equation}
It follows from the estimation in Lemma \ref{lem22} that $(y_n(\cdot),z_n(\cdot),\tilde{z}_n(\cdot),\gamma_n(\cdot,\cdot))$ is a Cauchy sequence in $L^{2,-\beta}_{\mathcal{F}}(0,\infty;\mathbb{R}^3)\times L^{2,-\beta}_{\mathcal{F},\nu}(0,\infty;\mathbb{R})$. Then it is easy to check that the limit $(y(\cdot),z(\cdot),\tilde{z}(\cdot),\gamma(\cdot,\cdot))\in L^{2,-\beta}_{\mathcal{F}}(0,\infty;\mathbb{R}^3)\times L^{2,-\beta}_{\mathcal{F},\nu}(0,\infty;\mathbb{R})$ is a solution to infinite horizon BSDEP \eqref{ifbsdep01}.
Indeed, when considering $\beta<0$, we deduce that
{\small\begin{equation}\label{E1}
\mathbb{E}\bigg(\int_t^\infty z_ndW_s-\int_t^\infty zdW_s\bigg)^2=\mathbb{E}\int_t^\infty |z_n-z|^2ds\leq \mathbb{E}\int_t^\infty e^{-\beta s}|z_n-z|^2ds\rightarrow0,\ \text{as}\ n\rightarrow\infty.
\end{equation}}
Similarly,
$\mathbb{E}\big[\int_t^\infty \tilde{z}_nd\xi_s-\int_t^\infty \tilde{z}d\xi_s\big]^2\rightarrow0,\ \text{as}\ n\rightarrow\infty$,
\begin{equation}\label{E3}
\begin{aligned}
&\mathbb{E}\bigg(\int_t^\infty\int_{\mathcal{E}} \gamma_{(n,s,e)}\tilde{N}(de,ds)-\int_t^\infty\int_{\mathcal{E}} \gamma_{(s,e)}\tilde{N}(de,ds)\bigg)^2
=\mathbb{E}\int_t^\infty\int_{\mathcal{E}}|\gamma_{(n,s,e)}-\gamma_{(s,e)}|^2\nu(de)ds\\
&\leq \mathbb{E}\int_t^\infty\int_{\mathcal{E}} e^{-\beta s}|\gamma_{(n,s,e)}-\gamma_{(s,e)}|^2\nu(de)ds\rightarrow0,\ \text{as}\ n\rightarrow\infty,
\end{aligned}
\end{equation}
and for any $\beta<0$, we have
\begin{equation*}
\begin{aligned}
&\mathbb{E}\bigg[\int_t^\infty G(y_n,z_n,\tilde{z}_n,\gamma_n)-G(y,z,\tilde{z},\gamma)ds\bigg]^2\\
&\leq \mathbb{E}\bigg[\int_t^\infty |G(y_n,z_n,\tilde{z}_n,\gamma_n)-G(y,z,\tilde{z},\gamma)|e^{-\frac{\beta}{2}s}e^{\frac{\beta}{2}s}ds\bigg]^2\\
&\leq \mathbb{E}\bigg[\int_t^\infty \Big|e^{-\frac{K_0}{2}s}|y_n-y|+e^{-\frac{K_1}{2}s}|z_n-z|+e^{-\frac{K_2}{2}s}|\tilde{z}_n-\tilde{z}|\\
&\qquad\qquad +e^{-\frac{K_3}{2}s}||\gamma_n-\gamma||\Big|^2e^{-\beta s}ds\int_t^\infty e^{\beta s}ds\bigg]\\
\end{aligned}
\end{equation*}
\begin{equation}\label{E4}
\begin{aligned}
&\leq C\mathbb{E}\int_0^\infty \big(|y_n-y|^2+|z_n-z|^2+|\tilde{z}_n-\tilde{z}|^2+||\gamma_n-\gamma||^2\big)e^{-\beta s}ds\rightarrow0,\ \text{as}\ n\rightarrow\infty,
\end{aligned}
\end{equation}
and
$\mathbb{E}\big[\int_t^\infty \varphi_n(s)-\varphi(s)ds\big]^2\rightarrow0,\ \text{as}\ n\rightarrow\infty$.

Finally, due to $\mathbb{E}\int_0^\infty|y_n-y|^2e^{-\beta s}ds\rightarrow0$, we have $\mathbb{E}\big[|y_n-y|^2\big]\leq \mathbb{E}\big[|y_n-y|^2e^{-\beta s}\big]\rightarrow0\ \text{as}\ n\rightarrow\infty$. Thus, $\int_t^\infty z_{(n,s)}dW_s,\int_t^\infty \tilde{z}_{(n,s)}d\xi_s,\int_t^\infty\int_{\mathcal{E}} \gamma_{(n,s,e)}\tilde{N}(de,ds),\int_t^\infty\varphi_n(s)ds,\int_t^\infty G(y_n,z_n,\tilde{z}_n,\gamma_n)ds$ all converge to $\int_t^\infty z_sdW_s,\int_t^\infty \tilde{z}_sd\xi_s,\int_t^\infty\int_{\mathcal{E}} \gamma_{(s,e)}\tilde{N}(de,ds),\int_t^\infty\varphi(s)ds,\int_t^\infty G(y,z,\tilde{z},\gamma)ds$ in $L^2(\Omega,\\\mathcal{F},\mathbb{P};\mathbb{R})$, respectively. From the above discussion, the existence of solution is proved in a smaller space $L^{2,-\beta}_{\mathcal{F}}(0,\infty;\mathbb{R}^3)\times L^{2,-\beta}_{\mathcal{F},\nu}(0,\infty;\mathbb{R})(\beta<0)$, which can guarantee the existence of solution in $L^{2,-\beta}_{\mathcal{F}}(0,\infty;\mathbb{R}^3)\times L^{2,-\beta}_{\mathcal{F},\nu}(0,\infty;\mathbb{R})(\beta>0)$.

{\bf Uniqueness.} It is immediate from Lemma \ref{lem22}. The proof is complete.
\end{proof}

\begin{remark}\label{rem25}
In the discussion of the existence of solutions, we mainly approximate the infinite horizon BSDE \eqref{ifbsdep01} by a truncated equation with a zero value terminal condition. In fact, in our setting, the terminal value at infinity of infinite horizon BSDE is not given, that is, a $\mathcal{F}_{\infty}$-measurable random variable $\zeta$ is not known. However, if we consider a finite truncated time interval $[0,n]$, and define the terminal value at $n$ by a form of conditional expectation, i.e, $\mathbb{E}[\zeta|\mathcal{F}_n]$, then we can show that the value at $n$ is known in the finite horizon. In other words, although we have no information about the value at infinity, we can know the estimate for $\xi$ at $n$ based on the informations at fixed finite interval $[0,n]$. Therefore, we may try another way, inspired by Darling and Pardoux \cite{DP97}, to find a solution to \eqref{ifbsdep01} by the following equation:
\begin{equation}\label{fhbsdep02}
\left\{
\begin{aligned}
-d\hat{y}^n_t&=g\big(x_t,\hat{y}^n_t,\hat{z}^n_t,\hat{\tilde{z}}^n_t,\hat{\gamma}^n_{(t,e)}\big)dt-\hat{z}^n_tdW_t-\hat{\tilde{z}}^n_td\xi_t-\int_{\mathcal{E}}\hat{\gamma}^n_{(t,e)}\tilde{N}(de,dt),\ t\in[0,n],\\
  \hat{y}_n^n&=\zeta_n:=\mathbb{E}[\zeta|\mathcal{F}_n],
\end{aligned}
\right.
\end{equation}
where $\zeta$ is an $\mathbb{R}$-valued $\mathcal{F}_{\infty}$-measurable random variable with $\mathbb{E}|\zeta|^2<\infty$. For interval $(n,\infty)$, we denote $\hat{y}^n_t,\hat{z}^n_t,\hat{\tilde{z}}^n_t$ and $\hat{\gamma}^n_t$ by
\begin{equation}\label{00}
\begin{aligned}
\hat{y}^n_t=\zeta_t=\mathbb{E}[\zeta|\mathcal{F}_t],\quad \hat{z}^n_t=\hat{\tilde{z}}^n_t=\hat{\gamma}^n_{(t,e)}=g(x,\hat{y}^n,\hat{z}^n,\hat{\tilde{z}}^n,\hat{\gamma}^n)=0.
\end{aligned}
\end{equation}
Applying It\^{o}'s formula, we have
\begin{equation*}
\begin{aligned}
&e^{-\frac{\beta}{2} m}\zeta_m-e^{-\frac{\beta}{2} n}\zeta_n-e^{-\frac{\beta}{2} t}(\hat{y}^m_t-\hat{y}^n_t)
 =\int_t^m-\frac{\beta}{2} e^{-\frac{\beta}{2} s}\hat{y}^m_sds-\int_t^n-\frac{\beta}{2} e^{-\frac{\beta}{2} s}\hat{y}^n_sds\\
&\quad-\int_t^me^{-\frac{\beta}{2} s}g(x,\hat{y}^m,\hat{z}^m,\hat{\tilde{z}}^m,\hat{\gamma}^m)ds+\int_t^ne^{-\frac{\beta}{2}s}g(x,\hat{y}^n,\hat{z}^n,\hat{\tilde{z}}^n,\hat{\gamma}^n)ds\\
&\quad+\int_t^me^{-\frac{\beta}{2}s}\hat{z}^m_sdW_s-\int_t^ne^{-\frac{\beta}{2}s}\hat{z}^n_sdW_s+\int_t^me^{-\frac{\beta}{2}s}\hat{\tilde{z}}^m_sd\xi_s-\int_t^ne^{-\frac{\beta}{2}s}\hat{\tilde{z}}^n_sd\xi_s\\
&\quad+\int_t^me^{-\frac{\beta}{2}s}\int_{\mathcal{E}}\hat{\gamma}^m_{(s,e)}\tilde{N}(de,ds)-\int_t^ne^{-\frac{\beta}{2} s}\int_\mathcal{E}\hat{\gamma}^n_{(s,e)}\tilde{N}(de,ds).
\end{aligned}
\end{equation*}
By \eqref{00}, it is equivalent to the following
\begin{equation*}
\begin{aligned}
&e^{-\frac{\beta}{2} m}\zeta_m-e^{-\frac{\beta}{2} n}\zeta_n-e^{-\frac{\beta}{2} t}\Delta\hat{y}_t=\int_t^{m\vee n}-\frac{\beta}{2} e^{-\frac{\beta}{2} s}\Delta\hat{y}_sds\\
&\quad-\int_t^{m\vee n}e^{-\frac{\beta}{2} s}\bigg[g(x,\hat{y}^m,\hat{z}^m,\hat{\tilde{z}}^m,\hat{\gamma}^m)
 -g(x,\hat{y}^n,\hat{z}^n,\hat{\tilde{z}}^n,\hat{\gamma}^n)\bigg]ds\\
&\quad+\int_t^{m\vee n}e^{-\frac{\beta}{2}s}\Delta\hat{z}_sdW_s+\int_t^{m\vee n}e^{-\frac{\beta}{2}s}\Delta\hat{\tilde{z}}_sd\xi_s+\int_t^{m\vee n}e^{-\frac{\beta}{2}s}\int_{\mathcal{E}}\Delta\hat{\gamma}\tilde{N}(de,ds),
\end{aligned}
\end{equation*}
where $\Delta\hat{y}_t=\hat{y}^m_t-\hat{y}^n_t, \Delta\hat{z}_t=\hat{z}^m_t-\hat{z}^n_t, \Delta\hat{\tilde{z}}_t=\hat{\tilde{z}}^m_t-\hat{\tilde{z}}^n_t, \Delta\hat{\gamma}_{(t,e)}=\hat{\gamma}^m_{(t,e)}-\hat{\gamma}^n_{(t,e)}$. Applying It\^{o}'s formula to $|e^{-\frac{\beta}{2}t}(\hat{y}^m_t-\hat{y}^n_t)|^2$ on $[t,m\vee n]$, and similarly,
\begin{equation*}
\begin{aligned}
&\mathbb{E}\big[e^{-\beta t}|\Delta\hat{y}_t|^2\big]+\mathbb{E}\int_t^{m\vee n}e^{-\beta s}\Big[(-\beta-2\mu_2-2K_1^2(t)-2K_2^2(t)-2K_3^2(t)) |\Delta\hat{y}_s|^2\\
&\quad+\frac{1}{2}(|\Delta\hat{z}|^2+|\Delta\hat{\tilde{z}}|^2+||\Delta\hat{\gamma}||^2)\Big]ds\leq \mathbb{E}|e^{-\frac{\beta}{2}m}\zeta_m-e^{-\frac{\beta}{2}n}\zeta_n|^2.
\end{aligned}
\end{equation*}
Taking $sup$ on both sides, we obtain
\begin{equation*}
\begin{aligned}
&\sup_{t\in\mathbb{R}_+}\mathbb{E}\big[e^{-\beta t}|\Delta\hat{y}_t|^2\big]+\mathbb{E}\int_0^\infty\mathbb{I}_{[0,m\vee n]}e^{-\beta s}\Big[(-\beta-2\mu_2-2K_1^2(t)-2K_2^2(t)-2K_3^2(t)) |\Delta\hat{y}_s|^2\\
&\quad+\frac{1}{2}(|\Delta\hat{z}|^2+|\Delta\hat{\tilde{z}}|^2+||\Delta\hat{\gamma}||^2)\Big]ds\leq \mathbb{E}|e^{-\frac{\beta}{2}m}\zeta_m-e^{-\frac{\beta}{2}n}\zeta_n|^2,
\end{aligned}
\end{equation*}
for $\beta<-2\mu_2-2K_1^2(t)-2K_2^2(t)-2K_3^2(t)$, where we have used the fact by \eqref{00} that
\begin{equation*}
\sup_{t\in\mathbb{R}_+}\mathbb{E}\big[e^{-\beta t}|\Delta\hat{y}_t|^2\big]\leq\sup_{0\leq t\leq m\vee n}\mathbb{E}\big[e^{-\beta t}|\Delta\hat{y}_t|^2\big]
+\sup_{m\vee n<t}\mathbb{E}\big[e^{-\beta t}|\Delta\hat{y}_t|^2\big]=\sup_{0\leq t\leq m\vee n}\mathbb{E}\big[e^{-\beta t}|\Delta\hat{y}_t|^2\big].
\end{equation*}
And by Jensen's inequality, we have
\begin{equation*}
\begin{aligned}
&\mathbb{E}|e^{-\frac{\beta}{2}m}\zeta_m-e^{-\frac{\beta}{2}n}\zeta_n|^2\leq \mathbb{E}e^{-\beta m}|\zeta_m|^2+\mathbb{E}e^{-\beta n}|\zeta_n|^2\\
&\leq e^{-\beta m}\mathbb{E}|\mathbb{E}[\zeta|\mathcal{F}_m]|^2+e^{-\beta n}\mathbb{E}|\mathbb{E}[\zeta|\mathcal{F}_n]|^2
\leq (e^{-\beta m}+e^{-\beta n})\mathbb{E}|\zeta|^2\rightarrow0,\ \text{as}\ m,n\rightarrow\infty.
\end{aligned}
\end{equation*}
It implies that there exists Cauchy sequence $(\hat{y}^n,\hat{z}^n,\hat{\tilde{z}}^n,\hat{\gamma}^n)$ and it is not hard to obtain that its limiting process $(y,z,\tilde{z},\gamma)$ solve \eqref{ifbsdep01}. The uniqueness is guaranteed by Lemma \ref{lem22}.
\end{remark}

Combining Theorem \ref{the22} and Remark \ref{rem25}, we find both kinds of truncated BSDEPs can approximate \eqref{ifbsdep01} well with zero terminal and given random variable $\zeta_n$ at $n$, respectively. This fact is also mentioned in \cite{MV14}, where \eqref{fhbsdep01} or \eqref{fhbsdep02} can be chosen to obtain the same solution process for \eqref{ifbsdep01}. An intuitive explanation is obvious in the following property with random jump (see also similar Lemma 2, \cite{MV14} without jump).

\begin{lemma}\label{lemma23}
Under {\bf (A5)--(A8)}, let $\zeta$ be an $\mathbb{R}$-valued $\mathcal{F}_{\infty}$-measurable random variable with $\mathbb{E}|\zeta|^2<\infty$. We consider, for each fixed $n\in\mathbb{N}$, the two BSDEs \eqref{fhbsdep01} and \eqref{fhbsdep02}, where $y_{(n,t)}=0$ on $(n,\infty)$ mentioned in the proof of Theorem \ref{the22}, besides, we set $\hat{y}^n_t=\zeta_t=\mathbb{E}[\zeta|\mathcal{F}_t],\ \hat{z}^n_t=\hat{\tilde{z}}^n_t=\hat{\gamma}^n_{(t,e)}=g(x,\hat{y}^n,\hat{z}^n,\hat{\tilde{z}}^n,\hat{\gamma}^n)=0$, for any $t>n$. Then we have
\begin{equation}\label{MV2014}
\begin{aligned}
\lim_{n\rightarrow\infty}\bigg[&\sup_{t\in\mathbb{R_+}}\mathbb{E}e^{-\beta t}|y_{(n,t)}-\hat{y}^n_t|^2+\mathbb{E}\int_0^\infty\mathbb{I}_{[0,n]}e^{-\beta t}\Big(|y_{(n,t)}-\hat{y}^n_t|^2+|z_{(n,t)}-\hat{z}^n_t|^2\\
&\quad+|\tilde{z}_{(n,t)}-\hat{\tilde{z}}^n_t|^2+||\gamma_{(n,t)}-\hat{\gamma}^n_t||^2\Big)dt\bigg]=0.
\end{aligned}
\end{equation}
\end{lemma}

\begin{remark}\label{remark26}
It is noted that the solution to infinite horizon BSDEP \eqref{ifbsdep01} is defined on $[0,\infty)$. Therefore, we should restrict $\beta<0$ in the final part \eqref{E1}--\eqref{E4} of existence's proof when we guarantee the limiting process solves \eqref{ifbsdep01} on $[0,\infty)$. Then we first get the existence result of solutions in a smaller space, by which it then naturally obtains its existence in a larger space. In fact, a recent result found in \cite{SZ20} shows that the value of $\beta$ is not necessary to be restricted to be negative. In other words, the existence and uniqueness of infinite horizon BSDE can hold with $\beta\in \mathbb{R}$, provided that the solutions $(y(\cdot),z(\cdot))\in L^{2,-\beta}_{\mathcal{F}}(0,\infty;\mathbb{R}^2)$ is defined as the Definition 2.1 in \cite{SZ20}. In our framework, the definition can be translated to the following one.

\begin{definition}\label{definition22}
For any $x(\cdot)\in L^{2,-\beta}_{\mathcal{F}}(0,\infty;\mathbb{R})$, a quadruple $(y(\cdot),z(\cdot),\tilde{z}(\cdot),\gamma(\cdot,\cdot))\in L^{2,-\beta}_{\mathcal{F}}(0,\infty;\\\mathbb{R}^3)\times L^{2,-\beta}_{\mathcal{F},\nu}(0,\infty;\mathbb{R})$ is called a solution to infinite horizon BSDEP \eqref{ifbsdep01}, if it is satisfied in the following sense: for any $T\in[0,\infty)$,
\begin{equation}\label{yT}
y_t=y_T+\int_t^Tg(x_s,y_s,z_s,\tilde{z}_s,\gamma_{(s,e)})ds-\int_t^Tz_sdW_s-\int_t^T\tilde{z}_sd\xi_s-\int_t^T\int_{\mathcal{E}}\gamma_{(s,e)}\tilde{N}(de,ds).
\end{equation}
\end{definition}
Therefore, the solution to \eqref{ifbsdep01} is only defined on any given finite horizon $[0,T]$. Then, in this case, $\beta<0$ can be relaxed. Indeed, it is not hard to check that the limit $(y(\cdot),z(\cdot),\tilde{z}(\cdot),\gamma(\cdot,\cdot))\in L^{2,-\beta}_{\mathcal{F}}(0,\infty;\mathbb{R}^3)\times L^{2,-\beta}_{\mathcal{F},\nu}(0,\infty;\mathbb{R})$ solves \eqref{ifbsdep01} on an arbitrary interval $[0,T]$ for $\beta\in\mathbb{R}$. So we can prove its existence and uniqueness result of solution to \eqref{ifbsdep01} in $L^{2,-\beta}_{\mathcal{F}}(0,\infty;\mathbb{R}^3)\times L^{2,-\beta}_{\mathcal{F},\nu}(0,\infty;\mathbb{R})$ without restricting $\beta<0$ first. Similar discussion are also  illustrated in \cite{WY21}.

In fact, by similar Laplace transform, \eqref{yT} is equivalent to the following equation
\begin{equation}\label{yinfty}
\begin{aligned}
e^{-\beta t}y_t&=\int_t^\infty\beta e^{-\beta s}y_sds+\int_t^\infty e^{-\beta s}g(x_s,y_s,z_s,\tilde{z}_s,\gamma_{(s,e)})ds-\int_t^\infty e^{-\beta s}z_sdW_s\\
&\quad-\int_t^\infty e^{-\beta s}\tilde{z}_sd\xi_s-\int_t^\infty \int_{\mathcal{E}}e^{-\beta s}\gamma_{(s,e)}\tilde{N}(de,ds),\ t\in[0,\infty).
\end{aligned}
\end{equation}
There is another explanation by \eqref{yinfty} to guarantee the existence and uniqueness result directly in larger space $L^{2,-\beta}_{\mathcal{F}}(0,\infty;\mathbb{R})$ (where the discounting weight $e^{-\beta t}(\beta>0)$ is more reasonable in economic) instead of some restrictions imposing on $\beta$. We also illustrate certain term in \eqref{yinfty} and other term can be discussed similarly.
\begin{equation*}
\begin{aligned}
&\mathbb{E}\bigg[\int_t^\infty e^{-\beta s}\big(g(x,y_n,z_n,\tilde{z}_n,\gamma_n)-g(x,y,z,\tilde{z},\gamma)\big)ds\bigg]^2\\
&\leq \mathbb{E}\bigg[\int_t^\infty e^{-\beta s}\big|g(x,y_n,z_n,\tilde{z}_n,\gamma_n)-g(x,y,z,\tilde{z},\gamma)\big|^2ds\int_t^\infty e^{-\beta s}ds\bigg]\\
\end{aligned}
\end{equation*}
\begin{equation}
\begin{aligned}
&\leq \mathbb{E}\int_t^\infty e^{-\beta s}\big(e^{-K_0s}|y_n-y|^2+e^{-K_1s}|z_n-z|^2+e^{-K_2s}|\tilde{z}_n-\tilde{z}|^2\\
&\qquad +e^{-K_3s}||\gamma_n-\gamma||^2\big)ds\rightarrow0,\ \text{as}\ n\rightarrow\infty,
\end{aligned}
\end{equation}
which means, in $L^2(\Omega,\mathcal{F},\mathbb{P};\mathbb{R})$,
\begin{equation}
\int_t^\infty e^{-\beta s}g(x_s,y_{(n,s)},z_{(n,s)},\tilde{z}_{(n,s)},\gamma_{(n,s,e)})ds\rightarrow \int_t^\infty e^{-\beta s}g(x_s,y_s,z_s,\tilde{z}_s,\gamma_{(s,e)})ds,
\end{equation}
where, $L^{2,-\beta_2}_{\mathcal{F}}(0,\infty;\mathbb{R})\subset L^{2,-\beta_1}_{\mathcal{F}}(0,\infty;\mathbb{R})$ if $\beta_1>\beta_2$ due to $\mathbb{E}\int_0^\infty e^{-\beta_1t}|y_t|^2dt\leq \mathbb{E}\int_0^\infty e^{-\beta_2t}|y_t|^2dt\\ <\infty$. That is, the larger $\beta>0$ is, the larger corresponding space is. Therefore, it is easy to check the limit process solves \eqref{yinfty}, i.e., $(y(\cdot),z(\cdot),\tilde{z}(\cdot),\gamma(\cdot,\cdot))$ satisfies \eqref{ifbsdep01} on interval $[0,T]$ for arbitrary $T>0$. So, in the sense, $\beta<0$ appearing in Theorem \ref{the22} is unnecessary.
\end{remark}

Up to now, for given $u\in U$, under assumptions {\bf (H0)--(H1)} and {\bf (A1)--(A8)}, the controlled FBSDEP \eqref{iffbsdep01} admits a unique solution $(x^u(\cdot),y^u(\cdot),z^u(\cdot),\tilde{z}^u(\cdot),\gamma^u(\cdot,\cdot))\in L^{2,-\beta}_{\mathcal{F}}(0,\infty;\mathbb{R}^4)\times L^{2,-\beta}_{\mathcal{F},\nu}(0,\infty;\mathbb{R})$. In addition, the cost functional \eqref{cf} becomes
\begin{equation}\label{cf01}
J(u)=\mathbb{E}\bigg[\int_0^\infty e^{-\beta t}\mathcal{Z}_tf(x_t^u,y_t^u,z_t^u,\tilde{z}_t^u,\int_{\mathcal{E}}\gamma_{(t,e)}^u\nu(de),u_t)dt+\phi(y_0^u)\bigg],
\end{equation}
subject to \eqref{iffbsdep01} and \eqref{mathcalZ}, and the admissible control set is denoted by
\begin{equation}\label{adcontrol}
\mathcal{U}_{ad}[0,\infty]=\left\{u\in U\Big|u\in\mathcal{F}_t^\xi,\bar{\mathbb{E}}\int_0^\infty e^{-\beta_1kt}|u_t|^{2k}dt<\infty,\ \text{for any }\beta_1\geq0,\ k\geq1\right\}.
\end{equation}

\begin{remark}\label{rem26}
As mentioned in Remark \ref{rem22}, we can give the well-posedness of cost functional \eqref{cf} by \eqref{cf01}. With {\bf (H1)}, we know that $f$ is quadratic growth with respect to $(x,y,z,\tilde{z},\gamma,u)$. Therefore, it is easy to check that $\bar{\mathbb{E}}\int_0^\infty e^{-\beta t}f(x^u,y^u,z^u,\tilde{z}^u,\int_{\mathcal{E}}\gamma^u\nu(de),u)dt<\infty$ by Lemma \ref{lemma32} in the following and \eqref{adcontrol}. Similarly, we have $\bar{\mathbb{E}}[\phi(y_0^u)]<\infty$. So we get the desired result.
\end{remark}

\section{Ergodic maximum principle}

In this section, we give the partially observed ergodic maximum principle of FBSDEP on infinite horizon.

\subsection{Some important estimates}

First, a necessary lemma is needed mentioned in Confortola \cite{C19}.
\begin{lemma}\label{Confortola}
Let $k\geq2$, and $a,b\in\mathbb{R}$. Then for every $\epsilon>0$, we have
\begin{equation*}
|a+b|^{2k}\leq(1+\epsilon)|a|^{2k}+c_\epsilon|b|^{2k},
\end{equation*}
where $c_\epsilon=\bigg(1-\Big(\frac{1}{1+\epsilon}\Big)^{\frac{1}{2k-1}}\bigg)^{1-2k}$.
\end{lemma}
We then give the following important estimates.
\begin{lemma}\label{lemma31}
Let {\bf (H0)} and {\bf (A1)--(A8)} hold. For any $u(\cdot)\in\mathcal{U}_{ad}[0,\infty]$, let $k\geq1$ such that
\begin{equation}
\begin{aligned}
&\int_0^\infty e^{-\beta_5t}\big(\mathbb{E}\tilde{b}^{2k}(0,0)\big)^{\frac{1}{2k}}+e^{-\beta_5t}\big(\mathbb{E}\sigma^{2k}(0,0)\big)^{\frac{1}{k}}+e^{-\beta_5t}\big(\mathbb{E}\tilde{\sigma}^{2k}(0,0)\big)^{\frac{1}{k}}\\
&\quad+e^{-\beta_5t}\bigg(\mathbb{E}\int_{\mathcal{E}}l^{2k}(0,0,e)\nu(de)\bigg)^{\frac{1}{2k}}dt<\infty,\\
&\mathbb{E}\int_0^\infty e^{-\beta_5kt}\int_{\mathcal{E}}l^{2k}(0,0,e)\nu(de)dt<\infty,\ \mathbb{E}\bigg(\int_0^\infty e^{-\beta_0t}g(x,0,0,0,0,0)dt\bigg)^{2k}<\infty,
\end{aligned}
\end{equation}
for $\beta_0\leq\min\{\beta_3,\beta_4\},\beta_1\leq\min\{\beta_2,\beta_3,\beta_4\},\beta_5\leq\beta_2$.
Then there exists $\beta_0,\cdots,\beta_5>0$ and constants $C_k,C_{k,\epsilon}$ such that the solution to \eqref{iffbsdep01} satisfying
\begin{equation}\label{xy2k}
\begin{aligned}
&\sup_{t\in\mathbb{R}_+}\mathbb{E}\big[e^{-\beta_2kt}|x_t^u|^{2k}\big]
 \leq C_k\bigg[\mathbb{E}|x_0|^{2k}+\mathbb{E}\int_0^\infty e^{-\beta_1kt}|u_t|^{2k}dt+\bigg(\int_0^\infty e^{-\beta_5t}\big(\mathbb{E}\tilde{b}^{2k}(0,0)\big)^{\frac{1}{2k}}dt\bigg)^{2k}\\
&\quad +\bigg(\int_0^\infty e^{-\beta_5t}\big(\mathbb{E}\sigma^{2k}(0,0)\big)^{\frac{1}{k}}dt\bigg)^{k}+\bigg(\int_0^\infty e^{-\beta_5t}\big(\mathbb{E}\tilde{\sigma}^{2k}(0,0)\big)^{\frac{1}{k}}dt\bigg)^k\\
&\quad +\bigg(\int_0^\infty e^{-\beta_5t}\bigg(\mathbb{E}\int_{\mathcal{E}}l^{2k}(0,0,e)\nu(de)\bigg)^{\frac{1}{2k}}dt\bigg)^{2k}+\mathbb{E}\bigg(\int_0^\infty e^{-\beta_5kt}\int_{\mathcal{E}}l^{2k}(0,0,e)\nu(de)dt\bigg)\bigg],\\
&\mathbb{E}\bigg[\sup_{t\in\mathbb{R}_+} e^{-\beta_3kt}|y_t^u|^{2k}+\bigg(\int_0^\infty e^{-\beta_4t}|z^u_t|^2dt\bigg)^k\\
&\quad +\bigg(\int_0^\infty e^{-\beta_4t}|\tilde{z}^u_t|^2dt\bigg)^k+\bigg(\int_0^\infty\int_{\mathcal{E}} e^{-\beta_4t}|\gamma^u_{(t,e)}|^2N(de,dt)\bigg)^k\bigg]\\
&\leq C_{k,\epsilon}\bigg[\mathbb{E}\int_0^\infty e^{-\beta_1kt}|u_t|^{2k}dt+\mathbb{E}\bigg(\int_0^\infty e^{-\beta_0t}g(x^u_t,0,0,0,0,0)dt\bigg)^{2k}\bigg].
\end{aligned}
\end{equation}
\end{lemma}
We give the proof in the Appendix for the reader's convenience. The following lemma is a direct consequence of Lemma \ref{lemma31}.

\begin{lemma}\label{lemma32}
Let {\bf (H0)} and {\bf (A1)--(A8)} hold. For any $u(\cdot)\in\mathcal{U}_{ad}[0,\infty]$, there exists $\beta_1,\cdots,\beta_4>0$ and constant $C_k$ such that for $k\geq1$,
\begin{equation*}
\begin{aligned}
&\sup_{t\in\mathbb{R}_+}\mathbb{E}\big[e^{-\beta_2kt}|x_t^u|^{2k}\big]+\mathbb{E}\bigg[\sup_{t\in\mathbb{R}_+} e^{-\beta_3kt}|y_t^u|^{2k}\bigg]\leq C_k\bigg[1+\mathbb{E}\int_0^\infty e ^{-\beta_1kt}|u_t|^{2k}dt\bigg],\\
&\mathbb{E}\bigg[\bigg(\int_0^\infty e^{-\beta_4t}|z^u_t|^2dt\bigg)^k+\bigg(\int_0^\infty e^{-\beta_4t}|\tilde{z}^u_t|^2dt\bigg)^k+\bigg(\int_0^\infty\int_{\mathcal{E}} e^{-\beta_4t}|\gamma^u_{(t,e)}|^2N(de,dt)\bigg)^k\bigg]\\
&\quad\leq C_k\bigg[1+\mathbb{E}\int_0^\infty e^{-\beta_1kt}|u_t|^{2k}dt\bigg],\quad \mathbb{E}\bigg[\sup_{t\in\mathbb{R}_+}|\mathcal{Z}^u_t|^k\bigg]<\infty,
\end{aligned}
\end{equation*}
where $\beta_1\leq\min\{\beta_2,\beta_3,\beta_4\}$.
\end{lemma}

Applying convex variation technique, for any $(\epsilon,v(\cdot))\in(0,1)\times\mathcal{U}_{ad}[0,\infty]$, let $(x^\epsilon(\cdot),y^\epsilon(\cdot),\\z^\epsilon(\cdot),\tilde{z}^\epsilon(\cdot),\gamma^\epsilon(\cdot,\cdot))\equiv(x^{\bar{u}+\epsilon v}(\cdot),y^{\bar{u}+\epsilon v}(\cdot),z^{\bar{u}+\epsilon v}(\cdot),\tilde{z}^{\bar{u}+\epsilon v}(\cdot),\gamma^{\bar{u}+\epsilon v}(\cdot,\cdot))$ and $\mathcal{Z}^\epsilon(\cdot)\equiv\mathcal{Z}^{\bar{u}+\epsilon v}(\cdot)$ be the solutions to \eqref{iffbsdep01}, \eqref{mathcalZ}, respectively, along with the perturbed control $u^\epsilon(\cdot)=\bar{u}(\cdot)+\epsilon v(\cdot)$, and let $(\bar{x}(\cdot),\bar{y}(\cdot),\bar{z}(\cdot),\bar{\tilde{z}}(\cdot),\bar{\gamma}(\cdot,\cdot))$ and $\bar{\mathcal{Z}}(\cdot)$ be the solutions to \eqref{iffbsdep01}, \eqref{mathcalZ}, respectively, along with the optimal control $\bar{u}(\cdot)$.

Similar as Lemma \ref{lemma32}, we have the following result.

\begin{lemma}\label{lemma33}
Let {\bf (H0)} and {\bf (A1)--(A8)} hold, there exists $\beta_2,\cdots,\beta_4>0$ and constant $C$ such that for $k\geq1$,
\begin{equation*}
\begin{aligned}
&\sup_{t\in\mathbb{R}_+}\mathbb{E}\big[e^{-\beta_2kt}|x_t^\epsilon-\bar{x}_t|^{2k}\big]\leq C\epsilon^{2k},\quad \mathbb{E}\bigg[\sup_{t\in\mathbb{R}_+}e^{-\beta_3kt}|y_t^\epsilon-\bar{y}_t|^{2k}\bigg]\leq C\epsilon^{2k},\\
&\mathbb{E}\bigg[\bigg(\int_0^\infty e^{-\beta_4t}|z_t^\epsilon-\bar{z}_t|^2dt\bigg)^k+\bigg(\int_0^\infty e^{-\beta_4t}|\tilde{z}_t^\epsilon-\bar{\tilde{z}}_t|^2dt\bigg)^k\\
&\quad+\bigg(\int_0^\infty\int_{\mathcal{E}} e^{-\beta_4t}|\gamma^\epsilon_{(t,e)}-\bar{\gamma}_{(t,e)}|^2N(de,dt)\bigg)^k\bigg]\leq C\epsilon^{2k},\quad \mathbb{E}\bigg[\sup_{t\in\mathbb{R}_+}|\mathcal\mathcal{Z}_t^\epsilon-\bar{\mathcal{Z}}_t|^{2k}\bigg]<C\epsilon^{2k}.
\end{aligned}
\end{equation*}
\end{lemma}

Then we introduce the variational equations as follows:
\begin{equation}\label{x1y1}
\left\{
\begin{aligned}
 dx_{(1,t)}&=\big(\bar{\tilde{b}}_xx_{(1,t)}+\bar{\tilde{b}}_uv_t\big)dt+\big(\bar{\sigma}_xx_{(1,t)}+\bar{\sigma}_uv_t\big)dW_t+\big(\bar{\tilde{\sigma}}_xx_{(1,t)}+\bar{\tilde{\sigma}}_uv_t\big)d\xi_t\\
           &\quad+\int_{\mathcal{E}}\big(\bar{l}_xx_{(1,t)}+\bar{l}_uv_t\big)\tilde{N}(de,dt),\\
-dy_{(1,t)}&=\bigg[\bar{g}_xx_{(1,t)}+\bar{g}_yy_{(1,t)}+\bar{g}_zz_{(1,t)}+\bar{g}_{\tilde{z}}\tilde{z}_{(1,t)}+\bar{g}_\gamma\int_{\mathcal{E}}\gamma_{(1,t,e)}\nu(de)+\bar{g}_uv_t\bigg]dt\\
           &\quad -z_{(1,t)}dW_t-\tilde{z}_{(1,t)}d\xi_t-\int_{\mathcal{E}}\gamma_{(1,t,e)}\tilde{N}(de,dt),\ t\in[0,\infty),\\
  x_{(1,0)}&=0,
\end{aligned}
\right.
\end{equation}
and
\begin{equation}\label{mathcalZ1}
\left\{
\begin{aligned}
d\mathcal{Z}_{(1,t)}&=e^{-\frac{\beta}{2}t}\Big[\bar{\mathcal{Z}}\big(\bar{h}_xx_{(1,t)}+\bar{h}_uv_t\big)+\bar{h}\mathcal{Z}_{(1,t)}\Big]d\xi_t,\\
 \mathcal{Z}_{(1,0)}&=0,
\end{aligned}
\right.
\end{equation}
where we denote $\bar{\rho}_{\psi}\equiv\bar{\rho}_{\psi}(\bar{x}_t,\bar{y}_t,\bar{z}_t,\bar{\tilde{z}}_t,\bar{\gamma}_{(t,e)},\bar{u}_t)$, for $\rho=\tilde{b},\sigma,\tilde{\sigma},l,g,h$ and $\psi=x,y,z,\tilde{z},\gamma,u$, for simplicity.

For any $v(\cdot)\in\mathcal{U}_{ad}[0,\infty]$, it is easy to check that \eqref{x1y1} and \eqref{mathcalZ1} admit unique solutions under assumptions {\bf (H0)} and {\bf (A1)--(A8)}, respectively. Moreover, similar as the above two lemmas, we have the following result.

\begin{lemma}\label{lemma34}
Let {\bf (H0)} and {\bf (A1)--(A8)} hold, there exists $\beta_2,\cdots,\beta_4>0$ such that for $k\geq1$,
\begin{equation*}
\begin{aligned}
&\sup_{t\in\mathbb{R}_+}\mathbb{E}\Big[e^{-\beta_2kt}|x_{(1,t)}|^{2k}\Big]<+\infty,\quad \mathbb{E}\bigg[\sup_{t\in\mathbb{R}_+} e^{-\beta_3kt}|y_{(1,t)}|^{2k}\bigg]<+\infty,\quad \mathbb{E}\bigg[\sup_{t\in\mathbb{R}_+}|\mathcal{Z}_{(1,t)}|^{2k}\bigg]<+\infty,\\
\end{aligned}
\end{equation*}
\begin{equation*}
\begin{aligned}
&\mathbb{E}\bigg[\bigg(\int_0^\infty e^{-\beta_4t}|z_{(1,t)}|^2dt\bigg)^k+\bigg(\int_0^\infty e^{-\beta_4t}|\tilde{z}_{(1,t)}|^2dt\bigg)^k\\
&\quad +\bigg(\int_0^\infty\int_{\mathcal{E}} e^{-\beta_4t}|\gamma_{(1,t,e)}|^2N(de,dt)\bigg)^k\bigg]<+\infty.
\end{aligned}
\end{equation*}
\end{lemma}

Next, we set
\begin{equation}
\tilde{\Phi}^\epsilon:=\frac{\Phi^\epsilon-\bar{\Phi}}{\epsilon}-\Phi_1,\ \text{for }\Phi=x,y,z,\tilde{z},\gamma,\mathcal{Z},
\end{equation}
and the following lemma is needed.

\begin{lemma}\label{lemma36}
Let {\bf (H0)} and {\bf (A1)--(A8)} hold, there exists $\beta_2,\beta_3,\beta_4>0$ such that for $k\geq1$,
\begin{equation}\label{variation estimate}
\begin{aligned}
&\lim_{\epsilon\rightarrow0}\sup_{t\in\mathbb{R}_+}\mathbb{E}\Big[e^{-\beta_2kt}|\tilde{x}_t^\epsilon|^{2k}\Big]=0,\quad
 \lim_{\epsilon\rightarrow0}\mathbb{E}\bigg[\sup_{t\in\mathbb{R}_+} e^{-\beta_3kt}|\tilde{y}_t^\epsilon|^{2k}\bigg]=0,\quad \lim_{\epsilon\rightarrow0}\mathbb{E}\bigg[\sup_{t\in\mathbb{R}_+}|\tilde{\mathcal{Z}}_t^\epsilon|^{2k}\bigg]=0,\\
&\lim_{\epsilon\rightarrow0}\mathbb{E}\bigg[\bigg(\int_0^\infty e^{-\beta_4t}|\tilde{z}_t^\epsilon|^2dt\bigg)^k
 +\bigg(\int_0^\infty e^{-\beta_4t}|\tilde{\tilde{z}}_t^\epsilon|^2dt\bigg)^k+\bigg(\int_0^\infty\int_{\mathcal{E}} e^{-\beta_4t}|\tilde{\gamma}_{(t,e)}^\epsilon|^2N(de,dt)\bigg)^k\bigg]=0,\\
\end{aligned}
\end{equation}
\end{lemma}

\begin{proof}
We first prove the first limit in \eqref{variation estimate}. We set
\begin{equation}
\begin{aligned}
\lambda_{\alpha}:&=\lambda_{\alpha}(\bar{x}+\theta\epsilon(\tilde{x}^\epsilon+x_1),\bar{y}+\theta\epsilon(\tilde{y}^\epsilon+y_1),\bar{z}+\theta\epsilon(\tilde{z}^\epsilon+z_1),\\
                 &\qquad \bar{\tilde{z}}+\theta\epsilon(\tilde{\tilde{z}}^\epsilon+\tilde{z}_1),\bar{\gamma}+\theta\epsilon(\tilde{\gamma}^\epsilon+\gamma_1),\bar{u}+\theta\epsilon v),\\
\bar{\lambda}_{\alpha}:&=\lambda_{\alpha}(\bar{x},\bar{y},\bar{z},\bar{\tilde{z}},\bar{\gamma},\bar{u}),\quad \text{for }\lambda=\tilde{b},\sigma,\tilde{\sigma},l,g\ \text{and }\alpha=x,y,z,\tilde{z},\gamma,u.
\end{aligned}
\end{equation}
And $\tilde{x}^\epsilon(\cdot)$ satisfies the following
\begin{equation*}
\begin{aligned}
d\tilde{x}^\epsilon_t&=\bigg[\int_0^1\tilde{b}_xd\theta\tilde{x}^\epsilon_t+\bigg(\int_0^1\tilde{b}_xd\theta-\bar{\tilde{b}}_x\bigg)x_{(1,t)}+\bigg(\int_0^1\tilde{b}_ud\theta-\bar{\tilde{b}}_u\bigg)v_t\bigg]dt\\
&\quad+\bigg[\int_0^1\sigma_xd\theta\tilde{x}^\epsilon_t+\bigg(\int_0^1\sigma_xd\theta-\bar{\sigma}_x\bigg)x_{(1,t)}+\bigg(\int_0^1\sigma_ud\theta-\bar{\sigma}_u\bigg)v_t\bigg]dW_t\\
&\quad+\bigg[\int_0^1\tilde{\sigma}_xd\theta\tilde{x}^\epsilon_t+\bigg(\int_0^1\tilde{\sigma}_xd\theta-\bar{\tilde{\sigma}}_x\bigg)x_{(1,t)}+\bigg(\int_0^1\tilde{\sigma}_ud\theta-\bar{\tilde{\sigma}}_u\bigg)v_t\bigg]d\xi_t\\
&\quad+\int_{\mathcal{E}}\bigg[\int_0^1l_xd\theta\tilde{x}^\epsilon_t+\bigg(\int_0^1l_xd\theta-\bar{l}_x\bigg)x_{(1,t)}+\bigg(\int_0^1l_ud\theta-\bar{l}_u\bigg)v_t\bigg]\tilde{N}(de,dt).
\end{aligned}
\end{equation*}
By the a priori estimate in Lemma \ref{lemma31}, we have
\begin{equation*}
\begin{aligned}
&\sup_{t\in\mathbb{R}_+}\mathbb{E}\big[e^{-\beta_2kt}|x_t^\epsilon|^{2k}\big]\\
&\leq C_k\bigg[\bigg(\int_0^\infty e^{-\beta_5t}\bigg(\mathbb{E}\bigg[\bigg(\int_0^1\tilde{b}_xd\theta-\bar{\tilde{b}}_x\bigg)x_{(1,t)}+\bigg(\int_0^1\tilde{b}_ud\theta-\bar{\tilde{b}}_u\bigg)v_t\bigg]^{2k}\bigg)^{\frac{1}{2k}}dt\bigg)^{2k}\\
&\quad+\bigg(\int_0^\infty e^{-\beta_5t}\bigg(\mathbb{E}\bigg[\bigg(\int_0^1\sigma_xd\theta-\bar{\sigma}_x\bigg)x_{(1,t)}+\bigg(\int_0^1\sigma_ud\theta-\bar{\sigma}_u\bigg)v_t\bigg]^{2k}\bigg)^{\frac{1}{k}}dt\bigg)^{k}\\
&\quad+\bigg(\int_0^\infty e^{-\beta_5t}\bigg(\mathbb{E}\bigg[\bigg(\int_0^1\tilde{\sigma}_xd\theta-\bar{\tilde{\sigma}}_x\bigg)x_{(1,t)}+\bigg(\int_0^1\tilde{\sigma}_ud\theta-\bar{\tilde{\sigma}}_u\bigg)v_t\bigg]^{2k}\bigg)^{\frac{1}{k}}dt\bigg)^{k}\\
&\quad+\bigg(\int_0^\infty e^{-\beta_5t}\bigg(\mathbb{E}\int_{\mathcal{E}}\bigg[\bigg(\int_0^1l_xd\theta-\bar{l}_x\bigg)x_{(1,t)}+\bigg(\int_0^1l_ud\theta-\bar{l}_u\bigg)v_t\bigg]^{2k}\nu(de)\bigg)^{\frac{1}{2k}}dt\bigg)^{2k}\\
\end{aligned}
\end{equation*}
\begin{equation}\label{ineq05}
\begin{aligned}
&\quad+\mathbb{E}\bigg(\int_0^\infty e^{-\beta_5kt}\int_{\mathcal{E}}\bigg[\bigg(\int_0^1l_xd\theta-\bar{l}_x\bigg)x_{(1,t)}+\bigg(\int_0^1l_ud\theta-\bar{l}_u\bigg)v_t\bigg]^{2k}\nu(de)dt\bigg)\bigg]\\
&:=\mathbf{A}_1+\mathbf{A}_2+\mathbf{A}_3+\mathbf{A}_4+\mathbf{A}_5.
\end{aligned}
\end{equation}
Note that
\begin{equation*}
\begin{aligned}
\mathbf{A}_1&\leq C_k\bigg(\int_0^\infty e^{-\beta_5kt}\bigg(\mathbb{E}\bigg(\int_0^1\tilde{b}_xd\theta-\bar{\tilde{b}}_x\bigg)^{2k}|x_{(1,t)}|^{2k}\bigg)dt\bigg)\\
&\quad +C_k\bigg(\int_0^\infty e^{-\beta_5kt}\bigg(\mathbb{E}\bigg(\int_0^1\tilde{b}_ud\theta-\bar{\tilde{b}}_u\bigg)^{2k}|v_t|^{2k}\bigg)dt\bigg)\\
&\leq C_k\bigg(\int_0^\infty e^{-\beta_5kt}\bigg(\mathbb{E}\bigg(\int_0^1\tilde{b}_xd\theta-\bar{\tilde{b}}_x\bigg)^{4k}\bigg)^{\frac{1}{2}}\bigg(\mathbb{E}|x_{(1,t)}|^{4k}\bigg)^{\frac{1}{2}}dt\bigg)\\
&\quad +C_k\bigg(\int_0^\infty e^{-\beta_5kt}\bigg(\mathbb{E}\bigg(\int_0^1\tilde{b}_ud\theta-\bar{\tilde{b}}_u\bigg)^{4k}\bigg)^{\frac{1}{2}}\bigg(\mathbb{E}|v_t|^{4k}\bigg)^{\frac{1}{2}}dt\bigg)\\
&\leq C_k\bigg(\sup_{t\in\mathbb{R}_+}\mathbb{E}\big[e^{-\beta_5kt}|x_{(1,t)}|^{4k}\big]\bigg)^{\frac{1}{2}}\bigg(\int_0^\infty e^{-\frac{\beta_5kt}{2}}\bigg(\mathbb{E}\bigg(\int_0^1\tilde{b}_xd\theta-\bar{\tilde{b}}_x\bigg)^{4k}\bigg)^{\frac{1}{2}}dt\bigg)\\
&\quad +C_k\bigg(\int_0^\infty e^{-\beta_5kt}\bigg(\mathbb{E}\bigg(\int_0^1\tilde{b}_ud\theta-\bar{\tilde{b}}_u\bigg)^{4k}\bigg)dt\bigg)^{\frac{1}{2}}\bigg(\mathbb{E}\int_0^\infty e^{-\beta_5kt}v^{4k}dt\bigg)^{\frac{1}{2}}\\
&\rightarrow0,\quad \text{as}\ \epsilon\rightarrow0,
\end{aligned}
\end{equation*}
and with the same technique, we have $\mathbf{A}_2\rightarrow0,\ \mathbf{A}_3\rightarrow0,\ \text{as}\ \epsilon\rightarrow0$,
\begin{equation*}
\begin{aligned}
\mathbf{A}_4&\leq C_k\bigg(\int_0^\infty e^{-\beta_5kt}\bigg(\mathbb{E}\int_{\mathcal{E}}\bigg(\int_0^1l_xd\theta-\bar{l}_x\bigg)^{2k}\nu(de)|x_{(1,t)}|^{2k}\bigg)dt\bigg)\\
&\quad +C_k\bigg(\int_0^\infty e^{-\beta_5kt}\bigg(\mathbb{E}\int_{\mathcal{E}}\bigg(\int_0^1l_ud\theta-\bar{l}_u\bigg)^{2k}\nu(de)|v_t|^{2k}\bigg)dt\bigg)\\
&\leq C_k\bigg(\int_0^\infty e^{-\beta_5kt}\bigg(\mathbb{E}\bigg(\int_{\mathcal{E}}\bigg(\int_0^1l_xd\theta-\bar{l}_x\bigg)^{2k}\nu(de)\bigg)^2\bigg)^{\frac{1}{2}}\bigg(\mathbb{E}|x_{(1,t)}|^{4k}\bigg)^{\frac{1}{2}}dt\bigg)\\
&\quad +C_k\bigg(\int_0^\infty e^{-\beta_5kt}\bigg(\mathbb{E}\bigg(\int_{\mathcal{E}}\bigg(\int_0^1l_ud\theta-\bar{l}_u\bigg)^{2k}\nu(de)\bigg)^2\bigg)^{\frac{1}{2}}\bigg(\mathbb{E}|v_t|^{4k}\bigg)^{\frac{1}{2}}dt\bigg)\\
&\leq C_k\bigg(\sup_{t\in\mathbb{R}_+}\mathbb{E}\big[e^{-\beta_5kt}|x_{(1,t)}|^{4k}\big]\bigg)^{\frac{1}{2}}\int_0^\infty e^{-\frac{\beta_5kt}{2}}\bigg(\mathbb{E}\bigg(\int_{\mathcal{E}}\bigg(\int_0^1l_xd\theta-\bar{l}_x\bigg)^{2k}\nu(de)\bigg)^2\bigg)^{\frac{1}{2}}dt\\
&\quad +C_k\bigg(\int_0^\infty e^{-\beta_5kt}\bigg(\mathbb{E}\bigg(\int_{\mathcal{E}}\bigg(\int_0^1l_ud\theta-\bar{l}_u\bigg)^{2k}\nu(de)\bigg)^2\bigg)dt\bigg)^{\frac{1}{2}}
 \bigg(\mathbb{E}\int_0^\infty e^{-\beta_5kt}|v_t|^{4k}dt\bigg)^{\frac{1}{2}}\\
&\rightarrow0,\quad \text{as}\ \epsilon\rightarrow0,
\end{aligned}
\end{equation*}
and similarly, $\mathbf{A}_5\rightarrow0, \text{as}\ \epsilon\rightarrow0$. Therefore, by \eqref{ineq05}, the first limit in \eqref{variation estimate} holds.

Next, since
\begin{equation*}
\begin{aligned}
-d\tilde{y}^\epsilon_t&=\bigg[\int_0^1g_xd\theta\tilde{x}_t^\epsilon+\int_0^1g_yd\theta\tilde{y}^\epsilon_t+\int_0^1g_zd\theta\tilde{z}^\epsilon_t+\int_0^1g_{\tilde{z}}d\theta\tilde{\tilde{z}}^\epsilon_t
 +\int_0^1g_\gamma d\theta\int_{\mathcal{E}}\tilde{\gamma}^\epsilon_{(t,e)}\nu(de)\\
&\quad +\bigg(\int_0^1g_xd\theta-\bar{g}_x\bigg)x_{(1,t)}+\bigg(\int_0^1g_yd\theta-\bar{g}_y\bigg)y_{(1,t)}+\bigg(\int_0^1g_zd\theta-\bar{g}_z\bigg)z_{(1,t)}\\
&\quad +\bigg(\int_0^1g_{\tilde{z}}d\theta-\bar{g}_{\tilde{z}}\bigg)\tilde{z}_{(1,t)}+\bigg(\int_0^1g_\gamma d\theta-\bar{g}_\gamma\bigg)\int_{\mathcal{E}}\gamma_{(1,t,e)}\nu(de)\\
&\quad+\bigg(\int_0^1g_ud\theta-\bar{g}_u\bigg)v_t\bigg]dt-\tilde{z}^\epsilon_tdW_t-\tilde{\tilde{z}}^\epsilon_td\xi_t-\int_{\mathcal{E}}\tilde{\gamma}^\epsilon_{(t,e)}\tilde{N}(de,dt),
\end{aligned}
\end{equation*}
by Lemma \ref{lemma31}, we have
\begin{equation}\label{ineq06}
\begin{aligned}
&\mathbb{E}\bigg[\sup_{t\in\mathbb{R}_+} e^{-\beta_3kt}|\tilde{y}_t^\epsilon|^{2k}+\bigg(\int_0^\infty e^{-\beta_4t}|\tilde{z}^\epsilon_t|^2dt\bigg)^k+\bigg(\int_0^\infty e^{-\beta_4t}|\tilde{\tilde{z}}^\epsilon_t|^2dt\bigg)^k\\
&\quad +\bigg(\int_0^\infty\int_{\mathcal{E}} e^{-\beta_4t}|\tilde{\gamma}^\epsilon_{(t,e)}|^2N(de,dt)\bigg)^k\bigg]\\
&\leq C_k\bigg[\mathbb{E}\bigg(\int_0^\infty e^{-\beta_0t}\bigg[\int_0^1g_xd\theta\tilde{x}^\epsilon_t+\bigg(\int_0^1g_xd\theta-\bar{g}_x\bigg)x_{(1,t)}+\bigg(\int_0^1g_yd\theta-\bar{g}_y\bigg)y_{(1,t)}\\
&\qquad +\bigg(\int_0^1g_zd\theta-\bar{g}_z\bigg)z_{(1,t)}+\bigg(\int_0^1g_{\tilde{z}}d\theta-\bar{g}_{\tilde{z}}\bigg)\tilde{z}_{(1,t)}\\
&\qquad +\bigg(\int_0^1g_\gamma d\theta-\bar{g}_\gamma\bigg)\int_{\mathcal{E}}\gamma_{(1,t,e)}\nu(de)+\bigg(\int_0^1g_ud\theta-\bar{g}_u\bigg)v_t\bigg]dt\bigg)^{2k}\bigg]\\
&:=\mathbf{B}_1+\mathbf{B}_2+\mathbf{B}_3+\mathbf{B}_4+\mathbf{B}_5+\mathbf{B}_6+\mathbf{B}_7.
\end{aligned}
\end{equation}
Note that
\begin{equation*}
\mathbf{B}_1\leq C_k\mathbb{E}\bigg(\int_0^\infty e^{-\beta_0kt}|\tilde{x}^\epsilon_t|^{2k} dt\bigg)
 \leq C_k\sup_{t\in\mathcal{R}_+}\mathbb{E}\bigg[e^{-\frac{\beta_0kt}{2}}|\tilde{x}^\epsilon_t|^{2k}\bigg]\bigg(\int_0^\infty e^{-\frac{\beta_0kt}{2}} dt\bigg)\rightarrow0,\quad \text{as}\ \epsilon\rightarrow0,
\end{equation*}
\begin{equation*}
\begin{aligned}
\mathbf{B}_2&\leq \mathbb{E}\bigg(\int_0^\infty e^{-\beta_0kt}|x_{(1,t)}|^{2k}dt\bigg)\bigg(\int_0^\infty \bigg(e^{-\frac{\beta_0t}{2}}\bigg(\int_0^1g_xd\theta-\bar{g}_x\bigg)\bigg)^{\frac{2k}{2k-1}}dt\bigg)^{2k-1}\\
&\leq C_k\bigg(\sup_{t\in\mathbb{R}_+}\mathbb{E}\Big[e^{-\frac{\beta_0kt}{2}}|x_{(1,t)}|^{4k}\Big]\bigg)^{\frac{1}{2}}
 \bigg(\mathbb{E}\bigg(\int_0^\infty \bigg(e^{-\frac{\beta_0t}{2}}\bigg(\int_0^1g_xd\theta-\bar{g}_x\bigg)\bigg)^{\frac{2k}{2k-1}}dt\bigg)^{4k-2}\bigg)^{\frac{1}{2}}\\
&\rightarrow0,\quad \text{as}\ \epsilon\rightarrow0,
\end{aligned}
\end{equation*}
\begin{equation*}
\begin{aligned}
\mathbf{B}_3&\leq \mathbb{E}\bigg[\sup_{t\in\mathcal{R}_+}e^{-\beta_0kt}|y_{(1,t)}|^{2k}\bigg(\int_0^\infty e^{-\frac{\beta_0t}{2}}\bigg(\int_0^1g_yd\theta-\bar{g}_y\bigg)dt\bigg)^{2k}\bigg]\\
&\leq \bigg(\mathbb{E}\bigg(\sup_{t\in\mathcal{R}_+}e^{-2\beta_0kt}|y_{(1,t)}|^{4k}\bigg)\bigg)^{\frac{1}{2}}\bigg(\mathbb{E}
 \bigg(\int_0^\infty e^{-\frac{\beta_0t}{2}}\bigg(\int_0^1g_yd\theta-\bar{g}_y\bigg)dt\bigg)^{4k}\bigg)^{\frac{1}{2}}\rightarrow0,\quad \text{as}\ \epsilon\rightarrow0.
\end{aligned}
\end{equation*}
\begin{equation*}
\begin{aligned}
\mathbf{B}_4&\leq \mathbb{E}\bigg[\bigg(\int_0^\infty e^{-\beta_0t}|z_{(1,t)}|^2dt\bigg)^{k}\bigg(\int_0^\infty e^{-\beta_0t}\bigg(\int_0^1g_zd\theta-\bar{g}_z\bigg)^2dt\bigg)^{k}\bigg]\\
&\leq \bigg(\mathbb{E}\bigg(\int_0^\infty e^{-\beta_0t}|z_{(1,t)}|^2dt\bigg)^{2k}\bigg)^{\frac{1}{2}}\bigg(\mathbb{E}\bigg(\int_0^\infty e^{-\beta_0t}
 \bigg(\int_0^1g_zd\theta-\bar{g}_z\bigg)^2dt\bigg)^{2k}\bigg)^{\frac{1}{2}}\rightarrow0,\quad \text{as}\ \epsilon\rightarrow0.
\end{aligned}
\end{equation*}
Similarly, we have
$\mathbf{B}_5\rightarrow0,\ \text{as}\ \epsilon\rightarrow0$. And
\begin{equation*}
\begin{aligned}
\mathbf{B}_6&\leq \mathbb{E}\bigg[\bigg(\int_0^\infty e^{-\beta_0t}\int_{\mathcal{E}}\gamma^2_{(1,t,e)}\nu(de)dt\bigg)^{k}\bigg(\int_0^\infty e^{-\beta_0t}\bigg(\int_0^1g_\gamma d\theta-\bar{g}_\gamma\bigg)^2dt\bigg)^{k}\bigg]\\
&\leq \bigg(\mathbb{E}\bigg(\int_0^\infty\int_{\mathcal{E}} e^{-\beta_0t}\gamma_{(1,t,e)}^2N(de,dt)\bigg)^{2k}\bigg)^{\frac{1}{2}}
 \bigg(\mathbb{E}\bigg(\int_0^\infty e^{-\beta_0t}\bigg(\int_0^1g_\gamma d\theta-\bar{g}_\gamma\bigg)^2dt\bigg)^{2k}\bigg)^{\frac{1}{2}}\\
&\rightarrow0,\quad \text{as}\ \epsilon\rightarrow0,
\end{aligned}
\end{equation*}
\begin{equation*}
\begin{aligned}
\mathbf{B}_7&\leq \mathbb{E}\bigg[\bigg(\int_0^\infty e^{-\beta_0kt}|v_t|^{2k}dt\bigg)\bigg(\int_0^\infty \bigg(e^{-\frac{\beta_0t}{2}}\bigg(\int_0^1g_u d\theta-\bar{g}_u\bigg)\bigg)^{\frac{2k}{2k-1}}dt\bigg)^{2k-1}\bigg]\\
&\leq C_k\bigg(\mathbb{E}\int_0^\infty e^{-\beta_0kt}|v_t|^{4k}dt\bigg)^{\frac{1}{2}}
 \bigg(\mathbb{E}\bigg(\int_0^\infty \bigg(e^{-\frac{\beta_0t}{2}}\bigg(\int_0^1g_u d\theta-\bar{g}_u\bigg)\bigg)^{\frac{2k}{2k-1}}dt\bigg)^{4k-2}\bigg)^{\frac{1}{2}}\\
&\rightarrow0,\quad \text{as}\ \epsilon\rightarrow0.
\end{aligned}
\end{equation*}
Therefore, by \eqref{ineq06}, the second and fourth limits in \eqref{variation estimate} hold.

Then we prove the third limit in \eqref{variation estimate}. Since $\tilde{\mathcal{Z}}^\epsilon(\cdot)$ satisfies the following equation
\begin{equation*}
\begin{aligned}
d\tilde{\mathcal{Z}}_t^\epsilon&=e^{-\frac{\beta}{2}t}\bigg[\tilde{\mathcal{Z}}^\epsilon_t h(x^\epsilon,u^\epsilon)+\mathcal{Z}_{(1,t)}\big(h(x^\epsilon,u^\epsilon)-\bar{h}\big)
                                +\bar{\mathcal{Z}}_t\int_0^1h_xd\theta\tilde{x}^\epsilon_t\\
                               &\qquad +\bar{\mathcal{Z}}_t\bigg(\int_0^1h_xd\theta-\bar{h}_x\bigg)x_{(1,t)}+\bar{\mathcal{Z}}_t\bigg(\int_0^1h_ud\theta-\bar{h}_u\bigg)v_t\bigg]d\xi_t,
\end{aligned}
\end{equation*}
by B-D-G's inequality, we have
\begin{equation}\label{ineq07}
\begin{aligned}
&\mathbb{E}\bigg[\sup_{t\in\mathbb{R}_+}|\tilde{\mathcal{Z}}_t^\epsilon|^{2k}\bigg]\leq \mathbb{E}\bigg(\int_0^\infty e^{-\beta t}|\tilde{\mathcal{Z}}_t^\epsilon|^2 |h(x^\epsilon,u^\epsilon)|^2dt\bigg)^k
 +\mathbb{E}\bigg(\int_0^\infty e^{-\beta t}|\mathcal{Z}_{(1,t)}|^2\big(h(x^\epsilon,u^\epsilon)-\bar{h}\big)^2dt\bigg)^k\\
&\quad +\mathbb{E}\bigg(\int_0^\infty e^{-\beta t}|\bar{\mathcal{Z}}_t|^2\bigg|\int_0^1h_xd\theta \bigg|^2|\tilde{x}^\epsilon_t|^2dt\bigg)^k
 +\mathbb{E}\bigg(\int_0^\infty e^{-\beta t}|\bar{\mathcal{Z}}_t|^2\bigg(\int_0^1h_xd\theta-\bar{h}_x\bigg)^2|x_{(1,t)}|^2dt\bigg)^k\\
&\quad +\mathbb{E}\bigg(\int_0^\infty e^{-\beta t}|\bar{\mathcal{Z}}_t|^2\bigg(\int_0^1h_ud\theta-\bar{h}_u\bigg)^2|v_t|^2dt\bigg)^k:=\mathbf{C}_1+\mathbf{C}_2+\mathbf{C}_3+\mathbf{C}_4+\mathbf{C}_5.
\end{aligned}
\end{equation}
Note that
\begin{equation*}
\mathbf{C}_1\leq C_k\mathbb{E}\int_0^\infty e^{-\frac{\beta kt}{2}}|\tilde{\mathcal{Z}}_t^\epsilon|^{2k} dt
\leq C_k\int_0^\infty e^{-\frac{\beta kt}{2}}\mathbb{E}\Big[\sup_{t\in\mathbb{R}_+}|\tilde{\mathcal{Z}}_t^\epsilon|^{2k}\Big] dt,
\end{equation*}
thus by Lemma \ref{lemma32} and Lemma \ref{lemma34}, we obtain
\begin{equation*}
\begin{aligned}
\mathbf{C}_2&\leq \mathbb{E}\bigg[\sup_{t\in\mathbb{R}_+}|\mathcal{Z}_{(1,t)}|^{2k}\bigg(\int_0^\infty e^{-\beta t}\big(h(x^\epsilon,u^\epsilon)-\bar{h}\big)^2dt\bigg)^k\bigg]\\
&\leq \bigg(\mathbb{E}\Big[\sup_{t\in\mathbb{R}_+}|\mathcal{Z}_{(1,t)}|^{4k}\Big]\bigg)^{\frac{1}{2}}
 \bigg(\mathbb{E}\bigg(\int_0^\infty e^{-\beta t}\big(h(x^\epsilon,u^\epsilon)-\bar{h}\big)^2dt\bigg)^{2k}\bigg)^{\frac{1}{2}}\rightarrow0,\quad \text{as}\ \epsilon\rightarrow0,
\end{aligned}
\end{equation*}
\begin{equation*}
\begin{aligned}
\mathbf{C}_3&\leq C_k\mathbb{E}\int_0^\infty e^{-\frac{\beta kt}{2}}|\bar{\mathcal{Z}}_t|^{2k}|\tilde{x}^\epsilon_t|^{2k}dt
\leq C_k\int_0^\infty e^{-\frac{\beta kt}{2}}(\mathbb{E}|\bar{\mathcal{Z}}_t|^{4k})^{\frac{1}{2}}(\mathbb{E}|\tilde{x}^\epsilon_t|^{4k})^{\frac{1}{2}}dt\\
&\leq C_k\bigg(\mathbb{E}\Big[\sup_{t\in\mathbb{R}_+}|\bar{\mathcal{Z}}_t|^{4k}\Big]\bigg)^{\frac{1}{2}}
 \int_0^\infty e^{-\frac{\beta kt}{4}}dt\bigg(\sup_{t\in\mathbb{R}_+}\mathbb{E}\Big[e^{-\frac{\beta kt}{2}}|\tilde{x}^\epsilon_t|^{4k}\Big]\bigg)^{\frac{1}{2}}\rightarrow0,\quad \text{as}\ \epsilon\rightarrow0,\\
\mathbf{C}_4&\leq C_k\mathbb{E}\bigg(\int_0^\infty e^{-\frac{\beta kt}{2}}|\bar{\mathcal{Z}}_t|^{2k}\bigg(\int_0^1h_xd\theta-\bar{h}_x\bigg)^{2k}|x_{(1,t)}|^{2k}dt\bigg)\\
&\leq C_k\bigg(\mathbb{E}\Big[\sup_{t\in\mathbb{R}_+}|\bar{\mathcal{Z}}_t|^{8k}\Big]\bigg)^\frac{1}{4}
 \bigg(\int_0^\infty e^{-\frac{\beta kt}{4}}\bigg(\mathbb{E}\bigg(\int_0^1h_xd\theta-\bar{h}_x\bigg)^{8k}\bigg)^{\frac{1}{4}}dt\bigg)\\
&\qquad \times\bigg(\sup_{t\in\mathbb{R}_+}\mathbb{E}\Big[e^{-\frac{\beta kt}{2}}|x_{(1,t)}|^{4k}\Big]\bigg)^{\frac{1}{2}}\rightarrow0,\quad \text{as}\ \epsilon\rightarrow0,\\
\mathbf{C}_5&\leq \mathbb{E}\bigg[\sup_{t\in\mathbb{R}_+}|\bar{\mathcal{Z}}_t|^{2k}\bigg(\int_0^\infty e^{-\beta t}\bigg(\int_0^1h_ud\theta-\bar{h}_u\bigg)^2|v_t|^2dt\bigg)^k\bigg]\\
&\leq \bigg(\mathbb{E}\Big[\sup_{t\in\mathbb{R}_+}|\bar{\mathcal{Z}}_t|^{4k}\Big]\bigg)^{\frac{1}{2}}
 \bigg(\mathbb{E}\bigg(\int_0^\infty e^{-\beta t}\bigg(\int_0^1h_ud\theta-\bar{h}_u\bigg)^2|v_t|^2dt\bigg)^{2k}\bigg)^{\frac{1}{2}}\\
&\leq C_k\bigg(\mathbb{E}\int_0^\infty e^{-\beta kt}|v_t|^{4k}dt
 \bigg(\int_0^\infty \bigg(e^{-\frac{\beta t}{2}}\bigg(\int_0^1h_ud\theta-\bar{h}_u\bigg)^2\bigg)^{\frac{2k}{2k-1}}dt\bigg)^{2k-1}\bigg)^{\frac{1}{2}}\\
&\leq C_k\bigg(\mathbb{E}\int_0^\infty e^{-\beta kt}|v_t|^{8k}dt\bigg)^{\frac{1}{4}}
 \bigg(\mathbb{E}\bigg(\int_0^\infty \bigg(e^{-\frac{\beta t}{2}}\bigg(\int_0^1h_ud\theta-\bar{h}_u\bigg)^2\bigg)^{\frac{2k}{2k-1}}dt\bigg)^{4k-2}\bigg)^{\frac{1}{4}}\\
&\rightarrow0,\quad \text{as}\ \epsilon\rightarrow0.
\end{aligned}
\end{equation*}
From \eqref{ineq07}, the third limit in \eqref{variation estimate} holds by Gronwall's inequality. The proof is complete.
\end{proof}

\begin{remark}
In fact, we can notice that the higher order estimates for $x,y,z,\tilde{z},\gamma,\mathcal{Z}$ are needed besides in $L^2$-space (i.e., $k=1$) in the previous lemmas of this subsection. Indeed, in the partially observed forward-backward stochastic system, one significant difference is the appearance of the Radon-Nikodym derivative $\mathcal{Z}$ defined in \eqref{RN} which satisfies a new ``state equation" when we make the measure transformation, and which increases the difficulty and complexity of estimates. Therefore, we need higher order estimates for $x,y,z,\tilde{z},\gamma,\mathcal{Z}$ and its perturbed terms. We also notice that the high order estimate for the forward state variable $x$ are given in \cite{WWX13}, but in which the estimates with order $(p=2k=2)$ for the backward state variables $(y,z)$ are not enough (too low) to support the derivation of the variational inequality (similar situation can be also seen in \cite{Wu10}, \cite{ZXL18}). However, we should stress that it could not be regarded as a mistake but some repairable  shortcomings. In fact, these small deficiencies can also be corrected and it is easy to check that the higher order estimates for $(y,z)$ can also be proved by the given condition in \cite{WWX13}(or \cite{Wu10}, \cite{ZXL18}). Thus, rectifying these deficiencies can not influence their main results. However, in our infinite horizon case, we have to give the high order estimates $(p=2k>2)$ for all the variables $(x,y,z,\tilde{z},\gamma)$ and $\mathcal{Z}$ rigorously. Otherwise, we fail to obtain the desired variational inequality.
\end{remark}

\begin{remark}
Besides the rigorous required high order estimates, we have to give the suitable norm for the estimates. Note that, especially, the norm for backward variable $y$ need that $\sup$ is in the expectation $\mathbb{E}$ in a strong form, however, the norm for forward variable $x$ need that $\sup$ is out of the expectation $\mathbb{E}$ in a weak form. More importantly, it can not be interchanged. Another find is the appearance of exponential discount $e^{-\beta t}$ in the estimates, and we can prove that these estimates hold for different values of $\beta_i (i=0,1,2,3,4,5)$ and all the individual required $\beta_i$ are less than or equal to the discount factor which make all the lemmas hold.
\end{remark}

\subsection{Variational inequality}

We have the following variational inequality, which is needed for the ergodic maximum principle.

\begin{lemma}\label{variational ineq}
Let {\bf (H0)} and {\bf (A1)--(A8)} hold, then we have
\begin{equation}\label{variation inequality}
\begin{aligned}
&\mathbb{E}\big[\phi_y(\bar{y}_0)y_{(1,0)}\big]
 +\mathbb{E}\int_0^\infty e^{-\beta t}\bigg(\bar{\mathcal{Z}}_t\bar{f}_xx_{(1,t)}+\bar{\mathcal{Z}}_t\bar{f}_yy_{(1,t)}+\bar{\mathcal{Z}}_t\bar{f}_zz_{(1,t)}+\bar{\mathcal{Z}}_t\bar{f}_{\tilde{z}}\tilde{z}_{(1,t)}\\
&\qquad +\bar{\mathcal{Z}}_t\bar{f}_{\gamma}\int_{\mathcal{E}}\gamma_{(1,t,e)}\nu(de)+\bar{f}\mathcal{Z}_{(1,t)}+\bar{\mathcal{Z}}_t\bar{f}_uv_t\bigg)dt\geq0,\quad \text{for any }v(\cdot)\in\mathcal{U}_{ad}[0,\infty],
\end{aligned}
\end{equation}
where the discount factor $\beta\geq\max\limits_i\{\beta_i\}$ and $\beta_i$ are all the individual discount factors appeared in the above estimates.
\end{lemma}
The detailed proof is given in the Appendix for the reader's convenience.

\subsection{Infinite horizon adjoint equations and ergodic maximum principle}

In this section, we give the ergodic maximum principle, in which we also need guarantee the existence and uniqueness of solutions to infinite horizon forward-backward adjoint SDEP.

We introduce the following three adjoint equations. For $t\in[0,\infty)$,
\begin{equation}\label{fadjoint}
\left\{
\begin{aligned}
dp_t&=\big(\bar{\mathcal{Z}}_t\bar{f}_y+\bar{g}_yp_t+\beta p_t\big)dt+\big(\bar{\mathcal{Z}}_t\bar{f}_z+\bar{g}_zp_t\big)dW_t+\big(\bar{\mathcal{Z}}_t\bar{f}_{\tilde{z}}+\bar{g}_{\tilde{z}}p_t\big)d\xi_t\\
    &\quad +\int_{\mathcal{E}}\big(\bar{\mathcal{Z}}_t\bar{f}_\gamma+\bar{g}_\gamma p_t\big)\tilde{N}(de,dt),\\
p_0&=\phi_y(\bar{y}_0),
\end{aligned}
\right.
\end{equation}
\begin{equation}\label{badjoint1}
-dQ_t=\Big(\bar{f}+e^{-\frac{\beta}{2}t}\bar{h}\tilde{M}_t-\beta Q_t\Big)dt-\tilde{M}_td\xi_t,
\end{equation}
\begin{equation}\label{badjoint2}
\begin{aligned}
-dq_t&=\bigg(\bar{\mathcal{Z}}_t\bar{f}_x+\bar{g}_xp_t+\bar{\tilde{b}}_xq_t+\bar{\sigma}_xm_t+\bar{\tilde{\sigma}}_x\tilde{m}_t+\int_{\mathcal{E}}\bar{l}_x n_{(t,e)}\nu(de)
      -\beta q_t+e^{-\frac{\beta}{2}t}\bar{\mathcal{Z}}_t\bar{h}_x\tilde{M}_t\bigg)dt\\
     &\quad -m_tdW_t-\tilde{m}_td\xi_t-\int_{\mathcal{E}}n_{(t,e)}\tilde{N}(de,dt),
\end{aligned}
\end{equation}
whose solutions are $p(\cdot)\in L^{2,-\beta}_{\mathcal{F}}(0,\infty;\mathbb{R})$, $(Q(\cdot),\tilde{M}(\cdot))\in L^{2,-\beta}_{\mathcal{F}}(0,\infty;\mathbb{R}^2)$, and $(q(\cdot),m(\cdot),\tilde{m}(\cdot),\\n(\cdot,\cdot))\in L^{2,-\beta}_{\mathcal{F}}(0,\infty;\mathbb{R}^3)\times L^{2,-\beta}_{\mathcal{F},\nu}(0,\infty;\mathbb{R})$, respectively.

First, we can give the uniquely solvability of forward adjoint SDEP \eqref{fadjoint} directly.

\begin{theorem}\label{the31}
Let {\bf (H1)} and {\bf (A5)} hold. Then there exist constants (i) $\beta$ satisfying $\beta\geq2\mu_1+B^2_{\bar{g}_z}+B^2_{\bar{g}_{\tilde{z}}}+B^2_{\bar{g}_\gamma}$ and (ii) $\mu_1$ satisfying $\mu_1\geq \beta+B_{\bar{g}_y}$, such that infinite horizon SDEP \eqref{fadjoint} admits a unique solution $p(\cdot)\in L^{2,-\beta}_{\mathcal{F}}(0,\infty;\mathbb{R})$, where $B_{\bar{g}_z},B_{\bar{g}_{\tilde{z}}},B_{\bar{g}_\gamma}$ and $B_{\bar{g}_y}$ are set by the bounds of partial derivatives $\bar{g}_z,\bar{g}_{\tilde{z}},\bar{g}_\gamma$ and $\bar{g}_y$ given in {\bf (A5)}.
\end{theorem}

\begin{proof}
Indeed, it is not hard to verify that the coefficients of \eqref{fadjoint} satisfy the monotonicity and Lipschitz conditions by {\bf (A5)} with the condition (i) being given, and all the derivatives of coefficients with respect to $p$ are bounded. Moreover, {\bf (A4)} can also be obtained by {\bf (H1)} and the estimates in Section 3.1. Therefore, under condition (ii), it implies that \eqref{fadjoint} admits a unique solution $p(\cdot)\in L^{2,-\beta}_{\mathcal{F}}(0,\infty;\mathbb{R})$.
\end{proof}

Then we study the uniquely solvability of infinite horizon backward adjoint SDE \eqref{badjoint1}.
\begin{theorem}\label{the32}
Let {\bf (H0)--(H1)} hold. Then there exist constants (i) $\beta$ satisfying $\beta\geq-\mu_2$ and (ii) $\mu_2$ satisfying $\beta+2\mu_2+2e^{-\beta t}B^2_{\bar{h}}<0$, where $B_{\bar{h}}$ is the bound of $\bar{h}$ given in {\bf (H0)}, such that infinite horizon BSDE \eqref{badjoint1} admits a unique solution $(Q(\cdot),\tilde{M}(\cdot))\in L^{2,-\beta}_{\mathcal{F}}(0,\infty;\mathbb{R}^2)$.
\end{theorem}

\begin{proof}
We first consider the following equation in finite horizon $[0,T]$:
\begin{equation}\label{QtildeMT}
\left\{
\begin{aligned}
-dQ_t&=\Big(\bar{f}+e^{-\frac{\beta}{2}t}\bar{h}\tilde{M}_t-\beta Q_t\Big)dt-\tilde{M}_td\xi_t,\\
  Q_T&=0,
\end{aligned}
\right.
\end{equation}
which, by Theorem 4.2 in Briand et al. \cite{BDHPS03}, admits a unique solution $(Q(\cdot),\tilde{M}(\cdot))\in\mathcal{S}^{2k}(0,T)\times L^{2k}_{\mathcal{F}}(0,T;\mathbb{R})$ for $k\geq1$. (In fact, the assumptions (H1)--(H5) in Theorem 4.2 of \cite{BDHPS03} can be satisfied exactly combining with condition (i), and we still set $\mu:=\mu_2$ for convenience. We omit more detailed analysis.) Next, inspired by Theorem 4 in \cite{PS00} in which, using {\bf (H0)--(H1)} and condition (i), (ii), the assumption (H4) are satisfied and $\bar{f}\in L^{2,-\beta}_{\mathcal{F}}(0,\infty;\mathbb{R})$. By the corresponding moments of the state processes $(x,y,z,\tilde{z},\gamma)$, we define for all $k\in N$,
\begin{equation*}
\varphi_t^k:=\mathbb{I}_{[0,k]}(t)\bar{f}\xrightarrow{L^{2,-\beta}_{\mathcal{F}}(0,\infty;\mathbb{R})}\bar{f}\quad \text{as}\  k\rightarrow\infty.
\end{equation*}
Define, for all $k\in N$, $(Q^k(\cdot),\tilde{M}^k(\cdot))$ satisfying
\begin{equation}\label{QtildeMk}
\left\{
\begin{aligned}
-dQ^k_t&=\Big(\varphi_t^k+e^{-\frac{\beta}{2}t}\bar{h}\tilde{M}^k_t-\beta Q^k_t\Big)dt-\tilde{M}^k_td\xi_t,\\
  Q_k^k&=0.
\end{aligned}
\right.
\end{equation}
In fact, \eqref{QtildeMk} can be regarded as the solution to the following equations: on $[0,k]$,
\begin{equation*}
\left\{
\begin{aligned}
-dQ^k_t&=\Big(\bar{f}+e^{-\frac{\beta}{2}t}\bar{h}\tilde{M}^k_t-\beta Q^k_t\Big)dt-\tilde{M}^k_td\xi_t,\\
  Q^k_k&=0,
\end{aligned}
\right.
\end{equation*}
which admits a unique solution $(Q^k(\cdot),\tilde{M}^k(\cdot))$ by the discussion for \eqref{QtildeMT}, and on $(k,\infty)$, it identically equals zero. It follows from Lemma \ref{lem22} with condition (ii) that
\begin{equation*}
\mathbb{E}\int_0^\infty e^{-\beta t}\Big(|\hat{\tilde{M}}_t|^2+|\hat{Q}_t|^2\Big)dt
\leq \frac{1}{\delta}\mathbb{E}\int_0^\infty e^{-\beta t}|\varphi_t^k-\varphi_t^{k^\prime}|^2dt\rightarrow0,\quad \text{as}\ k,k^\prime\rightarrow\infty,
\end{equation*}
where $\hat{\tilde{M}}(\cdot):=\tilde{M}^k(\cdot)-\tilde{M}^{k^\prime}(\cdot)$, $\hat{Q}(\cdot)=Q^k(\cdot)-Q^{k^\prime}(\cdot)$, and $(Q^k(\cdot),\tilde{M}^k(\cdot))$, $(Q^{k^\prime}(\cdot),\tilde{M}^{k^\prime}(\cdot))$ are the solutions to \eqref{QtildeMk} with $\varphi_t^k,\varphi_t^{k^\prime}$, respectively. Thus $\{(Q^k(\cdot),\tilde{M}^k(\cdot))\}$ is a Cauchy sequence in $L^{2,-\beta}_{\mathcal{F}}(0,\infty;\mathbb{R}^2)$, then we can check that its limit $(Q(\cdot),\tilde{M}(\cdot))\in L^{2,-\beta}_{\mathcal{F}}(0,\infty;\mathbb{R}^2)$ satisfies \eqref{badjoint1}. The uniqueness is also obtained by Lemma \ref{lem22}. The proof is complete.
\end{proof}

Then we consider the uniquely solvability of infinite horizon backward adjoint SDEP \eqref{badjoint2}.
\begin{theorem}\label{the33}
Let {\bf (H0)--(H1)} and {\bf (A1)--(A8)} hold. There exist constants $\beta,\mu_2$ satisfying (i) $B_{\bar{\tilde{b}}_x}-\beta\leq\mu_2$ and (ii) $\beta+2\mu_2+2B^2_{\bar{\sigma}_x}+2B^2_{\bar{\tilde{\sigma}}_x}+2B^2_{\bar{l}_x}<0$, where $B_{\bar{\sigma}_x},B_{\bar{\tilde{\sigma}}_x},B_{\bar{l}_x}$ and $B_{\bar{\tilde{b}}_x}$ are bounds decided by {\bf (H0)}, {\bf (A1)}, such that infinite horizon BSDEP \eqref{badjoint2} admits a unique solution $(q(\cdot),m(\cdot),\tilde{m}(\cdot),n(\cdot,\cdot))\in L^{2,-\beta}_{\mathcal{F}}(0,\infty;\mathbb{R}^3)\times L^{2,-\beta}_{\mathcal{F},\nu}(0,\infty;\mathbb{R})$.
\end{theorem}
\begin{proof}
We first consider the following equation in finite horizon $[0,T]$:
\begin{equation}\label{qT}
\left\{
\begin{aligned}
-dq_t&=\tilde{G}\big(q_t,m_t,\tilde{m}_t,n_{(t,e)}\big)dt-m_tdW_t-\tilde{m}_td\xi_t-\int_{\mathcal{E}}n_{(t,e)}\tilde{N}(de,dt),\\
q_T&=0,
\end{aligned}
\right.
\end{equation}
where we set
\begin{equation*}
\tilde{G}(q,m,\tilde{m},n):=\bar{\mathcal{Z}}_t\bar{f}_x+\bar{g}_xp_t+\bar{\tilde{b}}_xq+\bar{\sigma}_xm+\bar{\tilde{\sigma}}_x\tilde{m}+\int_{\mathcal{E}}\bar{l}_x n\nu(de)-\beta q+e^{-\frac{\beta}{2}t}\bar{\mathcal{Z}}_t\bar{h}_x\tilde{M}_t.
\end{equation*}
It is easy to check that the generator $\tilde{G}$ is Lipschitz in $(q,m,\tilde{m},n)$ by {\bf (H0)}, {\bf (A1)}. Then we get $\mathbb{E}\int_0^T|\tilde{G}(0,0,0,0)|^2dt<\infty$. Indeed, by {\bf (H1)} and the estimates in Section 3.1, we have
\begin{equation*}
\mathbb{E}\int_0^T|\tilde{G}(0,0,0,0)|^2dt\leq \mathbb{E}\int_0^T\Big(|\bar{\mathcal{Z}}_t|^2|\bar{f}_x|^2+|\bar{g}_x|^2|p_t|^2+e^{-\beta t}|\bar{\mathcal{Z}}_t|^2|\bar{h}_x|^2|\tilde{M}_t|^2\Big)dt.
\end{equation*}
Note that the first term
\begin{equation*}
\mathbb{E}\int_0^T|\bar{\mathcal{Z}}|^2|\bar{f}_x|^2dt\leq C\mathbb{E}\int_0^T|\bar{\mathcal{Z}}|^2\big(1+|\bar{x}|^2+|\bar{y}|^2+|\bar{z}|^2+|\bar{\tilde{z}}|^2+||\bar{\gamma}||^2+|\bar{u}|^2\big)dt,
\end{equation*}
with
\begin{equation*}
\mathbb{E}\int_0^T|\bar{\mathcal{Z}}|^2dt\leq T\mathbb{E}\bigg[\sup_{t\in[0,T]}|\bar{\mathcal{Z}}|^2\bigg]<\infty,
\end{equation*}
\begin{equation*}
\begin{aligned}
&\mathbb{E}\int_0^T|\bar{\mathcal{Z}}|^2|\bar{x}|^2dt\leq \mathbb{E}\bigg[\sup_{t\in[0,T]}|\bar{\mathcal{Z}}|^2\int_0^Te^{\beta t}e^{-\beta t}|\bar{x}|^2dt\bigg]\\
&\leq \bigg(\mathbb{E}\bigg[\sup_{t\in[0,T]}|\bar{\mathcal{Z}}|^4\bigg]\bigg)^{\frac{1}{2}}\bigg(\mathbb{E}\bigg(\int_0^Te^{\beta t}e^{-\beta t}|\bar{x}|^2dt\bigg)^2\bigg)^{\frac{1}{2}}\\
&\leq C\bigg(\int_0^Te^{2\beta t}dt\bigg)^{\frac{1}{2}}\bigg(\mathbb{E}\int_0^Te^{-2\beta t}|\bar{x}|^4dt\bigg)^{\frac{1}{2}}
 \leq C\bigg(\sup_{t\in[0,T]}\mathbb{E}\big[e^{-2\beta t}|\bar{x}|^4\big]\bigg)^{\frac{1}{2}}<\infty,
\end{aligned}
\end{equation*}
similarly, $\mathbb{E}\int_0^T|\bar{\mathcal{Z}}|^2|\bar{y}|^2dt<\infty$, and
\begin{equation*}
\begin{aligned}
&\mathbb{E}\int_0^T|\bar{\mathcal{Z}}|^2|\bar{z}|^2dt\leq \mathbb{E}\bigg[\sup_{t\in[0,T]}|\bar{\mathcal{Z}}|^2\int_0^Te^{\beta t}e^{-\beta t}|\bar{z}|^2dt\bigg]\\
&\leq \bigg(\mathbb{E}\bigg[\sup_{t\in[0,T]}|\bar{\mathcal{Z}}|^4\bigg]\bigg)^{\frac{1}{2}}\bigg(\mathbb{E}\bigg(\int_0^Te^{\beta t}e^{-\beta t}|\bar{z}|^2dt\bigg)^2\bigg)^{\frac{1}{2}}\\
&\leq C\bigg(\mathbb{E}\bigg(\sup_{t\in[0,T]}e^{\beta t}\int_0^Te^{-\beta t}|\bar{z}|^2dt\bigg)^2\bigg)^{\frac{1}{2}}
 \leq C\bigg(\mathbb{E}\bigg(\int_0^Te^{-\beta t}|\bar{z}|^2dt\bigg)^2\bigg)^{\frac{1}{2}}<\infty,
\end{aligned}
\end{equation*}
\begin{equation*}
\begin{aligned}
&\mathbb{E}\int_0^T|\bar{\mathcal{Z}}|^2||\bar{\gamma}||^2dt\leq \mathbb{E}\bigg[\sup_{t\in[0,T]}|\bar{\mathcal{Z}}|^2\int_0^Te^{\beta t}e^{-\beta t}||\bar{\gamma}||^2dt\bigg]\\
&\leq \bigg(\mathbb{E}\bigg[\sup_{t\in[0,T]}|\bar{\mathcal{Z}}|^4\bigg]\bigg)^{\frac{1}{2}}\bigg(\mathbb{E}\bigg(\int_0^Te^{\beta t}e^{-\beta t}||\bar{\gamma}||^2dt\bigg)^2\bigg)^{\frac{1}{2}}\\
&\leq C\bigg(\mathbb{E}\bigg(\sup_{t\in[0,T]}e^{\beta t}\int_0^Te^{-\beta t}||\bar{\gamma}||^2dt\bigg)^2\bigg)^{\frac{1}{2}}
 \leq C\bigg(\mathbb{E}\bigg(\int_0^T\int_{\mathcal{E}}e^{-\beta t}|\bar{\gamma}|^2N(de,dt)\bigg)^2\bigg)^{\frac{1}{2}}<\infty,
\end{aligned}
\end{equation*}
\begin{equation*}
\begin{aligned}
&\mathbb{E}\int_0^T|\bar{\mathcal{Z}}|^2|\bar{u}|^2dt\leq \mathbb{E}\bigg[\sup_{t\in[0,T]}|\bar{\mathcal{Z}}|^2\int_0^Te^{\beta t}e^{-\beta t}|\bar{u}|^2dt\bigg]\\
&\leq \bigg(\mathbb{E}\bigg[\sup_{t\in[0,T]}|\bar{\mathcal{Z}}|^4\bigg]\bigg)^{\frac{1}{2}}\bigg(\mathbb{E}\bigg(\int_0^Te^{\beta t}e^{-\beta t}|\bar{u}|^2dt\bigg)^2\bigg)^{\frac{1}{2}}
 \leq C\bigg(\mathbb{E}\bigg(\int_0^Te^{-2\beta t}|\bar{u}|^4dt\bigg)\bigg)^{\frac{1}{2}}<\infty,
\end{aligned}
\end{equation*}
the second term
\begin{equation*}
\mathbb{E}\int_0^T|\bar{g}_x|^2|p_t|^2dt\leq C\mathbb{E}\int_0^Te^{\beta t}e^{-\beta t}|p_t|^2dt\leq C\mathbb{E}\int_0^Te^{-\beta t}|p_t|^2dt<\infty,
\end{equation*}
and the third term
\begin{equation*}
\begin{aligned}
&\mathbb{E}\int_0^Te^{-\beta t}|\bar{\mathcal{Z}}|^2|\bar{h}_x|^2|\tilde{M}_t|^2dt\leq C\mathbb{E}\bigg[\sup_{t\in[0,T]}|\bar{\mathcal{Z}}|^2\int_0^Te^{-\beta t}|\tilde{M}_t|^2dt\bigg]\\
&\leq C\bigg(\mathbb{E}\bigg[\sup_{t\in[0,T]}|\bar{\mathcal{Z}}|^4\bigg]\bigg)^{\frac{1}{2}}\bigg(\mathbb{E}\bigg(\int_0^Te^{-\beta t}|\tilde{M}_t|^2dt\bigg)^2\bigg)^{\frac{1}{2}}
 \leq C\bigg(\mathbb{E}\bigg(\int_0^T|\tilde{M}_t|^2dt\bigg)^2\bigg)^{\frac{1}{2}}<\infty,
\end{aligned}
\end{equation*}
where the last line holds is due to the uniquely solvability of \eqref{QtildeMT} and $(Q(\cdot),\tilde{M}(\cdot))\in \mathcal{S}^{2k}(0,T)\times L^{2k}_{\mathcal{F}}(0,T;\mathbb{R})$.
Then we get $\mathbb{E}\int_0^T|\tilde{G}(0,0,0,0)|^2dt<\infty$.

By Proposition A.2 in the Appendix of \cite{QS13}, there exists a unique solution $(q(\cdot),m(\cdot),\tilde{m}(\cdot),\\n(\cdot,\cdot))$ of \eqref{qT} in $L^2_{\mathcal{F}}(0,T;\mathbb{R}^3)\times L^2_{\mathcal{F},\nu}(0,T;\mathbb{R})$.

Next, inspired by \cite{Yu17} in which similar(H2.1)--(H2.4) can be satisfied by the coefficients in our setting by choosing appropriate constants $\beta$ and $\mu_2$ satisfying conditions (i), (ii). However, the difference is that we extend the results into a bigger space $L^{2,-\beta}_{\mathcal{F}}(0,\infty;\mathbb{R})$ with $\beta>0$. Employing the idea in \cite{PS00}, define for all $i\in N$,
\begin{equation*}
\kappa^i_t:=\Big(\bar{\mathcal{Z}}_t\bar{f}_x+\bar{g}_xp_t+e^{-\frac{\beta}{2}t}\bar{\mathcal{Z}}_t\bar{h}_x\tilde{M}_t\Big)\mathbb{I}_{[0,i]}(t),
\end{equation*}
which converges to $\bar{\mathcal{Z}}_t\bar{f}_x+\bar{g}_xp_t+e^{-\frac{\beta}{2}t}\bar{\mathcal{Z}}_t\bar{h}_x\tilde{M}_t$ in $L^{2,-\beta}_{\mathcal{F}}(0,\infty;\mathbb{R})$, as $i\rightarrow\infty$. Then define $(q^i(\cdot),m^i(\cdot),\tilde{m}^i(\cdot),n^i(\cdot,\cdot))$ satisfying
\begin{equation}\label{qmtildemni}
\begin{aligned}
-dq^i_t&=\bigg(\bar{\tilde{b}}_x q^i_t+\bar{\sigma}_x m^i_t+\bar{\tilde{\sigma}}_x \tilde{m}^i_t+\int_{\mathcal{E}}\bar{l}_x n^i_{(t,e)}\nu(de)-\beta q^i_t+\kappa^i_t\bigg)dt\\
&\quad-m^i_tdW_t-\tilde{m}^i_td\xi_t-\int_{\mathcal{E}}n^i_{(t,e)}\tilde{N}(de,dt),\\
\end{aligned}
\end{equation}
which is the solution to the following BSDEP on $[0,i]$:
\begin{equation*}
\left\{
\begin{aligned}
-dq^i_t&=\bigg(\bar{\tilde{b}}_x q^i_t+\bar{\sigma}_x m^i_t+\bar{\tilde{\sigma}}_x \tilde{m}^i_t+\int_{\mathcal{E}}\bar{l}_x n^i_{(t,e)}\nu(de)-\beta q^i_t
        +\bar{\mathcal{Z}}_t\bar{f}_x+\bar{g}_x p_t+e^{-\frac{\beta}{2}t}\bar{\mathcal{Z}}_t\bar{h}_x\tilde{M}_t\bigg)dt\\
       &\quad-m^i_tdW_t-\tilde{m}^i_td\xi_t-\int_{\mathcal{E}}n^i_{(t,e)}\tilde{N}(de,dt),\\
  q_i^i&=0,
\end{aligned}
\right.
\end{equation*}
whose uniquely solvability can be guaranteed by the above discussion about \eqref{qT}. And on $(i,\infty)$, it identically equals zero. Then let $(q^i(\cdot),m^i(\cdot),\tilde{m}^i(\cdot),n^i(\cdot,\cdot))$, $(q^j(\cdot),m^j(\cdot),\tilde{m}^j(\cdot),n^j(\cdot,\cdot))$ are two solutions to \eqref{qmtildemni} with $\kappa^i_t,\kappa^j_t$, respectively. And set $\hat{q}=q^i-q^j,\hat{m}=m^i-m^j,\hat{\tilde{m}}=\tilde{m}^i-\tilde{m}^j,\hat{n}=n^i-n^j$, By Lemma \ref{lem22} and condition (ii), we get
\begin{equation}
\begin{aligned}
&\mathbb{E}\int_0^\infty e^{-\beta t}\Big(|\hat{m}|^2+|\hat{\tilde{m}}|^2+||\hat{n}||^2+|\hat{q}|^2\Big)dt
\leq \frac{1}{\delta}\mathbb{E}\int_0^\infty e^{-\beta t}|\kappa^i_t-\kappa^j_t|^2dt\xrightarrow{i,j\rightarrow\infty}0.
\end{aligned}
\end{equation}
Therefore, $\{(q^i(\cdot),m^i(\cdot),\tilde{m}^i(\cdot),n^i(\cdot,\cdot))\}$ is a Cauchy sequence in $L^{2,-\beta}_{\mathcal{F}}(0,\infty;\mathbb{R}^3)\times L^{2,-\beta}_{\mathcal{F},\nu}(0,\infty;\mathbb{R})$ and its limit $(q(\cdot),m(\cdot),\tilde{m}(\cdot),n(\cdot,\cdot))$ solves \eqref{badjoint2}. The uniqueness is the direct result of a priori estimate in Lemma \ref{lem22}. The proof is complete.
\end{proof}

\begin{remark}
By observing Theorem \ref{the31}, \ref{the32} and \ref{the33}, we can give a sufficient condition to guarantee the uniquely solvability of FBSDEPs \eqref{fadjoint}--\eqref{badjoint2} as follows:
\begin{equation}\label{sc}
\mu_1+\mu_2<\frac{1}{2}\lambda_{min},\quad \beta=\frac{\mu_1-\mu_2}{2}+\frac{B_{\bar{\tilde{b}}_x}-B_{\bar{g}_y}}{2},
\end{equation}
where $\lambda_{min}=\min\big\{-2e^{-\beta t}B^2_{\bar{h}}-B^2_{\bar{g}_z}-B^2_{\bar{g}_{\tilde{z}}}-B^2_{\bar{g}_\gamma},-2B^2_{\bar{\sigma}_x}-2B^2_{\bar{\tilde{\sigma}}_x}
-2B^2_{\bar{l}_x}-B^2_{\bar{g}_z}-B^2_{\bar{g}_{\tilde{z}}}-B^2_{\bar{g}_\gamma}\big\}$. In fact, combining condition (ii) in Theorem \ref{the31} with the second equality in \eqref{sc}, it implies conditions (i) both in Theorem \ref{the32} and in Theorem \ref{the33}.
\end{remark}

Let us introduce the Hamiltonian function $H:[0,\infty )\times\mathbb{R}^{12}\times U\rightarrow\mathbb{R}$ as
\begin{equation}\label{Hamiltonian}
\begin{aligned}
&H(t,x,y,z,\tilde{z},\gamma,\mathcal{Z},u;p,q,m,\tilde{m},n,\tilde{M}):=g(x,y,z,\tilde{z},\gamma,u)p+\tilde{b}(x,u)q+\sigma(x,u)m\\
&+\tilde{\sigma}(x,u)\tilde{m}+\int_{\mathcal{E}}l(x,u,e)n_{e}\nu(de)+e^{-\frac{\beta}{2}t}h(x,u)\tilde{M}\mathcal{Z}+f(x,y,z,\tilde{z},\gamma,u)\mathcal{Z}.
\end{aligned}
\end{equation}
Thus we can rewrite the forward-backward adjoint SDEPs \eqref{fadjoint}--\eqref{badjoint2} as follows:
\begin{equation}\label{fbadjoint1}
\left\{
\begin{aligned}
dp_t&=(\bar{H}_y+\beta p_t)dt+\bar{H}_zdW_t+\bar{H}_{\tilde{z}}d\xi_t+\int_{\mathcal{E}}\bar{H}_\gamma\tilde{N}(de,dt),\\
-dQ_t&=(\bar{H}_{\mathcal{Z}}-\beta Q_t)dt-\tilde{M}_td\xi_t,\\
-dq_t&=(\bar{H}_x-\beta q_t)dt-m_tdW_t-\tilde{m}_td\xi_t-\int_{\mathcal{E}}n_{(t,e)}\tilde{N}(de,dt),\\
p_0&=\phi_y(\bar{y}_0),
\end{aligned}
\right.
\end{equation}
where $\bar{H}_\Gamma:=H_\Gamma(t,\bar{x},\bar{y},\bar{z},\bar{\tilde{z}},\bar{\gamma},\bar{\mathcal{Z}},\bar{u};p,q,m,\tilde{m},n,\tilde{M})$, for $\Gamma=x,y,z,\tilde{z},\gamma,u$.

Before giving our main result, we need the following auxiliary lemma.

\begin{lemma}\label{reasonable condition at infinity}
Under assumptions {\bf (H0)--(H1)} and {\bf (A1)--(A8)}, and let conditions (i), (ii) in Theorem \ref{the31}, \ref{the32} and \ref{the33} hold, respectively. Let $(x_1(\cdot),y_1(\cdot),z_1(\cdot),\tilde{z}_1(\cdot),\gamma_1(\cdot,\cdot))$ and $\mathcal{Z}_1(\cdot)$ satisfy \eqref{x1y1} and \eqref{mathcalZ1}, respectively. Then we have the following limit properties for adjoint processes $(p(\cdot),Q(\cdot),q(\cdot))$:
\begin{equation}\label{transversality condition}
\lim_{t\rightarrow\infty}\mathbb{E}\big[e^{-\beta t}p_ty_{(1,t)}\big]=0,\quad
\lim_{t\rightarrow\infty}\mathbb{E}\big[e^{-\beta t}Q_t\mathcal{Z}_{(1,t)}\big]=0,\quad
\lim_{t\rightarrow\infty}\mathbb{E}\big[e^{-\beta t}q_tx_{(1,t)}\big]=0.
\end{equation}
\end{lemma}

\begin{proof}
First, we obtain that there exists unique solutions $(x_1(\cdot),y_1(\cdot),z_1(\cdot),\tilde{z}_1(\cdot),\gamma_1(\cdot,\cdot),\mathcal{Z}_1(\cdot))$ of \eqref{x1y1}, \eqref{mathcalZ1} in $L^{2,-\beta}_{\mathcal{F}}(0,\infty;\mathbb{R}^5)\times L^{2,-\beta}_{\mathcal{F},\nu}(0,\infty;\mathbb{R})$, and $(p(\cdot),Q(\cdot),\tilde{M}(\cdot),q(\cdot),m(\cdot),\tilde{m}(\cdot),n(\cdot,\cdot))$ of \eqref{fbadjoint1} in $L^{2,-\beta}_{\mathcal{F}}(0,\infty;\mathbb{R}^6)\times L^{2,-\beta}_{\mathcal{F},\nu}(0,\infty;\mathbb{R})$, respectively. Therefore, due to
\begin{equation*}
\mathbb{E}\int_0^\infty e^{-\beta t}p_ty_{(1,t)}dt\leq \bigg(\mathbb{E}\int_0^\infty e^{-\beta t}|p_t|^2dt\bigg)^{\frac{1}{2}}\bigg(\mathbb{E}\int_0^\infty e^{-\beta t}|y_{(1,t)}|^2dt\bigg)^{\frac{1}{2}}<\infty,
\end{equation*}
we have $\lim\limits_{t\rightarrow\infty}\mathbb{E}\big[e^{-\beta t}p_ty_{(1,t)}\big]$ exists. Moreover, we can obtain that
\begin{equation*}
\Big|\mathbb{E}\big[e^{-\beta t}p_ty_{(1,t)}\big]\Big|\leq \mathbb{E}\Big[e^{-\frac{\beta}{2} t}|p_t|e^{-\frac{\beta}{2} t}|y_{(1,t)}|\Big]
\leq \Big(\mathbb{E}\big[e^{-\beta t}|p_t|^2\big]\Big)^{\frac{1}{2}}\Big(\mathbb{E}\big[e^{-\beta t}|y_{(1,t)}|^2\big]\Big)^{\frac{1}{2}}\rightarrow0.
\end{equation*}
Therefore, $\varlimsup\limits_{t\rightarrow\infty}\Big|\mathbb{E}\big[e^{-\beta t}p_ty_{(1,t)}\big]\Big|=0$ and $\varliminf\limits_{t\rightarrow\infty}\Big|\mathbb{E}\big[e^{-\beta t}p_ty_{(1,t)}\big]\Big|=0$, which means that\\ $\lim\limits_{t\rightarrow\infty}\Big|\mathbb{E}\big[e^{-\beta t}p_ty_{(1,t)}\big]-0\Big|=0$. We have proved the first limit equality in \eqref{transversality condition}. The other two limit equalities can be proved similarly. The proof is complete.
\end{proof}

\begin{remark}
Compared with the limit equality (42) in \cite{HOP13}, a distinguished difference shows that, in fact, (42) is an assumption condition given in advance. However, similar limit equalities with exponential weighting in \eqref{transversality condition} can be verified immediately in our setting instead of as a hypothesis, which is a self-provable necessary result for acquiring necessary maximum conditions.
\end{remark}

The following theorem is the main result of this paper.

\begin{theorem}
Let {\bf (H0)--(H1)} and {\bf (A1)--(A8)} hold. Let $\bar{u}(\cdot)$ be an optimal control and $(\bar{x}(\cdot),\bar{y}(\cdot),\bar{z}(\cdot),\bar{\tilde{z}}(\cdot),\bar{\gamma}(\cdot,\cdot))$ the optimal state satisfying \eqref{iffbsdep01} (or equivalently, \eqref{iffbsdep}). Then there exists processes $p(\cdot)\in L^{2,-\beta}_{\mathcal{F}}(0,\infty;\mathbb{R})$, $(Q(\cdot),\tilde{M}(\cdot))\in L^{2,-\beta}_{\mathcal{F}}(0,\infty;\mathbb{R}^2)$ and $(q(\cdot),m(\cdot),\tilde{m}(\cdot),n(\cdot,\cdot))\\\in L^{2,-\beta}_{\mathcal{F}}(0,\infty;\mathbb{R}^3)\times L^{2,-\beta}_{\mathcal{F},\nu}(0,\infty;\mathbb{R})$ satisfying \eqref{fbadjoint1} such that the maximum condition holds:
\begin{equation}\label{ergodic maximum principle}
\mathbb{E}\big[\bar{H}_u|\mathcal{F}^\xi_t\big]\geq0,\quad a.e.\ t\in[0,\infty),\ \mathbb{P}\text{-}a.s.,
\end{equation}
where the Hamiltonian function $H$ is defined as \eqref{Hamiltonian}. Here the discount parameter $\beta$ can be chosen such that the previous lemmas and theorems hold.
\end{theorem}
\begin{proof}
First, for any $T\in[0,\infty)$, applying It\^{o}'s formula to $-e^{-\beta t}p_ty_{(1,t)}+e^{-\beta t}Q_t\mathcal{Z}_{(1,t)}+e^{-\beta t}q_tx_{(1,t)}$ on $[0,T]$ and taking expectation on both sides, where $x_{(1,t)},(y_{(1,t)},z_{(1,t)},\tilde{z}_{(1,t)},\gamma_{(1,t,e)}),\mathcal{Z}_{(1,t)}$ and $p_t,(Q_t,\tilde{M}_t),(q_t,m_t,\tilde{m}_t,n_{(t,e)})$ satisfy variational equations \eqref{x1y1}, \eqref{mathcalZ1} and adjoint equations \eqref{fadjoint}, \eqref{badjoint1} and \eqref{badjoint2}, respectively. Then taking $T\rightarrow\infty$ and substituting it into the variational inequality \eqref{variation inequality}, it follows, for any $v(\cdot)\in\mathcal{U}_{ad}[0,\infty]$, that
\begin{equation*}
\mathbb{E}\int_0^\infty e^{-\beta t}\bar{H}_uv_tdt\geq0,
\end{equation*}
i.e.
\begin{equation*}
\mathbb{E}\int_0^\infty e^{-\beta t}\mathbb{E}[\bar{H}_u|\mathcal{F}_t^\xi]v_tdt\geq0.
\end{equation*}
Utilizing the limit properties in Lemma \ref{reasonable condition at infinity}, we obtain \eqref{ergodic maximum principle}. The proof is complete.
\end{proof}

\section{Comparison with Existing Results and Conclusion}

In this section, we first give two kinds of existing infinite horizon stochastic systems in the literature, together with their related optimal control problems. Some comparison are made with our results in this paper.

To our best knowledge, there exist two kinds of infinite horizon stochastic systems for researchers, which are different in some sense. The first one, which could be traced to the general case in \cite{Peng91}, was a class of BSDE with terminal time being an $\mathcal{F}_t$-stopping time $\tau$ taking values in $[0,\infty]$ which admits a unique solution $(p(\cdot),q(\cdot))\in M^{2,K}(0,\tau;\mathbb{R}^m)\times M^{2,K}(0,\tau;\mathbb{R}^{m\times d})$ with the square-integrable $\mathcal{F}_\tau$-measure terminal value $p_{\tau}=X$. Meanwhile, the existence and uniqueness result holds for the interesting case that $\tau=\infty,X=0$, which is the so called infinite horizon case. That is, the value of $p(\cdot)$ is given at $t=\infty$. In this framework, many results were developed. We only mention some classical literatures such as Darling and Pardoux \cite{DP97} and Pardoux\cite{Pardoux99}, which gave the existence and uniqueness results of solutions to BSDEs with random terminal time under some weaker conditions for the coefficients. For example, generator $f$ are not Lipschitz but monotonicity with respect to $y$. However, the stronger assumptions in \cite{Peng91}, $f$ satisfies Lipschitz condition with respect to both $p$ and $q$, are given.

The second one, which was introduced to be an auxiliary result to prove the uniquely solvability of fully coupled infinite horizon FBSDE in \cite{PS00}, was a kind of infinite horizon BSDE in $[0,\infty)$. It is worth noting that the kind of infinite horizon BSDE did not consider the terminal value at $\infty$ in general. Similarly, the conditions are reinforced by assuming that $G$ is uniformly Lipschitz in $Y$ in \cite{PS00}, and the existence and uniqueness result is proved by a different technique. Some result are extended to the case with random jumps in \cite{Yu17} and the case in a bigger space in \cite{SZ20} under the same strong assumptions. More details, we only make an analysis for the infinite horizon BSDE (3.3) in \cite{PS00} as follows:
\begin{equation*}
-dY_t=\big[G(t,Y_t,Z_t)+\varphi_t\big]dt-Z_tdB_t,\quad t\in[0,\infty).
\end{equation*}
In fact, the infinite horizon BSDE above can be regraded as a general form containing the following equation
\begin{equation}\label{ifbsde}
-dY_t=g(t,Y_t,Z_t)dt-Z_tdB_t,\ t\in[0,\infty),
\end{equation}
when we set $G(t,Y_t,Z_t):=g(t,Y_t,Z_t)-g(t,0,0)$ and $\varphi_t=g(t,0,0)$. Therefore, the Lipschitz property of $G$ with respect to $(Y,Z)$ is equivalent to the Lipschitz property of $g$ with respect to $(Y,Z)$. Then the existence and uniqueness results of solutions to a class of infinite horizon BSDE like \eqref{ifbsde} with generator $g$ can be guaranteed similarly as (3.3) in \cite{PS00}, even for the case with random jumps.
Moreover, we also mention the adjoint equations introduced in \cite{MV14} and \cite{OV17} which are also a class of infinite horizon BSDEs without terminal value at $\infty$, where the corresponding generator of adjoint equation in \cite{MV14} is not Lipschitz in $p$ but Lipschitz in $q$ and that in \cite{OV17} is not Lipschitz neither in $p$ nor $q$, due to the weak monotonicity condition and polynomial growth conditions satisfied by drift and diffusion coefficients in the corresponding (forward) SDE on infinite horizon, respectively. Although the two cases are solved following the same technique of \cite{PS00} in the final step, the auxiliary Theorem 4.1 in \cite{BDHPS03} is utilized necessarily to guarantee the uniquely solvability of their BSDE on the truncated interval $[0,k]$ for preparation. Meanwhile, if we reinforce the conditions for drift and diffusion coefficients in \cite{MV14}, \cite{OV17} being both Lipschitz in $x$, then the problem is naturally simplified and the result also holds.

By comparison, there is an another important difference to point out that the classical methods in \cite{Peng91} is not adapted to the infinite horizon case considered in \cite{PS00}. In other words, the method, which was introduced to solve the uniquely solvability of BSDEs with stopping time, can only be applied especially in a special infinite horizon BSDE when considering $\tau=\infty$ and giving its corresponding value $X=0$. However, for the general infinite horizon BSDE on $[0,\infty)$ where the value at $\infty$ is not given in general in \cite{PS00}, it is ineffective.

For the above two kinds of infinite horizon BSDEs, theories are also applied to the optimal control problem and some interesting results are obtained. For the BSDE with random terminal time, Haadem et al. \cite{HOP13} extended the results to random jumps of Theorem 4.1 in \cite{Pardoux99}. And more importantly, they obtained the sufficient and necessary maximum principles by firstly introducing the limit inequalities in Theorem 4.1 and Theorem 6.1. Because it may be incorrect that the requirement in the finite horizon case that $p(T)=0$ is translated directly into $\lim\limits_{T\rightarrow\infty}p(T)=0$ for the infinite horizon, which was firstly shown in the deterministic case in Shell \cite{Shell69} and Halkin \cite{Halkin74}. In fact, they demonstrated by means of counterexamples that abnormality is possible, and the ``natural" asymptotic conditions like $\lim\limits_{t\rightarrow\infty}p(t)=0$ and $\lim\limits_{t\rightarrow\infty}\left\langle p(t),x(t)\right\rangle=0$ may be violated in the case of infinite horizon problems with free terminal state at infinity. Noting that the discount rate $\beta$ vanishes $(\beta=0)$ in the cost functional similar to the counterexample, it is for a long time that the opinion was common in the economic literature that such pathologies were possible only in the case without exponential discount factor. However, there still exist some counterexamples showing the ``pathological" in the case with positive discount rate nowadays (see Aseev and Kryazhimskii \cite{AK07}, Chapter 1, \S 6). Therefore, they give the limit inequality (16) in \cite{HOP13} to replace the condition satisfied by $p(T)$ at $T=\infty$ indirectly, when applying It\^{o}'s formula during the dual process without $\lim\limits_{T\rightarrow\infty}p(T)=0$. The same thing happens as the condition (42) in \cite{HOP13}, which is a direct assumption given by the authors. However, as far as we know, for some maximum principles discussed in the deterministic case (Aseev and Veliov \cite{AV19}), similar condition as (42) actually only be implied by introducing a kind of Cauchy-type formula and adding some complementary conditions instead of giving immediately. Therefore, we guess it is possible that some additional conditions are required to imply these kinds of asymptotic property of adjoint process in the stochastic case, which would be investigated in our future work.

Another feature in their framework shows that the integral part for running cost in the cost functional is defined similarly as that in finite horizon except for substituting terminal $T$ into $\infty$. In fact, a comparatively strong assumption was given in \cite{HOP13} to guarantee that the cost functional is well-posed. In other words, it is possible for the cost functional to be ill-defined over the infinite horizon without their assumption. Then, the so-called overtaking optimization approach was introduced by von Weizs\"{a}cker \cite{vW65} in 1965 to treat the kind of situation by ``approximately" comparing the values of the cost functional over every finite interval. For this aspect of study, we only mention some recent development in deterministic case (\cite{AV19}) and in stochastic case (Huang et al. \cite{HYZ21}). However, in the second case of infinite horizon BSDE, the same problem can be solved in a different way. We have mentioned that the kind of BSDE on $[0,\infty)$ and its related optimal control problem were considered in \cite{MV14} and \cite{OV17} where they consider the infinite objective integral in the cost functionals with exponential discounting, and they consider the exponential weighting solution space $L^{2,-\beta}_{\mathcal{F}}(\mathbb{R}^+,\mathbb{R}^n)$. We can find that if the adjoint equation admits a unique solution $(p(\cdot),q(\cdot))\in L^{2,-\beta}_{\mathcal{F}}(\mathbb{R}^+,\mathbb{R}^n)\times L^{2,-\beta}_{\mathcal{F}}(\mathbb{R}^+,\mathbb{R}^{n\times d})$, i.e., $\mathbb{E}\int_0^\infty e^{-\beta t}|p_t|^2dt<\infty$, it is easy to obtain that $\lim\limits_{t\rightarrow\infty}\mathbb{E}\big[e^{-\beta t}|p_t|^2\big]=0$, which is more natural used in the dual technique avoiding the limit inequality assumption in \cite{HOP13}. Therefore, we can not give an ideal condition satisfied by $p(\cdot)$ at terminal time like finite horizon case, $p(T)=0$, but we still could give a comparatively reasonable result, i.e., a kind of exponential weighting form for $p(\cdot)$ with expectation is converging to zero when the terminal time goes to infinity (see Lemma \ref{reasonable condition at infinity}).

\section*{Acknowledgement}

Many thanks for discussion and suggestions with Dr. Yuanzhuo Song at Shandong University.

\section*{Appendix}

In the appendix, we give some proofs of the results in Section 3.

\emph{Proof of Lemma \ref{lemma31}}. Applying It\^{o}'s formula to $e^{-\beta kt}|x^u_t|^{2k}$, and by {\bf (A1)--(A4)}, we have
\begin{equation}\label{x2k}
\begin{aligned}
&\mathbb{E}\big[e^{-\beta kt}|x^u_t|^{2k}\big]\leq\mathbb{E}|x_0|^{2k}+\mathbb{E}\int_0^t e^{-\beta ks}\bigg[2k|x_s^u|^{2k-1}\tilde{b}(0,0)+2\mu_1k|x_s^u|^{2k}+2k|x^u_s|^{2k-1}|u_s|\\
&\quad+k(2k-1)|x^u_s|^{2k-2}\big(\sigma^2(0,0)+\tilde{\sigma}^2(0,0)\big)+k(2k-1)|x^u_s|^{2k}+k(2k-1)|x^u_s|^{2k-2}|u_s|^2\\
&\quad+\int_{\mathcal{E}}\Big(|x_s^u|^{2k}+C_k\big(|l(0,0,e)|^{2k}+|x_s^u|^{2k}+|u_s|^{2k}\big)+2k|x^u_s|^{2k-1}l(0,0,e)+2k|x^u_s|^{2k}\\
&\quad+2k|x^u_s|^{2k-1}|u_s|\Big)\nu(de)\bigg]ds-\beta k\mathbb{E}\int_0^t e^{-\beta ks}|x_s^u|^{2k}ds.
\end{aligned}
\end{equation}
By the weighted Young's inequality and H\"{o}lder's inequality, we can get
\begin{equation*}
\begin{aligned}
&\mathbb{E}\int_0^t e^{-\beta ks}2k|x_s^u|^{2k-1}|u_s|ds=2k\int_0^t \mathbb{E}\big[e^{-\beta ks+\frac{\beta}{2}s}|x_s^u|^{2k-1}e^{-\frac{\beta}{2}s}|u_s|\big]ds\\
&\quad \leq 2k\int_0^t\big(\mathbb{E}\big[e^{-\beta ks}|x_s^u|^{2k}\big]\big)^{\frac{2k-1}{2k}}\big(\mathbb{E}\big[e^{-\beta ks}|u_s|^{2k}\big]\big)^{\frac{1}{2k}}ds\\
\end{aligned}
\end{equation*}
\begin{equation*}
\begin{aligned}
&\quad \leq 2k\Big(\sup_{t\in\mathbb{R}_+}\mathbb{E}\big[e^{-\beta ks}|x_s^u|^{2k}\big]\Big)^{\frac{2k-1}{2k}}\int_0^\infty\big(\mathbb{E}\big[e^{-\beta ks}|u_s|^{2k}\big]\big)^{\frac{1}{2k}}ds\\
&\quad \leq (2k-1)\delta^{\frac{2k}{2k-1}}\Big(\sup_{t\in\mathbb{R}_+}\mathbb{E}\big[e^{-\beta ks}|x_s^u|^{2k}\big]\Big)
+\frac{1}{\delta^{2k}}\bigg(\int_0^\infty\big(\mathbb{E}\big[e^{-\beta ks}|u_s|^{2k}\big]\big)^{\frac{1}{2k}}ds\bigg)^{2k},
\end{aligned}
\end{equation*}
\begin{equation*}
\begin{aligned}
&\mathbb{E}\int_0^t e^{-\beta ks}k(2k-1)|x_s^u|^{2k-2}|u_s|^2ds=k(2k-1)\int_0^t \mathbb{E}\big[e^{-\beta (k-1)s}|x_s^u|^{2k-2}e^{-\beta s}|u_s|^2\big]ds\\
&\quad \leq k(2k-1)\int_0^t\big(\mathbb{E}\big[e^{-\beta ks}|x_s^u|^{2k}\big]\big)^{\frac{k-1}{k}}\big(\mathbb{E}\big[e^{-\beta ks}|u_s|^{2k}\big]\big)^{\frac{1}{k}}ds\\
&\quad \leq k(2k-1)\Big(\sup_{t\in\mathbb{R}_+}\mathbb{E}\big[e^{-\beta ks}|x_s^u|^{2k}\big]\Big)^{\frac{k-1}{k}}\int_0^\infty\big(\mathbb{E}\big[e^{-\beta ks}|u_s|^{2k}\big]\big)^{\frac{1}{k}}ds\\
&\quad \leq (k-1)(2k-1)\delta^{\frac{k}{k-1}}\Big(\sup_{t\in\mathbb{R}_+}\mathbb{E}\big[e^{-\beta ks}|x_s^u|^{2k}\big]\Big)
+\frac{(2k-1)}{\delta^{k}}\bigg(\int_0^\infty\big(\mathbb{E}\big[e^{-\beta ks}|u_s|^{2k}\big]\big)^{\frac{1}{k}}ds\bigg)^k,\\
&\mathbb{E}\int_0^t e^{-\beta ks}\int_{\mathcal{E}}2k|x_s^u|^{2k-1}|u_s|\nu(de)ds=2k\int_0^t\int_{\mathcal{E}}\mathbb{E}\big[e^{-\beta ks}|x_s^u|^{2k-1}|u_s|\big]\nu(de)ds\\
&\quad \leq C(2k-1)\delta^{\frac{2k}{2k-1}}\Big(\sup_{t\in\mathbb{R}_+}\mathbb{E}\big[e^{-\beta ks}|x_s^u|^{2k}\big]\Big)
+\frac{C}{\delta^{2k}}\bigg(\int_0^\infty\big(\mathbb{E}\big[e^{-\beta ks}|u_s|^{2k}\big]\big)^{\frac{1}{2k}}ds\bigg)^{2k}.
\end{aligned}
\end{equation*}
Similarly, we have
\begin{equation*}
\begin{aligned}
&\mathbb{E}\int_0^t e^{-\beta ks}2k|x_s^u|^{2k-1}|\tilde{b}(0,0)|ds+\mathbb{E}\int_0^t e^{-\beta ks}k(2k-1)|x_s^u|^{2k-2}(|\sigma(0,0)|^2+|\tilde{\sigma}(0,0)|^2)ds\\
&\quad +\mathbb{E}\int_0^t e^{-\beta ks}2k|x_s^u|^{2k-1}\int_{\mathcal{E}}|\tilde{l}(0,0,e)|\nu(de)ds\\
&\leq C_{\delta,k}\sup_{t\in\mathbb{R}_+}\mathbb{E}\big[e^{-\beta ks}|x_s^u|^{2k}\big]
+\frac{1}{\delta^{2k}}\bigg(\int_0^\infty e^{-\frac{\beta}{2}s}\big(\mathbb{E}|\tilde{b}(0,0)|^{2k}\big)^{\frac{1}{2k}}ds\bigg)^{2k}\\
&\quad +\frac{(2k-1)}{\delta^{k}}\bigg(\int_0^\infty e^{-\beta s}\big(\mathbb{E}|\sigma(0,0)|^{2k}\big)^{\frac{1}{k}}ds\bigg)^k+\frac{(2k-1)}{\delta^{k}}
\bigg(\int_0^\infty e^{-\beta s}\big(\mathbb{E}|\tilde{\sigma}(0,0)|^{2k}\big)^{\frac{1}{k}}ds\bigg)^k\\
&\quad +\frac{C_k}{\delta^{2k}}\bigg(\int_0^\infty e^{-\frac{\beta}{2}s}\bigg(\mathbb{E}\int_{\mathcal{E}}|l(0,0,e)|^{2k}\nu(de)\bigg)^{\frac{1}{2k}}ds\bigg)^{2k}.
\end{aligned}
\end{equation*}
Therefore, \eqref{x2k} becomes
\begin{equation*}
\begin{aligned}
&\mathbb{E}\big[e^{-\beta kt}|x^u_t|^{2k}\big]\leq\mathbb{E}|x_0|^{2k}+(C_{\mu_1,k}-\beta k)\mathbb{E}\int_0^t e^{-\beta ks}|x_s^u|^{2k}ds+C_k\mathbb{E}\int_0^\infty e^{-\beta ks}|u_s|^{2k}ds\\
&\quad+C_{\delta,k}\bigg(\int_0^\infty\big(\mathbb{E}\big[e^{-\beta ks}|u_s|^{2k}\big]\big)^{\frac{1}{2k}}ds\bigg)^{2k}
 +C_{\delta,k}\bigg(\int_0^\infty\big(\mathbb{E}\big[e^{-\beta ks}|u_s|^{2k}\big]\big)^{\frac{1}{k}}ds\bigg)^k\\
&\quad+C_{k,\delta}\sup_{s\in\mathbb{R}_+}\mathbb{E}\big[e^{-\beta ks}|x_s^u|^{2k}\big]+C_{k,\delta}\bigg(\int_0^\infty e^{-\frac{\beta}{2}s}\big(\mathbb{E}|\tilde{b}(0,0)|^{2k}\big)^{\frac{1}{2k}}ds\bigg)^{2k}\\
&\quad+C_{k,\delta}\bigg(\int_0^\infty e^{-\beta s}\big(\mathbb{E}|\sigma(0,0)|^{2k}\big)^{\frac{1}{k}}ds\bigg)^k
 +C_{k,\delta}\bigg(\int_0^\infty e^{-\beta s}\big(\mathbb{E}|\tilde{\sigma}(0,0)|^{2k}\big)^{\frac{1}{k}}ds\bigg)^k\\
&\quad+C_{k,\delta}\bigg(\int_0^\infty e^{-\frac{\beta}{2}s}\bigg(\mathbb{E}\int_{\mathcal{E}}|l(0,0,e)|^{2k}\nu(de)\bigg)^{\frac{1}{2k}}ds\bigg)^{2k}
 +\mathbb{E}\int_0^\infty e^{-\beta ks}\int_{\mathcal{E}}|l(0,0,e)|^{2k}\nu(de)ds.
\end{aligned}
\end{equation*}
Choose suitable $\mu_1$ such that $C_{\mu_1,k}-\beta k<0$, then take suitable $\delta(k)$ such that $1-C_{k,\delta}>0$, it is not hard to obtain the following estimate
\begin{equation*}
\begin{aligned}
&\sup_{t\in\mathbb{R}_+}\mathbb{E}\big[e^{-\beta_2 kt}|x^u_t|^{2k}\big]\leq C_{k}\bigg\{\mathbb{E}|x_0|^{2k}+\mathbb{E}\int_0^\infty e^{-\beta_1 ks}|u_s|^{2k}ds\\
&\quad+\bigg(\int_0^\infty e^{-\beta_5s}\big(\mathbb{E}|\tilde{b}(0,0)|^{2k}\big)^{\frac{1}{2k}}ds\bigg)^{2k}+\bigg(\int_0^\infty e^{-\beta_5 s}\big(\mathbb{E}|\sigma(0,0)|^{2k}\big)^{\frac{1}{k}}ds\bigg)^k\\
&\quad+\bigg(\int_0^\infty e^{-\beta_5 s}\big(\mathbb{E}|\tilde{\sigma}(0,0)|^{2k}\big)^{\frac{1}{k}}ds\bigg)^k
 +\bigg(\int_0^\infty e^{-\beta_5s}\bigg(\mathbb{E}\int_{\mathcal{E}}|l(0,0,e)|^{2k}\nu(de)\bigg)^{\frac{1}{2k}}ds\bigg)^{2k}\\
&\quad+\mathbb{E}\int_0^\infty e^{-\beta_5 ks}\int_{\mathcal{E}}|l(0,0,e)|^{2k}\nu(de)ds\bigg\}.
\end{aligned}
\end{equation*}
where $\beta_1\leq\beta_2,\beta_5\leq\beta_2$. Therefore, the first estimate in \eqref{xy2k} has been proved.

Then we prove the second estimate in \eqref{xy2k}. Applying It\^{o}'s formula to $e^{-\beta t}|y^u_t|^2$, we have
\begin{equation}\label{itoy2}
\begin{aligned}
&|y^u_0|^2+\int_0^Te^{-\beta t}(|z_t^u|^2+|\tilde{z}_t^u|^2)dt+\int_0^T\int_\mathcal{E}e^{-\beta t}|\gamma_{(t,e)}^u|^2N(de,dt)\\
&=e^{-\beta T}|y_T^u|^2+\int_0^T\big[\beta e^{-\beta t}|y_t^u|^2+2e^{-\beta t}y^u_tg\big(x^u_t,y^u_t,z^u_t,\tilde{z}^u_t,\gamma^u_{(t,e)},u_t\big)\big]dt\\
&\quad-\int_0^T2e^{-\beta t}y^u_tz^u_tdW_t-\int_0^T2e^{-\beta t}y^u_t\tilde{z}^u_td\xi_t-\int_0^T\int_{\mathcal{E}}2e^{-\beta t}y^u_t\gamma^u_{(t,e)}\tilde{N}(de,dt).
\end{aligned}
\end{equation}
By {\bf (A6)-(A7)}, we have
\begin{equation*}
\begin{aligned}
&2e^{-\beta t}y^ug(x^u,y^u,z^u,\tilde{z}^u,\gamma^u,u)\leq e^{-\beta t}\Big[2\big(\mu_2+e^{-K_1t}+e^{-K_2t}+e^{-K_3t}+e^{-K_4t}\big)|y^u|^2\\
&\qquad+\frac{1}{2}\big(|z^u|^2+|\tilde{z}^u_t|^2+||\gamma^u||^2+|u|^2\big)+2y^ug(x^u,0,0,0,0,0)\Big],
\end{aligned}
\end{equation*}
and we obtain
\begin{equation*}
\begin{aligned}
&\frac{1}{2}\int_0^Te^{-\beta t}\big(|z^u_t|^2+|\tilde{z}^u_t|^2\big)dt+\int_0^T\int_{\mathcal{E}}e^{-\beta t}|\gamma^u_{(t,e)}|^2N(de,dt)\\
&\leq e^{-\beta T}|y^u_T|^2+\int_0^T\Big[\beta e^{-\beta t}|y^u_t|^2+e^{-\beta t}\big[2\big(\mu_2+e^{-K_1t}+e^{-K_2t}+e^{-K_3t}+e^{-K_4t}\big)|y^u_t|^2\\
&\quad+\frac{1}{2}\big(||\gamma^u_{(t,e)}||^2+|u_t|^2\big)+2y^u_tg(x^u_t,0,0,0,0,0)\big]\Big]dt-\int_0^T2e^{-\beta t}y^u_tz^u_tdW_t\\
&\quad-\int_0^T2e^{-\beta t}y^u_t\tilde{z}^u_td\xi_t-\int_0^T\int_{\mathcal{E}}2e^{-\beta t}y^u_t\gamma^u_{(t,e)}\tilde{N}(de,dt).
\end{aligned}
\end{equation*}
By {\bf (A8)}, for any $\beta>0$, we get
\begin{equation*}
\begin{aligned}
&\int_0^Te^{-\beta t}|z^u_t|^2dt+\int_0^Te^{-\beta t}|\tilde{z}^u_t|^2dt+\int_0^T\int_{\mathcal{E}}e^{-\beta t}|\gamma^u_{(t,e)}|^2N(de,dt)\\
&\leq C\bigg\{\sup_{t\in\mathbb{R}_+}e^{-\beta t}|y^u_t|^2+\frac{1}{2}\int_0^T\int_{\mathcal{E}}e^{-\beta t}|\gamma^u_{(t,e)}|^2\nu(de)dt+\frac{1}{2}\int_0^Te^{-\beta t}|u_t|^2dt\\
\end{aligned}
\end{equation*}
\begin{equation*}
\begin{aligned}
&\qquad +2\int_0^Te^{-\beta t}y^u_tg(x^u_t,0,0,0,0,0)dt-\int_0^T2e^{-\beta t}y^u_tz^u_tdW_t\\
&\qquad -\int_0^T2e^{-\beta t}y^u_t\tilde{z}^u_td\xi_t-\int_0^T\int_{\mathcal{E}}2e^{-\beta t}y^u_t\gamma^u_{(t,e)}\tilde{N}(de,dt)\bigg\}.
\end{aligned}
\end{equation*}
Furthermore, we can deduce that
\begin{equation}\label{ineq03}
\begin{aligned}
&\mathbb{E}\bigg|\int_0^Te^{-\beta t}|z^u_t|^2dt\bigg|^k+\mathbb{E}\bigg|\int_0^Te^{-\beta t}|\tilde{z}^u_t|^2dt\bigg|^k+\mathbb{E}\bigg|\int_0^T\int_{\mathcal{E}}e^{-\beta t}|\gamma^u_{(t,e)}|^2N(de,dt)\bigg|^k\\
&\leq C_k\bigg\{\mathbb{E}\bigg[\sup_{t\in\mathbb{R}_+}e^{-\beta kt}|y^u_t|^{2k}\bigg]+\frac{1}{2^k}\mathbb{E}\bigg|\int_0^T\int_{\mathcal{E}}e^{-\beta t}|\gamma^u_{(t,e)}|^2\nu(de)dt\bigg|^k
 +\frac{1}{2^k}\mathbb{E}\bigg|\int_0^Te^{-\beta t}|u_t|^2dt\bigg|^k\\
&\qquad +\mathbb{E}\bigg|\int_0^Te^{-\beta t}y^u_tg(x^u_t,0,0,0,0,0)dt\bigg|^k+\mathbb{E}\bigg|\int_0^T2e^{-\beta t}y^u_tz^u_tdW_t\bigg|^k\\
&\qquad +\mathbb{E}\bigg|\int_0^T2e^{-\beta t}y^u_t\tilde{z}^u_td\xi_t\bigg|^k+\mathbb{E}\bigg|\int_0^T\int_{\mathcal{E}}2e^{-\beta t}y^u_t\gamma^u_{(t,e)}\tilde{N}(de,dt)\bigg|^k\bigg\}.
\end{aligned}
\end{equation}
Note that
\begin{equation*}
\begin{aligned}
&\mathbb{E}\bigg|\int_0^Te^{-\beta t}y^u_tz^u_tdW_t\bigg|^k\leq C_k\mathbb{E}\bigg|\int_0^Te^{-2\beta t}|y^u_t|^2|z^u_t|^2dt\bigg|^\frac{k}{2}\\
&\leq C_k\mathbb{E}\bigg[\sup_{t\in\mathbb{R}_+}e^{-\frac{\beta kt}{2}}|y^u_t|^k\bigg|\int_0^Te^{-\beta t}|z^u_t|^2dt\bigg|^\frac{k}{2}\bigg]\\
&\leq \frac{C_k}{2}\mathbb{E}\bigg[\sup_{t\in\mathbb{R}_+}e^{-\beta kt}|y^u_t|^{2k}\bigg]+\frac{C_k}{2}\mathbb{E}\bigg|\int_0^\infty e^{-\beta t}|z^u_t|^2dt\bigg|^k,
\end{aligned}
\end{equation*}
\begin{equation*}
\mathbb{E}\bigg|\int_0^Te^{-\beta t}y^u_t\tilde{z}^u_td\xi_t\bigg|^k\leq \frac{C_k}{2}\mathbb{E}\bigg[\sup_{t\in\mathbb{R}_+}e^{-\beta kt}|y^u_t|^{2k}\bigg]
+\frac{C_k}{2}\mathbb{E}\bigg|\int_0^\infty e^{-\beta t}|\tilde{z}^u_t|^2dt\bigg|^k,
\end{equation*}
and
\begin{equation*}
\begin{aligned}
&\mathbb{E}\bigg|\int_0^T\int_{\mathcal{E}}e^{-\beta t}y^u_t\gamma^u_{(t,e)}\tilde{N}(de,dt)\bigg|^k\leq C_k\mathbb{E}\bigg|\int_0^T\int_{\mathcal{E}}e^{-2\beta t}|y^u_t|^2|\gamma^u_{(t,e)}|^2N(de,dt)\bigg|^\frac{k}{2}\\
&\leq C_k\mathbb{E}\bigg[\sup_{t\in\mathbb{R}_+}e^{-\frac{\beta kt}{2}}|y^u_t|^k\bigg|\int_0^T\int_{\mathcal{E}}e^{-\beta t}|y^u_t|^2|\gamma^u_{(t,e)}|^2N(de,dt)\bigg|^\frac{k}{2}\bigg]\\
&\leq \frac{C_k}{2}\mathbb{E}\bigg[\sup_{t\in\mathbb{R}_+}e^{-\beta kt}|y^u_t|^{2k}\bigg]+\frac{C_k}{2}\mathbb{E}\bigg|\int_0^\infty\int_{\mathcal{E}}e^{-\beta t}|\gamma^u_{(t,e)}|^2N(de,dt)\bigg|^k,
\end{aligned}
\end{equation*}
thus, \eqref{ineq03} can be rewritten as
\begin{equation}\label{ineq04}
\begin{aligned}
&\mathbb{E}\bigg(\int_0^\infty e^{-\beta t}|z^u_t|^2dt\bigg)^k+\mathbb{E}\bigg(\int_0^\infty e^{-\beta t}|\tilde{z}^u_t|^2dt\bigg)^k
 +\mathbb{E}\bigg(\int_0^\infty\int_{\mathcal{E}}e^{-\beta t}|\gamma^u_{(t,e)}|^2N(de,dt)\bigg)^k\\
&\leq C_k\bigg\{\mathbb{E}\bigg[\sup_{t\in\mathbb{R}_+}e^{-\beta kt}|y^u_t|^{2k}\bigg]+\mathbb{E}\int_0^\infty e^{-\frac{\beta kt}{2}}|u_t|^{2k}dt
 +\mathbb{E}\bigg(\int_0^\infty e^{-\frac{\beta}{2} t}g(x^u_t,0,0,0,0,0)dt\bigg)^{2k}\bigg\},
\end{aligned}
\end{equation}
where we use the inequality as follows:
\begin{equation}\label{estimate for N}
\mathbb{E}\bigg(\int_0^\infty\int_{\mathcal{E}}e^{-\beta t}|\gamma^u_{(t,e)}|^2\nu(de)dt\bigg)^k\leq \mathbb{E}\bigg(\int_0^\infty\int_{\mathcal{E}}e^{-\beta t}|\gamma^u_{(t,e)}|^2N(de,dt)\bigg)^k.
\end{equation}
Indeed, set $G(t):=\int_{\mathcal{E}}e^{-\beta t}|\gamma^u_{(t,e)}|^2\nu(de)$, and it follows that
\begin{equation*}
\begin{aligned}
\mathbb{E}\bigg(\int_0^\infty G(t)dt\bigg)^k&=\mathbb{E}\int_0^\infty k\bigg(\int_0^sG(r)dr\bigg)^{k-1}G(s)ds\\
&=\mathbb{E}\int_0^\infty k\bigg(\int_0^sG(r)dr\bigg)^{k-1}\int_{\mathcal{E}}e^{-\beta s}|\gamma^u_{(s,e)}|^2\nu(de)ds\\
&=k\mathbb{E}\int_0^\infty \int_{\mathcal{E}}e^{-\beta s}|\gamma^u_{(s,e)}|^2\bigg(\int_0^sG(r)dr\bigg)^{k-1}N(de,ds)\\
&\leq k\mathbb{E}\bigg[\int_0^\infty \int_{\mathcal{E}}e^{-\beta s}|\gamma^u_{(s,e)}|^2N(de,ds)\bigg(\int_0^\infty G(r)dr\bigg)^{k-1}\bigg]\\
&\leq k\bigg[\mathbb{E}\bigg(\int_0^\infty \int_{\mathcal{E}}e^{-\beta s}|\gamma^u_{(s,e)}|^2N(de,ds)\bigg)^k\bigg]^{\frac{1}{k}}\bigg[\mathbb{E}\bigg(\int_0^\infty G(r)dr\bigg)^k\bigg)\bigg]^{\frac{k-1}{k}},\\
\end{aligned}
\end{equation*}
Thus, we have the inequality \eqref{estimate for N}.

Then applying It\^{o}'s formula to $e^{-\beta kt}|y^u_t|^{2k}$, with {\bf (A6)-(A7)} and Remark \ref{remark24}, we obtain
\begin{equation}\label{ineq6}
\begin{aligned}
&e^{-\beta kt}|y_t^u|^{2k}+\int_t^Te^{-\beta ks}k(2k-1)|y_t^u|^{2k-2}(|z_t^u|^2+|\tilde{z}_t^u|^2)ds\\
&\quad +\int_t^Te^{-\beta ks}\int_{\mathcal{E}}\big(|y^u_{s-}+\gamma^u_{(s,e)}|^{2k}-|y^u_{s-}|^{2k}-2k|y^u_{s-}|^{2k-1}\gamma^u_{(s,e)}\big)N(de,ds)\\
&\leq e^{-\beta kT}|y_T^u|^{2k}+\int_t^T(\beta k+2k\mu_2)e^{-\beta ks}|y_s^u|^{2k}ds\\
&\quad +\int_t^Te^{-\beta ks}\Big[2k|y_s^u|^{2k-1}\Big(e^{-\frac{K_1}{2}s}|z_s^u|+e^{-\frac{K_2}{2}s}|\tilde{z}_s^u|+K_3(s)||\gamma_{(s,e)}^u||+e^{-\frac{K_4}{2}s}|u_s|\Big)\\
&\quad +2k|y^u_s|^{2k-1}g(x^u_s,0,0,0,0,0)\Big]ds-\int_t^T2ke^{-\beta ks}|y^u_s|^{2k-1}zdW_s\\
&\quad -\int_t^T2ke^{-\beta ks}|y^u_s|^{2k-1}\tilde{z}d\xi_s-\int_t^T2ke^{-\beta ks}\int_{\mathcal{E}}|y^u_s|^{2k-1}\gamma^u_{(s,e)}\tilde{N}(de,ds).
\end{aligned}
\end{equation}
Note that again
\begin{equation*}
\begin{aligned}
&2k|y|^{2k-1}e^{-\frac{K_1}{2}t}|z|\leq 2ke^{-K_1t}|y|^{2k}+\frac{k(2k-1)}{2}|y|^{2k-2}|z|^2,\\
&2k|y|^{2k-1}e^{-\frac{K_2}{2}t}|\tilde{z}|\leq 2ke^{-K_2t}|y|^{2k}+\frac{k(2k-1)}{2}|y|^{2k-2}|\tilde{z}|^2,\\
&2k|y|^{2k-1}K_3(t)||\gamma||\leq 2k|y|^{2k}+\frac{k}{2}K^2_3(t)|y|^{2k-2}||\gamma||^2,\\
&2k|y|^{2k-1}e^{-\frac{K_4}{2}t}|u|\leq 2ke^{-K_4t}|y|^{2k}+\frac{k}{2}|y|^{2k-2}|u|^2,\\
\end{aligned}
\end{equation*}
it follows that \eqref{ineq6} becomes
\begin{equation}\label{ineq7}
\begin{aligned}
&e^{-\beta kt}|y_t^u|^{2k}+\int_t^Te^{-\beta ks}\frac{k(2k-1)}{2}|y^u_s|^{2k-2}(|z^u_s|^2+|\tilde{z}^u_s|^2)ds\\
&\quad+\int_t^Te^{-\beta ks}\int_{\mathcal{E}}\big(|y^u_{s-}+\gamma^u_{(s,e)}|^{2k}-|y^u_{s-}|^{2k}-2k|y^u_{s-}|^{2k-1}\gamma^u_{(s,e)}\big)N(de,ds)\\
&\leq e^{-\beta kT}|y_T^u|^{2k}+\int_t^T\big(\beta k+2k\mu_2+2ke^{-K_1s}+2ke^{-K_2s}+2k+2ke^{-K_4s}\big)e^{-\beta ks}|y^u_s|^{2k}ds\\
&\quad+\int_t^Te^{-\beta ks}\bigg(\frac{k}{2}K^2_3(s)|y^u_s|^{2k-2}||\gamma^u_{(s,e)}||^2+\frac{k}{2}|y^u_s|^{2k-2}|u_s|^2\\
&\quad+2k|y^u_s|^{2k-1}g(x^u_s,0,0,0,0,0)\bigg)ds-\int_t^T2ke^{-\beta ks}|y^u_s|^{2k-1}z^u_sdW_s\\
&\quad-\int_t^T2ke^{-\beta ks}|y^u_s|^{2k-1}\tilde{z}^u_sd\xi_s-\int_t^T2ke^{-\beta ks}\int_{\mathcal{E}}|y^u_s|^{2k-1}\gamma^u_{(s,e)}\tilde{N}(de,ds).
\end{aligned}
\end{equation}

By Lemma \ref{Confortola}, taking $a=y_{t-}+\gamma$, $b=-y_{t-}$, it follows that $|\gamma|^{2k}\leq (1+\epsilon)|y_{t-}+\gamma|^{2k}+c_\epsilon|y_{t-}|^{2k}$, and
\begin{equation*}
|y_{t-}+\gamma|^{2k}-|y_{t-}|^{2k}\geq\frac{1}{1+\epsilon}|\gamma|^{2k}-\bigg(1+\frac{c_\epsilon}{1+\epsilon}\bigg)|y_{t-}|^{2k}.
\end{equation*}
Then \eqref{ineq7} yields
\begin{equation}\label{ineq8}
\begin{aligned}
&e^{-\beta kt}|y_t^u|^{2k}+\int_t^Te^{-\beta ks}\frac{k(2k-1)}{2}|y^u_s|^{2k-2}(|z^u_s|^2+|\tilde{z}^u_s|^2)ds\\
&\quad +\int_t^Te^{-\beta ks}\int_{\mathcal{E}}\frac{1}{1+\epsilon}|\gamma^u_{(s,e)}|^{2k}N(de,ds)\\
&\leq e^{-\beta kT}|y_T^u|^{2k}+\int_t^T\big(\beta k+2k\mu_2+2ke^{-K_1s}+2ke^{-K_2s}+2k+2ke^{-K_4s}\big)e^{-\beta ks}|y^u_s|^{2k}ds\\
&\quad +\int_t^Te^{-\beta ks}\int_{\mathcal{E}}\bigg(1+\frac{c_\epsilon}{1+\epsilon}\bigg)|y^u_{s-}|^{2k}N(de,ds)+\int_t^Te^{-\beta ks}\int_{\mathcal{E}}2k|y^u_{s-}|^{2k-1}\gamma^u_{(s,e)}N(de,ds)\\
&\quad +\int_t^Te^{-\beta ks}\bigg(\frac{k}{2}K^2_3(s)|y^u_s|^{2k-2}||\gamma^u_{(s,e)}||^2+\frac{k}{2}|y^u_s|^{2k-2}|u_s|^2\\
&\quad +2k|y^u_s|^{2k-1}g(x^u_s,0,0,0,0,0)\bigg)ds-\int_t^T2ke^{-\beta ks}|y^u_s|^{2k-1}z^u_sdW_s\\
&\quad -\int_t^T2ke^{-\beta ks}|y^u_s|^{2k-1}\tilde{z}^u_sd\xi_s-\int_t^T2ke^{-\beta ks}\int_{\mathcal{E}}|y^u_s|^{2k-1}\gamma^u_{(s,e)}\tilde{N}(de,ds).
\end{aligned}
\end{equation}
Then, taking $t=0$ and choosing $\delta$ sufficiently small (the value of $\delta$ will be taken in the following), we have
\begin{equation}\label{ineq9}
\begin{aligned}
&\mathbb{E}\int_0^Te^{-\beta ks}\frac{k(2k-1)}{2}|y^u_s|^{2k-2}\big(|z^u_s|^2+|\tilde{z}^u_s|^2\big)ds+\mathbb{E}\int_0^Te^{-\beta ks}\int_{\mathcal{E}}\frac{1}{1+\epsilon}|\gamma^u_{(s,e)}|^{2k}N(de,ds)\\
&\leq \mathbb{E}\big[e^{-\beta kT}|y_T^u|^{2k}\big]+\mathbb{E}\int_0^Te^{-\beta ks}\int_{\mathcal{E}}2k|y^u_{s-}|^{2k-1}\gamma^u_{(s,e)}\nu(de)ds\\
&\quad +\mathbb{E}\int_0^T\bigg(\beta k+2k\mu_2+2ke^{-K_1s}+2ke^{-K_2s}+2k+2ke^{-K_4s}+\frac{C(1+\epsilon+c_\epsilon)}{1+\epsilon}\bigg)e^{-\beta ks}|y^u_s|^{2k}ds\\
&\quad +\mathbb{E}\int_0^Te^{-\beta ks}\bigg(\frac{k}{2}K^2_3(s)|y^u_s|^{2k-2}||\gamma^u_{(s,e)}||^2+\frac{k}{2}|y^u_s|^{2k-2}|u_s|^2+2k|y^u_s|^{2k-1}g(x^u_s,0,0,0,0,0)\bigg)ds,\\
&\leq \mathbb{E}\big[e^{-\beta kT}|y_T^u|^{2k}\big]+\mathbb{E}\int_0^Te^{-\beta ks}\int_{\mathcal{E}}\bigg(2\delta k+\frac{k}{2}K^2_3(s)\bigg)|y^u_s|^{2k-2}|\gamma^u_{(s,e)}|^2\nu(de)ds\\
&\quad +\mathbb{E}\int_0^T\bigg(\beta k+2k\mu_2+2ke^{-K_1s}+2ke^{-K_2s}+2k+2ke^{-K_4s}+\frac{C(1+\epsilon+c_\epsilon)}{1+\epsilon}\\
&\qquad\quad +\frac{Ck}{2\delta}\bigg)e^{-\beta ks}|y^u_s|^{2k}ds+\mathbb{E}\int_0^Te^{-\beta ks}\bigg(\frac{k}{2}|y^u_s|^{2k-2}|u_s|^2+2k|y^u_s|^{2k-1}g(x^u_s,0,0,0,0,0)\bigg)ds,
\end{aligned}
\end{equation}
and
\begin{equation}\label{ineq10}
\begin{aligned}
&\mathbb{E}\bigg[\sup_{t\in[0,T]}e^{-\beta kt}|y_t^u|^{2k}\bigg]\leq \mathbb{E}\big[e^{-\beta kT}|y_T^u|^{2k}\big]\\
&\quad +\mathbb{E}\int_0^Te^{-\beta ks}\int_{\mathcal{E}}\bigg(2\delta k+\frac{k}{2}K^2_3(s)\bigg)|y^u_s|^{2k-2}|\gamma^u_{(s,e)}|^2\nu(de)ds\\
&\quad +\mathbb{E}\int_0^T\bigg(\beta k+2k\mu_2+2ke^{-K_1s}+2ke^{-K_2s}+2k+2ke^{-K_4s}+\frac{C(1+\epsilon+c_\epsilon)}{1+\epsilon}\\
&\qquad +\frac{Ck}{2\delta}\bigg)e^{-\beta ks}|y^u_s|^{2k}ds+\mathbb{E}\int_0^Te^{-\beta ks}\bigg(\frac{k}{2}|y^u_s|^{2k-2}|u_s|^2+2k|y^u_s|^{2k-1}g(x^u_s,0,0,0,0,0)\bigg)ds\\
&\quad +C_k\mathbb{E}\bigg(\int_0^Te^{-2\beta ks}|y^u_s|^{4k-2}|z^u_s|^2ds\bigg)^{\frac{1}{2}}+C_k\mathbb{E}\bigg(\int_0^Te^{-2\beta ks}|y^u_s|^{4k-2}|\tilde{z}^u_s|^2ds\bigg)^{\frac{1}{2}}\\
&\quad +C_k\mathbb{E}\bigg(\int_0^Te^{-2\beta ks}\int_{\mathcal{E}}|y^u_s|^{4k-2}|\gamma^u_{(s,e)}|^2 N(de,ds)\bigg)^{\frac{1}{2}}.
\end{aligned}
\end{equation}
Note that
\begin{equation*}
\begin{aligned}
&C_k\mathbb{E}\bigg(\int_0^Te^{-2\beta ks}|y^u_s|^{4k-2}|z^u_s|^2ds\bigg)^{\frac{1}{2}}
 \leq \frac{1}{6}\mathbb{E}\bigg[\sup_{s\in[0,T]}e^{-\beta ks}|y^u_s|^{2k}\bigg]+\frac{3C_k^2}{2}\mathbb{E}\int_0^Te^{-\beta ks}|y|^{2k-2}|z^u_s|^2ds,\\
&C_k\mathbb{E}\bigg(\int_0^Te^{-2\beta ks}|y^u_s|^{4k-2}|\tilde{z}^u_s|^2ds\bigg)^{\frac{1}{2}}\leq \frac{1}{6}\mathbb{E}\bigg[\sup_{s\in[0,T]}e^{-\beta ks}|y^u_s|^{2k}\bigg]
 +\frac{3C_k^2}{2}\mathbb{E}\int_0^Te^{-\beta ks}|y^u_s|^{2k-2}|\tilde{z}^u_s|^2ds,\\
 &C_k\mathbb{E}\bigg(\int_0^Te^{-2\beta ks}\int_{\mathcal{E}}|y^u_s|^{4k-2}|\gamma^u_{(s,e)}|^2 N(de,ds)\bigg)^{\frac{1}{2}}\leq \frac{1}{6}\mathbb{E}\bigg[\sup_{s\in[0,T]}e^{-\beta ks}|y^u_s|^{2k}\bigg]\\
\end{aligned}
\end{equation*}
\begin{equation*}
\begin{aligned}
&+\frac{3C_k^2\lambda}{2}\mathbb{E}\int_0^Te^{-\beta ks}|y^u_s|^{2k}ds
 +\frac{3C_k^2}{2\lambda}\mathbb{E}\int_0^T\int_{\mathcal{E}}e^{-\beta ks}|\gamma^u_{(s,e)}|^{2k}N(de,ds),
\end{aligned}
\end{equation*}
where we take $\lambda(k)=\frac{1}{6k}<1$ sufficiently small. Then \eqref{ineq10} becomes
\begin{equation}\label{ineq11}
\begin{aligned}
&\mathbb{E}\bigg[\sup_{t\in[0,T]}e^{-\beta kt}|y_t^u|^{2k}\bigg]\\
&\leq 2\mathbb{E}\big[e^{-\beta kT}|y_T^u|^{2k}\big]+\mathbb{E}\int_0^Te^{-\beta kt}\int_{\mathcal{E}}\big[4\delta k+kK^2_3(t)\big]|y^u_t|^{2k-2}|\gamma^u_{(t,e)}|^2\nu(de)dt\\
&\quad +\mathbb{E}\int_0^T2\bigg(\beta k+2k\mu_2+2ke^{-K_1t}+2ke^{-K_2t}+2k+2ke^{-K_4t}+\frac{C(1+\epsilon+c_\epsilon)}{1+\epsilon}\\
&\qquad +\frac{Ck}{2\delta}+\frac{3C_k^2\lambda}{2}\bigg)e^{-\beta kt}|y^u_t|^{2k}dt+2\mathbb{E}\int_0^Te^{-\beta kt}\bigg(\frac{k}{2}|y^u_t|^{2k-2}|u_t|^2\\
&\qquad +2k|y^u_t|^{2k-1}g(x^u_t,0,0,0,0,0)\bigg)dt\\
&\quad +\frac{3C_k^2}{\lambda}\bigg[\mathbb{E}\int_0^Te^{-\beta kt}|y^u_t|^{2k-2}\big(|z^u_t|^2+|\tilde{z}^u_t|^2\big)dt+\mathbb{E}\int_0^T\int_{\mathcal{E}}e^{-\beta kt}|\gamma^u_{(t,e)}|^{2k}N(de,dt)\bigg].
\end{aligned}
\end{equation}
It follows from \eqref{ineq9} that
\begin{equation}\label{ineq12}
\begin{aligned}
&\mathbb{E}\bigg[\sup_{t\in[0,T]}e^{-\beta kt}|y_t^u|^{2k}\bigg]\leq \bigg(2+\frac{3C_k^2}{\lambda}M(k,\epsilon)^{-1}\bigg)\mathbb{E}\big[e^{-\beta kT}|y_T^u|^{2k}\big]\\
&\quad +\bigg(1+\frac{3C_k^2}{2\lambda}M(k,\epsilon)^{-1}\bigg)\mathbb{E}\int_0^Te^{-\beta kt}\Big(4\delta k+kK^2_3(t)\Big)\int_{\mathcal{E}}|y^u_t|^{2k-2}|\gamma^u_{(t,e)}|^2\nu(de)dt\\
&\quad +\bigg(1+\frac{3C_k^2}{2\lambda}M(k,\epsilon)^{-1}\bigg)\mathbb{E}\int_0^T2\bigg(\beta k+2k\mu_2+2ke^{-K_1t}+2ke^{-K_2t}+2k+2ke^{-K_4t}\\
&\qquad +\frac{C(1+\epsilon+c_\epsilon)}{1+\epsilon}+\frac{Ck}{2\delta}+\frac{3C_k^2\lambda}{2}\bigg)e^{-\beta kt}|y^u_t|^{2k}dt\\
&\quad +\bigg(2+\frac{3C_k^2}{\lambda}M(k,\epsilon)^{-1}\bigg)\mathbb{E}\int_0^Te^{-\beta kt}\bigg[\frac{k}{2}|y^u_t|^{2k-2}|u_t|^2+2k|y^u_t|^{2k-1}g(x^u_t,0,0,0,0,0)\bigg]dt,
\end{aligned}
\end{equation}
where we set $M(k,\epsilon)=\min\big\{\frac{k(2k-1)}{2},\frac{1}{1+\epsilon}\big\}$. Then taking $\beta k+2k\mu_2+2ke^{-K_1t}+2ke^{-K_2t}+2k+2ke^{-K_4t}+\frac{C(1+\epsilon+c_\epsilon)}{1+\epsilon}+\frac{Ck}{2\delta}+\frac{3C_k^2\lambda}{2}\leq0$ and $\delta(k)=\frac{2\lambda^2-kK^2_3(t)}{4k}$ (where $K_3(t)$ is taken sufficiently small satisfied {\bf (A7)}), we have
\begin{equation}\label{ineq13}
\begin{aligned}
&\mathbb{E}\bigg[\sup_{t\in[0,T]}e^{-\beta kt}|y_t^u|^{2k}\bigg]\leq \bigg(2+\frac{3C_k^2}{\lambda}M(k,\epsilon)^{-1}\bigg)\mathbb{E}\big[e^{-\beta kT}|y_T^u|^{2k}\big]\\
&\quad +\bigg(2\lambda^2+3C_k^2\lambda M(k,\epsilon)^{-1}\bigg)\mathbb{E}\int_0^Te^{-\beta kt}\int_{\mathcal{E}}|y^u_t|^{2k-2}|\gamma^u_{(t,e)}|^2\nu(de)dt\\
&\quad +2k\bigg(2+\frac{3C_k^2}{\lambda}M(k,\epsilon)^{-1}\bigg)\mathbb{E}\int_0^Te^{-\beta kt}\Big(|y^u_t|^{2k-2}|u_t|^2+|y^u_t|^{2k-1}g(x^u_t,0,0,0,0,0)\Big)dt.
\end{aligned}
\end{equation}
Noting that
\begin{equation*}
\begin{aligned}
&\bigg(2\lambda^2+3C_k^2\lambda M(k,\epsilon)^{-1}\bigg)\mathbb{E}\int_0^Te^{-\beta kt}\int_{\mathcal{E}}|y^u_t|^{2k-2}|\gamma^u_{(t,e)}|^2\nu(de)dt\\
&\leq\bigg(2\lambda^2+3C_k^2\lambda M(k,\epsilon)^{-1}\bigg)^{\frac{k}{2(k-1)}}\frac{k-1}{k}\mathbb{E}\bigg[\sup_{t\in[0,T]}e^{-\beta kt}|y^u_t|^{2k}\bigg]\\
&\quad+\bigg(2\lambda^2+3C_k^2\lambda M(k,\epsilon)^{-1}\bigg)^{\frac{k}{2}}\frac{1}{k}\mathbb{E}\bigg(\int_0^T\int_{\mathcal{E}}e^{-\beta t}|\gamma^u_{(t,e)}|^2N(de,dt)\bigg)^k,
\end{aligned}
\end{equation*}
and taking $\alpha(k)=\lambda^2(k)$ such that
\begin{equation*}
\begin{aligned}
&2k\bigg(2+\frac{3C_k^2}{\lambda}M(k,\epsilon)^{-1}\bigg)\mathbb{E}\int_0^Te^{-\beta kt}|y^u_t|^{2k-2}|u_t|^2dt\\
&\leq2k\bigg(2+\frac{3C_k^2}{\lambda}M(k,\epsilon)^{-1}\bigg)\bigg[\alpha \mathbb{E}\bigg[\sup_{t\in[0,T]}e^{-\beta kt}|y^u_t|^{2k}\bigg]+C_\alpha \mathbb{E}\bigg(\int_0^Te^{-\beta t}|u_t|^2dt\bigg)^k\bigg]\\
&\leq\bigg(4\lambda^2k+6C_k\lambda M(k,\epsilon)^{-1}\bigg)\mathbb{E}\bigg[\sup_{t\in[0,T]}e^{-\beta kt}|y^u_t|^{2k}\bigg]\\
&\quad +\bigg(4k+\frac{6C_k}{\lambda}M(k,\epsilon)^{-1}\bigg)C_\alpha \mathbb{E}\bigg(\int_0^Te^{-\beta t}|u_t|^2dt\bigg)^k,
\end{aligned}
\end{equation*}
and
\begin{equation*}
\begin{aligned}
&2k\bigg(2+\frac{3C_k^2}{\lambda}M(k,\epsilon)^{-1}\bigg)\mathbb{E}\int_0^Te^{-\beta kt}|y^u_t|^{2k-1}g(x^u_t,0,0,0,0,0)dt\\
&\leq \bigg(4\lambda^2k+6C_k\lambda M(k,\epsilon)^{-1}\bigg)\mathbb{E}\bigg[\sup_{t\in[0,T]}e^{-\beta kt}|y^u_t|^{2k}\bigg]\\
&\quad+\bigg(4k+\frac{6C_k}{\lambda}M(k,\epsilon)^{-1}\bigg)C_\alpha \mathbb{E}\bigg(\int_0^Te^{-\frac{\beta}{2}t}g(x^u_t,0,0,0,0,0)dt\bigg)^{2k}.
\end{aligned}
\end{equation*}
Therefore, taking limit with $T\rightarrow\infty$, \eqref{ineq13} becomes
\begin{equation}\label{ineq15}
\begin{aligned}
&\mathbb{E}\bigg[\sup_{t\in\mathbb{R}_+}e^{-\beta kt}|y_t^u|^{2k}\bigg]
 \leq\bigg(2\lambda^2+3C_k^2\lambda M(k,\epsilon)^{-1}\bigg)^{\frac{k}{2(k-1)}}\frac{k-1}{k}\mathbb{E}\bigg[\sup_{t\in\mathbb{R}_+}e^{-\beta kt}|y^u_t|^{2k}\bigg]\\
&\quad+\bigg(2\lambda^2+3C_k^2\lambda M(k,\epsilon)^{-1}\bigg)^{\frac{k}{2}}\frac{1}{k}\mathbb{E}\bigg(\int_0^\infty\int_{\mathcal{E}}e^{-\beta t}|\gamma^u_{(t,e)}|^2N(de,dt)\bigg)^k\\
&\quad+\bigg(8\lambda^2k+12C_k\lambda M(k,\epsilon)^{-1}\bigg)\mathbb{E}\bigg[\sup_{t\in\mathbb{R}_+}e^{-\beta kt}|y^u_t|^{2k}\bigg]+\bigg(4k+\frac{6C_k}{\lambda}M(k,\epsilon)^{-1}\bigg)\\
&\quad \times C_\alpha \bigg[\mathbb{E}\bigg(\int_0^\infty e^{-\beta t}|u_t|^2dt\bigg)^k+\mathbb{E}\bigg(\int_0^\infty e^{-\frac{\beta}{2}t}g(x^u_t,0,0,0,0,0)dt\bigg)^{2k}\bigg],
\end{aligned}
\end{equation}
where we have used the fact that $\mathbb{E}\big[e^{-\beta kT}|y_T^u|^{2k}\big]=\mathbb{E}\big[e^{-\frac{\beta}{2}T}|y^u_T|e^{-\beta (k-\frac{1}{2})T}|y^u_T|^{2k-1}\big]\\
\leq \big(\mathbb{E}\big[e^{-\beta T}|y^u_T|^2\big]\big)^{\frac{1}{2}}\big(\mathbb{E}\big[e^{-\beta (2k-1)T}|y^u_T|^{4k-2}\big]\big)^{\frac{1}{2}} \xrightarrow{T\rightarrow\infty}0$.

Combining \eqref{ineq15} with \eqref{ineq04}, we take $\lambda_0$ small enough such that
\begin{equation*}
\begin{aligned}
&1-\bigg\{\bigg(2\lambda_0^2+3C_k^2\lambda_0 M(k,\epsilon)^{-1}\bigg)^{\frac{k}{2(k-1)}}\frac{k-1}{k}+\bigg(2\lambda_0^2+3C_k^2\lambda_0 M(k,\epsilon)^{-1}\bigg)^{\frac{k}{2}}C_k\\
&\qquad +\bigg(8\lambda_0^2k+12C_k\lambda_0 M(k,\epsilon)^{-1}\bigg)\bigg\}>0,
\end{aligned}
\end{equation*}
thus, we get the desired result as follows
\begin{equation*}
\begin{aligned}
&\mathbb{E}\bigg[\sup_{t\in\mathbb{R}_+}e^{-\beta_3 kt}|y_t^u|^{2k}+\bigg(\int_0^\infty e^{-\beta_4 t}|z^u_t|^2dt\bigg)^k+\bigg(\int_0^\infty e^{-\beta_4 t}|\tilde{z}^u_t|^2dt\bigg)^k\\
&\qquad +\bigg(\int_0^\infty\int_{\mathcal{E}}e^{-\beta_4 t}|\gamma^u_{(t,e)}|^2N(de,dt)\bigg)^k\bigg]\\
&\leq C_{k,\epsilon}\bigg[\mathbb{E}\int_0^\infty e^{-\beta_1 kt}|u_t|^{2k}dt+\mathbb{E}\bigg(\int_0^\infty e^{-\beta_0t}g(x^u_t,0,0,0,0,0)dt\bigg)^{2k}\bigg].
\end{aligned}
\end{equation*}
where $C_{k,\epsilon}=C_{k,\lambda_0,\epsilon,\alpha}$ depends on some $k,\epsilon$ and sufficiently small $\lambda_0(k),\alpha(k)$. Moreover, $\beta_1\leq\min\{\beta_3,\beta_4\}$ and $\beta_0\leq\min\{\beta_3,\beta_4\}$. The proof is complete.

\vspace{2mm}

\emph{Proof of Lemma \ref{variational ineq}}. We can deduce the variational inequality based on the fact that
\begin{equation}\label{variation inequality 00}
\begin{aligned}
&\epsilon^{-1}[J(u^\epsilon)-J(\bar{u})]=\epsilon^{-1}\bigg\{\mathbb{E}\big[\phi(y_0^\epsilon)-\phi(\bar{y}_0)\big]\\
&\ +\mathbb{E}\int_0^\infty e^{-\beta t}\big[F(x^\epsilon,y^\epsilon,z^\epsilon,\tilde{z}^\epsilon,\gamma^\epsilon,\mathcal{Z}^\epsilon,u^\epsilon)
-F(\bar{x},\bar{y},\bar{z},\bar{\tilde{z}},\bar{\gamma},\bar{\mathcal{Z}},\bar{u})\big]dt\bigg\}:=\mathbf{D}_1+\mathbf{D}_2\geq0,
\end{aligned}
\end{equation}
where we set
\begin{equation}\label{F}
F(x,y,z,\tilde{z},\gamma,\mathcal{Z},u):=\mathcal{Z}f(x,y,z,\tilde{z},\gamma,u).
\end{equation}
We first deal with the first term. Note that
\begin{equation*}
\begin{aligned}
\mathbf{D}_1&=\mathbb{E}\bigg[\bigg(\int_0^1\phi_yd\theta-\bar{\phi}_y\bigg)\tilde{y}^\epsilon_0\bigg]+\mathbb{E}\bigg[\bigg(\int_0^1\phi_yd\theta-\bar{\phi}_y\bigg)y_{(1,0)}\bigg]
 +\mathbb{E}\big[\bar{\phi}_y\tilde{y}^\epsilon_0\big]+\mathbb{E}\big[\bar{\phi}_yy_{(1,0)}\big],
\end{aligned}
\end{equation*}
where
\begin{equation*}
\begin{aligned}
&\mathbb{E}\bigg[\bigg(\int_0^1\phi_yd\theta-\bar{\phi}_y\bigg)\tilde{y}^\epsilon_0\bigg]=\mathbb{E}\bigg[\bigg(\int_0^1\phi_yd\theta-\bar{\phi}_y\bigg)e^{-\frac{\beta}{2}0}\tilde{y}^\epsilon_0\bigg]\\
&\leq \mathbb{E}\bigg[\sup_{t\in\mathcal{R}_+}e^{-\frac{\beta}{2}t}\tilde{y}^\epsilon_t\bigg(\int_0^1\phi_yd\theta-\bar{\phi}_y\bigg)\bigg]\\
&\leq \bigg(\mathbb{E}\bigg[\sup_{t\in\mathcal{R}_+}e^{-\beta t}|\tilde{y}^\epsilon_t|^2\bigg]\bigg)^{\frac{1}{2}}
\bigg(\mathbb{E}\bigg(\int_0^1\phi_yd\theta-\bar{\phi}_y\bigg)^2\bigg)^{\frac{1}{2}}\rightarrow0,\quad \text{as}\ \epsilon\rightarrow0,
\end{aligned}
\end{equation*}
and similarly, $\mathbb{E}[(\int_0^1\phi_yd\theta-\bar{\phi}_y)y_{(1,0)}]\rightarrow0,\ \text{as}\ \epsilon\rightarrow0,$ and
\begin{equation*}
\begin{aligned}
&\mathbb{E}\big[\bar{\phi}_y\tilde{y}^\epsilon_0\big]\leq C\mathbb{E}\big[(1+|\bar{y}_0|)\tilde{y}^\epsilon_0\big]=C\mathbb{E}\big[e^{-\frac{\beta}{2}0}\tilde{y}^\epsilon_0+e^{-\beta0}|\bar{y}_0|\tilde{y}^\epsilon_0\big]\\
&\leq C\mathbb{E}\bigg[\sup_{t\in\mathbb{R}_+}e^{-\frac{\beta}{2}t}\tilde{y}^\epsilon_t\bigg]
+C\mathbb{E}\bigg[\sup_{t\in\mathbb{R}_+}e^{-\frac{\beta}{2} t}|\bar{y}_t|\sup_{t\in\mathbb{R}_+}e^{-\frac{\beta}{2} t}\tilde{y}^\epsilon_t\bigg]\\
&\leq C\bigg(\mathbb{E}\bigg[\sup_{t\in\mathbb{R}_+}e^{-\beta t}|\tilde{y}^\epsilon_t|^2\bigg]\bigg)^{\frac{1}{2}}
+C\bigg(\mathbb{E}\bigg[\sup_{t\in\mathbb{R}_+}e^{-\beta t}|\bar{y}_t|^2\bigg]\bigg)^{\frac{1}{2}}
\bigg(\mathbb{E}\bigg[\sup_{t\in\mathbb{R}_+}e^{-\beta t}|\tilde{y}^\epsilon_t|^2\bigg]\bigg)^{\frac{1}{2}}\rightarrow0,\ \ \text{as}\ \epsilon\rightarrow0.
\end{aligned}
\end{equation*}
Therefore, we have
\begin{equation}\label{10}
\lim_{\epsilon\rightarrow0}\mathbf{D}_1=\mathbb{E}\big[\bar{\phi}_yy_{(1,0)}\big].
\end{equation}

Next, we deal with the second term. Note that (we omit some $t$ in many integrands)
\begin{equation*}
\begin{aligned}
\mathbf{D}_2&=\mathbb{E}\int_0^\infty e^{-\beta t}\bigg[\int_0^1F_xd\theta(\tilde{x}^\epsilon+x_1)+\int_0^1F_yd\theta(\tilde{y}^\epsilon+y_1)+\int_0^1F_zd\theta(\tilde{z}^\epsilon+z_1)\\
&\qquad +\int_0^1F_{\tilde{z}}d\theta(\tilde{\tilde{z}}^\epsilon+\tilde{z}_1)+\int_0^1F_\gamma d\theta\int_{\mathcal{E}}(\tilde{\gamma}^\epsilon
 +\gamma_1)\nu(de)+\int_0^1F_{\mathcal{Z}}d\theta(\tilde{\mathcal{Z}}^\epsilon+\mathcal{Z}_1)+\int_0^1F_ud\theta v\bigg]dt\\
&:=\mathbf{D}_{21}+\mathbf{D}_{22}+\mathbf{D}_{23}+\mathbf{D}_{24}+\mathbf{D}_{25}+\mathbf{D}_{26}+\mathbf{D}_{27},
\end{aligned}
\end{equation*}
and in the following we will deal with the above terms one by one. First, we have
\begin{equation*}
\begin{aligned}
\mathbf{D}_{21}&=\mathbb{E}\int_0^\infty e^{-\beta t}\bigg(\int_0^1F_xd\theta-\bar{F}_x\bigg)\tilde{x}^\epsilon dt+\mathbb{E}\int_0^\infty e^{-\beta t}\bar{F}_x\tilde{x}^\epsilon dt\\
&\quad +\mathbb{E}\int_0^\infty e^{-\beta t}\bigg(\int_0^1F_xd\theta-\bar{F}_x\bigg)x_1dt+\mathbb{E}\int_0^\infty e^{-\beta t}\bar{F}_xx_1dt,
\end{aligned}
\end{equation*}
where
\begin{equation*}
\begin{aligned}
&\mathbb{E}\int_0^\infty e^{-\beta t}\bigg(\int_0^1F_xd\theta-\bar{F}_x\bigg)\tilde{x}^\epsilon dt\\
&\leq \bigg(\mathbb{E}\int_0^\infty e^{-\beta t}|\tilde{x}^\epsilon|^2 dt\bigg)^{\frac{1}{2}}\bigg(\mathbb{E}\int_0^\infty e^{-\beta t}\bigg(\int_0^1F_xd\theta-\bar{F}_x\bigg)^2dt\bigg)^{\frac{1}{2}}\\
&\leq \bigg(\sup_{t\in\mathbb{R}_+}\mathbb{E}\big[e^{-\frac{\beta}{2} t}|\tilde{x}^\epsilon|^2\big]\int_0^\infty e^{-\frac{\beta}{2} t} dt\bigg)^{\frac{1}{2}}
 \bigg(\mathbb{E}\int_0^\infty e^{-\beta t}\bigg(\int_0^1F_xd\theta-\bar{F}_x\bigg)^2dt\bigg)^{\frac{1}{2}}\rightarrow0,\quad \text{as}\ \epsilon\rightarrow0.
\end{aligned}
\end{equation*}
Similarly, $\mathbb{E}\int_0^\infty e^{-\beta t}\big(\int_0^1F_xd\theta-\bar{F}_x\big)x_1 dt\rightarrow0,\ \text{as}\ \epsilon\rightarrow0$,
and by \eqref{F}, {\bf (H1)},
\begin{equation*}
\begin{aligned}
&\mathbb{E}\int_0^\infty e^{-\beta t}\bar{F}_x\tilde{x}^\epsilon dt=\mathbb{E}\int_0^\infty e^{-\beta t}\bar{\mathcal{Z}}\bar{f}_x\tilde{x}^\epsilon dt\\
&\leq C\mathbb{E}\int_0^\infty e^{-\beta t}\bar{\mathcal{Z}}\big(1+|\bar{x}|+|\bar{y}|+|\bar{z}|+|\bar{\tilde{z}}|+||\bar{\gamma}||+|\bar{u}|\big)\tilde{x}^\epsilon dt.
\end{aligned}
\end{equation*}
We can divide the above into several parts to deal with as follows:
\begin{equation*}
\begin{aligned}
&\mathbb{E}\int_0^\infty e^{-\beta t}\bar{\mathcal{Z}}\tilde{x}^\epsilon dt=\int_0^\infty e^{-\beta t}\mathbb{E}\big[\bar{\mathcal{Z}}\tilde{x}^\epsilon\big] dt
 \leq \int_0^\infty e^{-\beta t}\big(\mathbb{E}\bar{\mathcal{Z}}^2\big)^{\frac{1}{2}}\big(\mathbb{E}|\tilde{x}^\epsilon|^2\big)^{\frac{1}{2}}dt\\
&\leq \bigg(\mathbb{E}\bigg[\sup_{t\in\mathbb{R}_+}\bar{\mathcal{Z}}^2\bigg]\bigg)^{\frac{1}{2}}\int_0^\infty e^{-\frac{\beta}{2} t}dt
 \bigg(\sup_{t\in\mathbb{R}_+}\mathbb{E}\big[e^{-\beta t}|\tilde{x}^\epsilon|^2\big]\bigg)^{\frac{1}{2}}\rightarrow0,\quad \text{as}\ \epsilon\rightarrow0,
\end{aligned}
\end{equation*}
\begin{equation*}
\begin{aligned}
&\mathbb{E}\int_0^\infty e^{-\beta t}\bar{\mathcal{Z}}|\bar{x}|\tilde{x}^\epsilon dt
 \leq \int_0^\infty e^{-\beta t}\big(\mathbb{E}\bar{\mathcal{Z}}^4\big)^{\frac{1}{4}}\big(\mathbb{E}\bar{x}^4\big)^{\frac{1}{4}}\big(\mathbb{E}|\tilde{x}^\epsilon|^2\big)^{\frac{1}{2}} dt\\
&\leq \bigg(\mathbb{E}\bigg[\sup_{t\in\mathbb{R}_+}\bar{\mathcal{Z}}^4\bigg]\bigg)^{\frac{1}{4}}
 \int_0^\infty e^{-\frac{\beta}{4} t} dt\bigg(\sup_{t\in\mathbb{R}_+}\mathbb{E}\big[e^{-\beta t}\bar{x}^4\big]\bigg)^{\frac{1}{4}}
 \bigg(\sup_{t\in\mathbb{R}_+}\mathbb{E}\big[e^{-\beta t}|\tilde{x}^\epsilon|^2\big]\bigg)^{\frac{1}{2}}\rightarrow0,\quad \text{as}\ \epsilon\rightarrow0,
\end{aligned}
\end{equation*}
\begin{equation*}
\begin{aligned}
&\mathbb{E}\int_0^\infty e^{-\beta t}\bar{\mathcal{Z}}|\bar{y}|\tilde{x}^\epsilon dt
 \leq \mathbb{E}\bigg[\bigg(\sup_{t\in\mathcal{R}_+}e^{-\frac{\beta}{2} t}\bar{y}\bigg)\bigg(\sup_{t\in\mathcal{R}_+}\bar{\mathcal{Z}}\bigg)\bigg(\int_0^\infty e^{-\frac{\beta}{2} t}\tilde{x}^\epsilon dt\bigg)\bigg]\\
&\leq \bigg(\mathbb{E}\bigg[\sup_{t\in\mathcal{R}_+}e^{-\beta t}\bar{y}^2\bigg]\bigg)^\frac{1}{2}
 \bigg(\mathbb{E}\bigg[\sup_{t\in\mathcal{R}_+}\bar{\mathcal{Z}}^4\bigg]\bigg)^{\frac{1}{4}}
 \bigg(\mathbb{E}\bigg(\int_0^\infty e^{-\frac{\beta}{2} t}\tilde{x}^\epsilon dt\bigg)^4\bigg)^{\frac{1}{4}}\\
&\leq C\bigg(\mathbb{E}\int_0^\infty e^{-\beta t}|\tilde{x}^\epsilon|^4 dt\bigg)^{\frac{1}{4}}
 \leq C\bigg(\sup_{t\in\mathbb{R}_+}\mathbb{E}\big[e^{-\frac{\beta}{2} t}|\tilde{x}^\epsilon|^4\big] \bigg(\int_0^\infty e^{-\frac{\beta}{2} t}dt\bigg)\bigg)^{\frac{1}{4}}\rightarrow0,\quad \text{as}\ \epsilon\rightarrow0,
\end{aligned}
\end{equation*}
\begin{equation*}
\begin{aligned}
&\mathbb{E}\int_0^\infty e^{-\beta t}\bar{\mathcal{Z}}|\bar{z}|\tilde{x}^\epsilon dt
 \leq \bigg(\mathbb{E}\int_0^\infty e^{-\beta t}\bar{z}^2dt\bigg)^{\frac{1}{2}}
 \bigg(\mathbb{E}\bigg[\sup_{t\in\mathbb{R}_+}\bar{\mathcal{Z}}^4\bigg]\int_0^\infty e^{-\beta t}dt\bigg)^{\frac{1}{4}}\bigg(\mathbb{E}\int_0^\infty e^{-\beta t}|\tilde{x}^\epsilon|^4dt\bigg)^{\frac{1}{4}}\\
&\leq C\bigg(\sup_{t\in\mathbb{R}_+}\mathbb{E}\big[e^{-\frac{\beta}{2} t}|\tilde{x}^\epsilon|^4\big]\int_0^\infty e^{-\frac{\beta}{2} t}dt\bigg)^{\frac{1}{4}}\rightarrow0,\quad \text{as}\ \epsilon\rightarrow0.
\end{aligned}
\end{equation*}
Similarly, $\mathbb{E}\int_0^\infty e^{-\beta t}\bar{\mathcal{Z}}|\bar{\tilde{z}}|\tilde{x}^\epsilon dt\rightarrow0,\ \text{as}\ \epsilon\rightarrow0$,
and
\begin{equation*}
\begin{aligned}
&\mathbb{E}\int_0^\infty e^{-\beta t}\bar{\mathcal{Z}}||\bar{\gamma}||\tilde{x}^\epsilon dt\\
&\leq \bigg(\mathbb{E}\int_0^\infty \int_{\mathcal{E}}e^{-\beta t}|\bar{\gamma}|^2N(de,dt)\bigg)^{\frac{1}{2}}\bigg(\mathbb{E}\bigg[\sup_{t\in\mathbb{R}_+}\bar{\mathcal{Z}}^4\bigg]
 \int_0^\infty e^{-\beta t}dt\bigg)^{\frac{1}{4}}\bigg(\mathbb{E}\int_0^\infty e^{-\beta t}|\tilde{x}^\epsilon|^4dt\bigg)^{\frac{1}{4}}\\
&\leq C\bigg(\sup_{t\in\mathbb{R}_+}\mathbb{E}\big[e^{-\frac{\beta}{2} t}|\tilde{x}^\epsilon|^4\big]\int_0^\infty e^{-\frac{\beta}{2} t}dt\bigg)^{\frac{1}{4}}\rightarrow0,\quad \text{as}\ \epsilon\rightarrow0,
\end{aligned}
\end{equation*}
\begin{equation*}
\begin{aligned}
&\mathbb{E}\int_0^\infty e^{-\beta t}\bar{\mathcal{Z}}|\bar{u}|\tilde{x}^\epsilon dt\leq \bigg(\mathbb{E}\int_0^\infty e^{-\beta t}\bar{\mathcal{Z}}^4dt\bigg)^{\frac{1}{4}}
 \bigg(\mathbb{E}\int_0^\infty e^{-\beta t}|\bar{u}|^4dt\bigg)^{\frac{1}{4}}\bigg(\mathbb{E}\int_0^\infty e^{-\beta t}|\tilde{x}^\epsilon|^2dt\bigg)^{\frac{1}{2}}\\
&\leq C\bigg(\mathbb{E}\bigg[\sup_{t\in\mathbb{R}_+}\bar{\mathcal{Z}}^4\bigg]\int_0^\infty e^{-\beta t}dt\bigg)^{\frac{1}{4}}
 \bigg(\sup_{t\in\mathbb{R}_+}\mathbb{E}\big[e^{-\frac{\beta}{2} t}|\tilde{x}^\epsilon|^2\big]\int_0^\infty e^{-\frac{\beta}{2} t}dt\bigg)^{\frac{1}{2}}\rightarrow0,\quad \text{as}\ \epsilon\rightarrow0.
\end{aligned}
\end{equation*}
Therefore, we have
\begin{equation}\label{I}
\lim_{\epsilon\rightarrow0}\mathbf{D}_{21}=\mathbb{E}\int_0^\infty e^{-\beta t}\bar{\mathcal{Z}}\bar{f}_xx_1dt.
\end{equation}
Similarly, we have
$\lim_{\epsilon\rightarrow0}\mathbf{D}_{22}=\mathbb{E}\int_0^\infty e^{-\beta t}\bar{\mathcal{Z}}\bar{f}_yy_1dt$,
the only difference is the norm of $y$ where the sup is in the expectation $\mathbb{E}$.

Note that the third term can be managed as follows:
\begin{equation*}
\begin{aligned}
\mathbf{D}_{23}&=\mathbb{E}\int_0^\infty e^{-\beta t}\bigg(\int_0^1F_zd\theta-\bar{F}_z\bigg)\tilde{z}^\epsilon dt+\mathbb{E}\int_0^\infty e^{-\beta t}\bar{F}_z\tilde{z}^\epsilon dt\\
&\quad+\mathbb{E}\int_0^\infty e^{-\beta t}\bigg(\int_0^1F_zd\theta-\bar{F}_z\bigg)z_1dt+\mathbb{E}\int_0^\infty e^{-\beta t}\bar{F}_zz_1dt,
\end{aligned}
\end{equation*}
where
\begin{equation*}
\begin{aligned}
&\mathbb{E}\int_0^\infty e^{-\beta t}\bigg(\int_0^1F_zd\theta-\bar{F}_z\bigg)\tilde{z}^\epsilon dt\\
&\leq \bigg(\mathbb{E}\int_0^\infty e^{-\beta t}|\tilde{z}^\epsilon|^2 dt\bigg)^{\frac{1}{2}}
 \bigg(\mathbb{E}\int_0^\infty e^{-\beta t}\bigg(\int_0^1F_zd\theta-\bar{F}_z\bigg)^2dt\bigg)^{\frac{1}{2}}\rightarrow0,\quad \text{as}\ \epsilon\rightarrow0.
\end{aligned}
\end{equation*}
Similarly, $\mathbb{E}\int_0^\infty e^{-\beta t}\big(\int_0^1F_zd\theta-\bar{F}_z\big)z_1 dt\rightarrow0,\ \text{as}\ \epsilon\rightarrow0$,
and
\begin{equation*}
\begin{aligned}
&\mathbb{E}\int_0^\infty e^{-\beta t}\bar{F}_z\tilde{z}^\epsilon dt
\leq \mathbb{E}\int_0^\infty e^{-\beta t}\bar{\mathcal{Z}}\big(1+|\bar{x}|+|\bar{y}|+|\bar{z}|+|\bar{\tilde{z}}|+||\bar{\gamma}||+|\bar{u}|\big)\tilde{z}^\epsilon dt.
\end{aligned}
\end{equation*}
Then we deal with the above term by term as follows:
\begin{equation*}
\begin{aligned}
&\mathbb{E}\int_0^\infty e^{-\beta t}\bar{\mathcal{Z}}\tilde{z}^\epsilon dt\leq \bigg(\mathbb{E}\int_0^\infty e^{-\beta t}\bar{\mathcal{Z}}^2 dt\bigg)^{\frac{1}{2}}
 \bigg(\mathbb{E}\int_0^\infty e^{-\beta t}|\tilde{z}^\epsilon|^2 dt\bigg)^{\frac{1}{2}}\\
&\leq \bigg(\mathbb{E}\bigg[\sup_{t\in\mathbb{R}_+}\bar{\mathcal{Z}}^2\bigg]\int_0^\infty e^{-\beta t} dt\bigg)^{\frac{1}{2}}
 \bigg(\mathbb{E}\int_0^\infty e^{-\beta t}|\tilde{z}^\epsilon|^2 dt\bigg)^{\frac{1}{2}}\rightarrow0,\quad \text{as}\ \epsilon\rightarrow0,
\end{aligned}
\end{equation*}
\begin{equation*}
\begin{aligned}
&\mathbb{E}\int_0^\infty e^{-\beta t}\bar{\mathcal{Z}}|\bar{x}|\tilde{z}^\epsilon dt\leq \bigg(\mathbb{E}\int_0^\infty e^{-\beta t}\bar{\mathcal{Z}}^4 dt\bigg)^{\frac{1}{4}}\bigg(\mathbb{E}\int_0^\infty e^{-\beta t}|\bar{x}|^4 dt\bigg)^{\frac{1}{4}}\bigg(\mathbb{E}\int_0^\infty e^{-\beta t}|\tilde{z}^\epsilon|^2 dt\bigg)^{\frac{1}{2}}\\
&\leq \bigg(\mathbb{E}\bigg[\sup_{t\in\mathbb{R}_+}\bar{\mathcal{Z}}^4\bigg]\int_0^\infty e^{-\beta t} dt\bigg)^{\frac{1}{4}}
 \bigg(\sup_{t\in\mathbb{R}_+}\mathbb{E}\big[e^{-\frac{\beta}{2} t}|\bar{x}|^4\big]\int_0^\infty e^{-\frac{\beta}{2} t} dt\bigg)^{\frac{1}{4}}
 \bigg(\mathbb{E}\int_0^\infty e^{-\beta t}|\tilde{z}^\epsilon|^2 dt\bigg)^{\frac{1}{2}}\\
&\rightarrow0,\quad \text{as}\ \epsilon\rightarrow0.
\end{aligned}
\end{equation*}
Similarly, $\mathbb{E}\int_0^\infty e^{-\beta t}\bar{\mathcal{Z}}|\bar{y}|\tilde{z}^\epsilon dt\rightarrow0,\ \text{as}\ \epsilon\rightarrow0$, and
\begin{equation*}
\begin{aligned}
&\mathbb{E}\int_0^\infty e^{-\beta t}\bar{\mathcal{Z}}|\bar{z}|\tilde{z}^\epsilon dt\leq \mathbb{E}\bigg[\sup_{t\in\mathbb{R}_+}\bar{\mathcal{Z}}\int_0^\infty e^{-\beta t}|\bar{z}|\tilde{z}^\epsilon dt\bigg]\\
&\leq \mathbb{E}\bigg[\sup_{t\in\mathbb{R}_+}\bar{\mathcal{Z}}\bigg(\int_0^\infty e^{-\beta t}|\bar{z}|^2 dt\bigg)^{\frac{1}{2}}\bigg(\int_0^\infty e^{-\beta t}|\tilde{z}^\epsilon|^2 dt\bigg)^{\frac{1}{2}}\bigg]\\
&\leq \bigg(\mathbb{E}\bigg[\sup_{t\in\mathbb{R}_+}\bar{\mathcal{Z}}^2\bigg]\bigg)^{\frac{1}{2}}
 \bigg(\mathbb{E}\bigg(\int_0^\infty e^{-\beta t}|\bar{z}|^2 dt\bigg)^2\bigg)^{\frac{1}{4}}\bigg(\mathbb{E}\bigg(\int_0^\infty e^{-\beta t}|\tilde{z}^\epsilon|^2 dt\bigg)^2\bigg)^{\frac{1}{4}}
 \rightarrow0,\quad \text{as}\ \epsilon\rightarrow0,
\end{aligned}
\end{equation*}
\begin{equation*}
\begin{aligned}
&\mathbb{E}\int_0^\infty e^{-\beta t}\bar{\mathcal{Z}}||\bar{\gamma}||\tilde{z}^\epsilon dt\leq \mathbb{E}\bigg[\sup_{t\in\mathbb{R}_+}\bar{\mathcal{Z}}\int_0^\infty e^{-\beta t}||\bar{\gamma}||\tilde{z}^\epsilon dt\bigg]\\
&\leq \mathbb{E}\bigg[\sup_{t\in\mathbb{R}_+}\bar{\mathcal{Z}}\bigg(\int_0^\infty e^{-\beta t}||\bar{\gamma}||^2 dt\bigg)^{\frac{1}{2}}\bigg(\int_0^\infty e^{-\beta t}|\tilde{z}^\epsilon|^2 dt\bigg)^{\frac{1}{2}}\bigg]\\
&\leq \bigg(\mathbb{E}\bigg[\sup_{t\in\mathbb{R}_+}\bar{\mathcal{Z}}^2\bigg]\bigg)^{\frac{1}{2}}
 \bigg(\mathbb{E}\bigg(\int_0^\infty\int_{\mathcal{E}} e^{-\beta t}|\bar{\gamma}|^2 N(de,dt)\bigg)^2\bigg)^{\frac{1}{4}}\bigg(\mathbb{E}\bigg(\int_0^\infty e^{-\beta t}|\tilde{z}^\epsilon|^2 dt\bigg)^2\bigg)^{\frac{1}{4}}\\
&\rightarrow0,\quad \text{as}\ \epsilon\rightarrow0,
\end{aligned}
\end{equation*}
\begin{equation*}
\begin{aligned}
&\mathbb{E}\int_0^\infty e^{-\beta t}\bar{\mathcal{Z}}|\bar{u}|\tilde{z}^\epsilon dt\leq \bigg(\mathbb{E}\int_0^\infty e^{-\beta t}\bar{\mathcal{Z}}^4dt\bigg)^{\frac{1}{4}}
 \bigg(\mathbb{E}\int_0^\infty e^{-\beta t}|\bar{u}|^4dt\bigg)^{\frac{1}{4}}\bigg(\mathbb{E}\int_0^\infty e^{-\beta t}|\tilde{z}^\epsilon|^2 dt\bigg)^{\frac{1}{2}}\\
&\rightarrow0,\quad \text{as}\ \epsilon\rightarrow0.
\end{aligned}
\end{equation*}
Therefore, we have
\begin{equation}\label{III}
\lim_{\epsilon\rightarrow0}\mathbf{D}_{23}=\mathbb{E}\int_0^\infty e^{-\beta t}\bar{\mathcal{Z}}\bar{f}_zz_1dt.
\end{equation}
Similarly, $\lim_{\epsilon\rightarrow0}\mathbf{D}_{24}=\mathbb{E}\int_0^\infty e^{-\beta t}\bar{\mathcal{Z}}\bar{f}_{\tilde{z}}\tilde{z}_1dt$.
Note that the fifth term can be managed as follows:
\begin{equation*}
\begin{aligned}
\mathbf{D}_{25}&=\mathbb{E}\int_0^\infty e^{-\beta t}\bigg(\int_0^1F_\gamma d\theta-\bar{F}_\gamma\bigg)\int_{\mathcal{E}}\tilde{\gamma}^\epsilon\nu(de) dt
 +\mathbb{E}\int_0^\infty e^{-\beta t}\bar{F}_\gamma\int_{\mathcal{E}}\tilde{\gamma}^\epsilon\nu(de) dt\\
&\quad +\mathbb{E}\int_0^\infty e^{-\beta t}\bigg(\int_0^1F_\gamma d\theta-\bar{F}_\gamma\bigg)\int_{\mathcal{E}}\gamma_1\nu(de)dt
 +\mathbb{E}\int_0^\infty e^{-\beta t}\bar{F}_{\gamma}\int_{\mathcal{E}}\gamma_1\nu(de)dt,
\end{aligned}
\end{equation*}
where
\begin{equation*}
\begin{aligned}
&\mathbb{E}\int_0^\infty e^{-\beta t}\bigg(\int_0^1F_\gamma d\theta-\bar{F}_\gamma\bigg)\int_{\mathcal{E}}\tilde{\gamma}^\epsilon\nu(de) dt\\
&\leq C\bigg(\mathbb{E}\int_0^\infty\int_{\mathcal{E}} e^{-\beta t}|\tilde{\gamma}^\epsilon|^2N(de,dt)\bigg)^{\frac{1}{2}}
 \bigg(\mathbb{E}\int_0^\infty e^{-\beta t}\bigg(\int_0^1F_\gamma d\theta-\bar{F}_\gamma\bigg)^2 dt\bigg)^{\frac{1}{2}}\rightarrow0,\quad \text{as}\ \epsilon\rightarrow0.
\end{aligned}
\end{equation*}
Similarly, $\mathbb{E}\int_0^\infty e^{-\beta t}\big(\int_0^1F_\gamma d\theta-\bar{F}_\gamma\big)\int_{\mathcal{E}}\gamma_1\nu(de) dt\rightarrow0,\ \text{as}\ \epsilon\rightarrow0$,
and
\begin{equation*}
\begin{aligned}
&\mathbb{E}\int_0^\infty e^{-\beta t}\bar{F}_\gamma\int_{\mathcal{E}}\tilde{\gamma}^\epsilon\nu(de) dt\leq \mathbb{E}\int_0^\infty e^{-\beta t}\bar{\mathcal{Z}}\big(1+|\bar{x}|+|\bar{y}|+|\bar{z}|+|\bar{\tilde{z}}|+||\bar{\gamma}||+|\bar{u}|\big)\int_{\mathcal{E}}\tilde{\gamma}^\epsilon\nu(de) dt.
\end{aligned}
\end{equation*}
We can deal with the above in the following:
\begin{equation*}
\begin{aligned}
&\mathbb{E}\int_0^\infty e^{-\beta t}\bar{\mathcal{Z}}\int_{\mathcal{E}}\tilde{\gamma}^\epsilon\nu(de) dt\\
&\leq C\bigg(\mathbb{E}\int_0^\infty e^{-\beta t}\bar{\mathcal{Z}}^2dt\bigg)^{\frac{1}{2}}\bigg(\mathbb{E}\int_0^\infty e^{-\beta t}\int_{\mathcal{E}}|\tilde{\gamma}^\epsilon|^2\nu(de) dt\bigg)^{\frac{1}{2}}\\
&\leq C\bigg(\mathbb{E}\bigg[\sup_{t\in\mathbb{R}_+}\bar{\mathcal{Z}}^2\bigg]\int_0^\infty e^{-\beta t}dt\bigg)^{\frac{1}{2}}
 \bigg(\mathbb{E}\int_0^\infty\int_{\mathcal{E}} e^{-\beta t}|\tilde{\gamma}^\epsilon|^2N(de,dt)\bigg)^{\frac{1}{2}}\rightarrow0,\quad \text{as}\ \epsilon\rightarrow0,
\end{aligned}
\end{equation*}
\begin{equation*}
\begin{aligned}
&\mathbb{E}\int_0^\infty e^{-\beta t}\bar{\mathcal{Z}}|\bar{x}|\int_{\mathcal{E}}\tilde{\gamma}^\epsilon \nu(de)dt\\
&\leq C\bigg(\mathbb{E}\int_0^\infty e^{-\beta t}\bar{\mathcal{Z}}^4 dt\bigg)^{\frac{1}{4}}\bigg(\mathbb{E}\int_0^\infty e^{-\beta t}|\bar{x}|^4 dt\bigg)^{\frac{1}{4}}
 \bigg(\mathbb{E}\int_0^\infty e^{-\beta t}\int_{\mathcal{E}}|\tilde{\gamma}^\epsilon|^2 \nu(de)dt\bigg)^{\frac{1}{2}}\\
&\leq C\bigg(\mathbb{E}\bigg[\sup_{t\in\mathbb{R}_+}\bar{\mathcal{Z}}^4\bigg]\int_0^\infty e^{-\beta t} dt\bigg)^{\frac{1}{4}}
 \bigg(\sup_{t\in\mathbb{R}_+}\mathbb{E}\big[e^{-\frac{\beta}{2} t}|\bar{x}|^4\big]\int_0^\infty e^{-\frac{\beta}{2} t} dt\bigg)^{\frac{1}{4}}\\
&\quad \times \bigg(\mathbb{E}\int_0^\infty\int_{\mathcal{E}} e^{-\beta t}|\tilde{\gamma}^\epsilon|^2 N(de,dt)\bigg)^{\frac{1}{2}}\rightarrow0,\quad \text{as}\ \epsilon\rightarrow0.
\end{aligned}
\end{equation*}
Similarly, $\mathbb{E}\int_0^\infty e^{-\beta t}\bar{\mathcal{Z}}|\bar{y}|\int_{\mathcal{E}}\tilde{\gamma}^\epsilon\nu(de) dt\rightarrow0,\ \text{as}\ \epsilon\rightarrow0$,
\begin{equation*}
\begin{aligned}
&\mathbb{E}\int_0^\infty e^{-\beta t}\bar{\mathcal{Z}}|\bar{z}|\int_{\mathcal{E}}\tilde{\gamma}^\epsilon\nu(de) dt
\leq \mathbb{E}\bigg[\sup_{t\in\mathbb{R}_+}\bar{\mathcal{Z}}\int_0^\infty e^{-\beta t}|\bar{z}|\int_{\mathcal{E}}\tilde{\gamma}^\epsilon\nu(de) dt\bigg]\\
&\leq C\mathbb{E}\bigg[\sup_{t\in\mathbb{R}_+}\bar{\mathcal{Z}}\bigg(\int_0^\infty e^{-\beta t}|\bar{z}|^2 dt\bigg)^{\frac{1}{2}}
 \bigg(\int_0^\infty e^{-\beta t}\int_{\mathcal{E}}|\tilde{\gamma}^\epsilon|^2 \nu(de)dt\bigg)^{\frac{1}{2}}\bigg]\\
&\leq C\bigg(\mathbb{E}\bigg[\sup_{t\in\mathbb{R}_+}\bar{\mathcal{Z}}^2\bigg]\bigg)^{\frac{1}{2}}
 \bigg(\mathbb{E}\bigg(\int_0^\infty e^{-\beta t}|\bar{z}|^2 dt\bigg)^2\bigg)^{\frac{1}{4}}\bigg(\mathbb{E}\bigg(\int_0^\infty\int_{\mathcal{E}}e^{-\beta t}|\tilde{\gamma}^\epsilon|^2N(de,dt)\bigg)^2\bigg)^{\frac{1}{4}}\\
&\rightarrow0,\quad \text{as}\ \epsilon\rightarrow0,
\end{aligned}
\end{equation*}
\begin{equation*}
\begin{aligned}
&\mathbb{E}\int_0^\infty e^{-\beta t}\bar{\mathcal{Z}}||\bar{\gamma}||\int_{\mathcal{E}}\tilde{\gamma}^\epsilon\nu(de) dt
\leq \mathbb{E}\bigg[\sup_{t\in\mathbb{R}_+}\bar{\mathcal{Z}}\int_0^\infty e^{-\beta t}||\bar{\gamma}||\int_{\mathcal{E}}\tilde{\gamma}^\epsilon\nu(de) dt\bigg]\\
&\leq C\mathbb{E}\bigg[\sup_{t\in\mathbb{R}_+}\bar{\mathcal{Z}}\bigg(\int_0^\infty e^{-\beta t}||\bar{\gamma}||^2 dt\bigg)^{\frac{1}{2}}
 \bigg(\int_0^\infty e^{-\beta t}\int_{\mathcal{E}}|\tilde{\gamma}^\epsilon|^2 \nu(de)dt\bigg)^{\frac{1}{2}}\bigg]\\
&\leq C\bigg(\mathbb{E}\bigg[\sup_{t\in\mathbb{R}_+}\bar{\mathcal{Z}}^2\bigg]\bigg)^{\frac{1}{2}}\bigg(\mathbb{E}\bigg(\int_0^\infty\int_{\mathcal{E}} e^{-\beta t}|\bar{\gamma}|^2 N(de,dt)\bigg)^2\bigg)^{\frac{1}{4}}\\
&\quad \times \bigg(\mathbb{E}\bigg(\int_0^\infty\int_{\mathcal{E}} e^{-\beta t}|\tilde{\gamma}^\epsilon|^2N(de,dt)\bigg)^2\bigg)^{\frac{1}{4}}\rightarrow0,\quad \text{as}\ \epsilon\rightarrow0,
\end{aligned}
\end{equation*}
\begin{equation*}
\begin{aligned}
&\mathbb{E}\int_0^\infty e^{-\beta t}\bar{\mathcal{Z}}|\bar{u}|\int_{\mathcal{E}}\tilde{\gamma}^\epsilon\nu(de) dt\\
&\leq C\bigg(\mathbb{E}\int_0^\infty e^{-\beta t}\bar{\mathcal{Z}}^4dt\bigg)^{\frac{1}{4}}\bigg(\mathbb{E}\int_0^\infty e^{-\beta t}|\bar{u}|^4dt\bigg)^{\frac{1}{4}}\bigg(\mathbb{E}\int_0^\infty\int_{\mathcal{E}} e^{-\beta t}|\tilde{\gamma}^\epsilon|^2 N(de,dt)\bigg)^{\frac{1}{2}}\\
&\rightarrow0,\quad \text{as}\ \epsilon\rightarrow0.
\end{aligned}
\end{equation*}
Therefore, we have
\begin{equation}\label{V}
\lim_{\epsilon\rightarrow0}\mathbf{D}_{25}=\mathbb{E}\int_0^\infty e^{-\beta t}\bar{\mathcal{Z}}\bar{f}_{\gamma}\int_{\mathcal{E}}\gamma_1\nu(de)dt.
\end{equation}
The last two term can be managed similarly, so we give the followings directly,
\begin{equation}\label{VII}
\begin{aligned}
&\lim_{\epsilon\rightarrow0}\big(\mathbf{D}_{26}+\mathbf{D}_{27}\big)=\mathbb{E}\int_0^\infty e^{-\beta t}\bar{f}\mathcal{Z}_1dt+\mathbb{E}\int_0^\infty e^{-\beta t}\bar{F}_uv dt.\\
\end{aligned}
\end{equation}
Therefore, from \eqref{10}--\eqref{VII}, our variational inequality can be derived as follows:
\begin{equation*}
\begin{aligned}
&\lim_{\epsilon\rightarrow0}\epsilon^{-1}\big[J(u^\epsilon)-J(\bar{u})\big]=\mathbb{E}\big[\phi_y(\bar{y}_0)y_{(1,0)}\big]
 +\mathbb{E}\int_0^\infty e^{-\beta t}\bigg(\bar{\mathcal{Z}}_t\bar{f}_xx_{(1,t)}+\bar{\mathcal{Z}}_t\bar{f}_yy_{(1,t)}\\
&\qquad +\bar{\mathcal{Z}}_t\bar{f}_zz_{(1,t)}+\bar{\mathcal{Z}}_t\bar{f}_{\tilde{z}}\tilde{z}_{(1,t)}
 +\bar{\mathcal{Z}}_t\bar{f}_{\gamma}\int_{\mathcal{E}}\gamma_{(1,t,e)}\nu(de)+\bar{f}\mathcal{Z}_{(1,t)}+\bar{\mathcal{Z}}_t\bar{f}_uv_t\bigg)dt\geq0,
\end{aligned}
\end{equation*}
The proof is complete.

\end{document}